\documentclass[a4paper,11pt,reqno]{amsart}

\usepackage[T1]{fontenc}
\usepackage{lmodern}
\usepackage{amsmath,amssymb,amsthm,mathtools,hyperref,geometry,cleveref,mathrsfs,xcolor,microtype}

\allowdisplaybreaks
\definecolor{DarkBlue}{rgb}{0.1,0.1,0.55}
\definecolor{DarkRed}{rgb}{0.55,0.1,0.1}

\hypersetup{colorlinks=true,linkcolor=DarkBlue,citecolor=DarkRed,urlcolor=DarkBlue}

\numberwithin{equation}{section}
\theoremstyle{plain}
\newtheorem{theorem}{Theorem}[section]
\newtheorem{conjecture}{Conjecture}[section]
\newtheorem{proposition}[theorem]{Proposition}
\newtheorem{corollary}[theorem]{Corollary}
\newtheorem{lemma}[theorem]{Lemma}
\newtheorem{remark}[theorem]{Remark}

\newcommand{\R}{\mathbb R}
\newcommand{\C}{\mathbb C}
\newcommand{\Sph}{\mathbb S}
\newcommand{\Tr}{\operatorname{Tr}}
\newcommand{\spec}{\operatorname{spec}}
\newcommand{\Scal}{\operatorname{Scal}}
\newcommand{\rank}{\operatorname{rank}}
\newcommand{\diag}{\operatorname{diag}}
\newcommand{\ip}[2]{\langle #1,#2\rangle}

\title[A solution to Lu's second gap conjecture]{A solution to Lu's second gap conjecture}

\author[W. R. Ding]{Weiran Ding$^{1}$}
\address{$^{1}$School of Mathematical Sciences, South China Normal University, Guangzhou 510000, P. R. CHINA.}
\email{dingwr0806@m.scnu.edu.cn}

\author[F. G. Li]{Fagui Li$^{2,*}$}
\address{$^{2,*}$Frontier Interdisciplinary Domain, Beijing Institute of Technology, Zhuhai, Guangdong 519088, P. R. CHINA.}
\email{lifagui@bitzh.edu.cn}

\author[X. Z. Yang]{Xize Yang$^{3}$}
\address{$^{3}$Chern Institute of Mathematics and LPMC, Nankai University, Tianjin 300071, P. R. CHINA.}
\email{xize.yang@nankai.edu.cn}

\author[Y. H. Zhang]{Yunheng Zhang$^{4}$}
\address{$^{4}$School of Mathematical Sciences, Laboratory of Mathematics and Complex Systems, Beijing Normal University, Beijing 100875, P. R. CHINA.}
\email{yunheng@mail.bnu.edu.cn}

\subjclass[2020]{53C42, 53C24}
\keywords{Minimal submanifold, Lu's conjecture, second gap, scalar curvature, codimension}
\thanks{* the corresponding author.}

\begin{document}

\begin{abstract}
%Let $M^n\to\Sph^{n+q}(1)$ be a closed connected minimal immersion, where $n\ge3$, and set $Q=S+\lambda_2$, with $S=|h|^2$ and $\lambda_2$ the second largest eigenvalue of Lu's fundamental matrix. We determine the sharp codimension threshold for Lu's second-gap conjecture. For every $2\le q\le n$, there exists $\gamma_{n,q}>0$ such that constant $Q>n$ implies $Q\ge n+\gamma_{n,q}$; together with the classical hypersurface case, the conjecture therefore holds for all $1\le q\le n$. Conversely, for every $q\ge n+1$ we construct closed connected homogeneous minimal embeddings, followed when necessary by totally geodesic inclusions, with constant scalar curvature and constant $Q$-values dense in $(n,2n)$. Hence $q=n+1$ is the exact threshold between the gap and nongap regimes. In codimension two we further obtain the explicit admissible gap $\gamma_{n,2}=\exp(-10^{16}n^2)$.

Let $M^n\to\Sph^{n+q}(1)$ be a closed connected minimal immersion, where $n\ge3$, and set $Q=S+\lambda_2$, with $S=|h|^2$ and $\lambda_2$ the second largest eigenvalue of Lu's fundamental matrix. We determine the sharp codimension range for Lu's second-gap conjecture. For every $2\le q\le n$, there exists $\gamma_{n,q}>0$ such that, if $Q$ is constant and $Q>n$, then $Q\ge n+\gamma_{n,q}$. Conversely, for every $q\ge n+1$, we construct closed connected homogeneous minimal embeddings, followed when necessary by totally geodesic inclusions, with constant scalar curvature and constant $Q$-values dense in $(n,2n)$. Thus, in every dimension $n\ge3$, Lu's conjecture holds precisely for $q\le n$. Combined with the theorem of Peng-Terng for hypersurfaces and the recent resolution of the two-dimensional case, this gives a complete resolution of Lu's second-gap conjecture: for every $n\ge2$, the conjecture holds exactly when $q\le n$ and fails when $q\ge n+1$. In codimension two we further obtain the explicit admissible gap $\gamma_{n,2}=\exp(-10^{16}n^2)$.
\end{abstract}

\maketitle

\section{Introduction}
Let $F:M^n\to\Sph^{n+q}(1)$ be an isometric minimal immersion of a closed connected manifold into the unit sphere. We write $\Sph^k$ for the unit round sphere and $\Sph^k(r)$ for the round sphere of radius $r$. Denote by $h$ the second fundamental form and by $A^1,\cdots,A^q$ the shape operators with respect to a local orthonormal normal frame. We use the following notation:
\[
\begin{aligned}
\mathcal A&=(\langle A^\alpha, A^\beta\rangle)_{q\times q},\\
\lambda_1&\ge\lambda_2\ge\cdots\ge\lambda_q\ge0,\\
S&=|h|^2=\Tr\mathcal A,\\
Q&=S+\lambda_2,
\end{aligned}
\]
where $\mathcal A$ denotes Lu's fundamental matrix, which we simply call the fundamental matrix, and its eigenvalues $\lambda_i$ are counted with multiplicity. We set $\lambda_2=0$ when $q=1$. By the Gauss equation, the scalar curvature of $M$ is
\[
\Scal_M=n(n-1)-S.
\]
Thus, constant scalar curvature is equivalent to constant $S$. For closed minimal hypersurfaces in the unit sphere, Chern's conjecture asserts that, in each fixed dimension, the possible constant values of $S$ form a discrete set \cite{CDK}. Simons' theorem \cite{Simons} and the equality classification of Chern, do Carmo, and Kobayashi \cite{CDK} established the first gap and its rigidity, while Peng and Terng \cite{PengTerng,PengTerng83} established a positive second gap above $S=n$ under constant scalar curvature. Subsequent improvements were obtained by Yang and Cheng \cite{yang_cheng_1998}, Suh and Yang \cite{suh_yang_2007}, Ding and Xin \cite{DingXin11}, and Lei, Xu, and Xu \cite{LeiXuXu17}. Recently, for closed minimal hypersurfaces with constant $S$ and constant cubic trace $f_3=\Tr(A^3)$, where $A$ is the shape operator, Ge, Tan, Yan, and Zhang \cite{GTYZ26} and Guan \cite{Guan26} proved the sharp second-gap estimate $S>n\Rightarrow S\ge2n$ and characterized the equality case by the minimal Cartan isoparametric hypersurfaces with three principal curvatures. Tan, Tang, Xie, and Yan \cite{TTXY26} proved that the constant values of $S$ are locally finite for closed embedded minimal hypersurfaces with constant cubic trace $f_3$. Their theorem allows both the topology and the constant value of $f_3$ to vary. For related classification results and sufficient conditions for isoparametricity, see \cite{Chang93,DengGuWei17,GeLiuLuoYan26,HeXuZhao26,TangWeiYan20,TangYan23}, etc.\par
In higher codimension, the discreteness problem has a different answer. Firester and Tsiamis \cite{FT26} constructed closed embedded minimal submanifolds whose constant $S$-values are dense in a bounded interval for every $n\ge3$, $q\ge4$, and also for even $n\ge4$, $q\ge3$. These examples disprove the higher-codimensional formulation of Chern's discreteness conjecture. Positive gap results nevertheless hold under additional geometric assumptions. In particular, Ge, Li, and Zhang \cite{GLZFlatGap} proved that, for $n\ge3$ and $q\ge2$, closed minimal submanifolds with constant $S$ and flat normal bundle admit a second gap above $S=n$, with a gap size at least $\frac{n}{87}$. For the first gap in higher codimension, Li and Li \cite{LiLi} and Chen and Xu \cite{ChenXu} improved the classical pinching range to $\frac{2n}{3}$. More recently, Li and Zhao \cite{LiZhaoGap}, Lei \cite{Lei26}, and Xu and Zhao \cite{XuZhao26} obtained explicit improvements of this range for $n\ge3$, without assuming either constant scalar curvature or a flat normal bundle.\par
Lu \cite[Theorem 6]{Lu11} refined the first gap theorem by using the quantity $Q=S+\lambda_2$: a closed minimal submanifold satisfying $0\le Q\le n$ is totally geodesic, a Clifford hypersurface in a great sphere, or a Veronese surface. In particular, for $n\ge3$, the equality $Q=n$ occurs only for a Clifford hypersurface. Based on this result, Lu proposed the following second-gap conjecture.

\begin{conjecture}[Lu's second-gap conjecture \cite{Lu11}]\label{conj:lu-second-gap}
For every pair $(n,q)$ there exists $\varepsilon(n,q)>0$ with the following property. If $M^n$ is a closed minimal submanifold of $\Sph^{n+q}$ and $S+\lambda_2$ is constant with $S+\lambda_2>n$, then
\[
S+\lambda_2>n+\varepsilon(n,q).
\]
\end{conjecture}
The two-dimensional case of Lu's second gap conjecture has been settled by a series of recent results. Ding, Ge, Li, and Yang \cite{DGLY} proved the conjecture for minimal $2$-spheres in arbitrary codimension and obtained gap estimates for general closed minimal surfaces under normal scalar curvature pinching conditions. Ge, Li, and Zhang \cite{GLZLuSurfaces} classified closed minimal surfaces in $\Sph^4$ with constant $Q$: the only possible values are $Q=0$ and $Q=2$, with no assumption that $S$ is constant. On the other hand, Li and Zhao \cite{LZ26} constructed linearly full closed embedded flat minimal tori in every odd codimension $q\ge3$, with $S\equiv2$ and constant $Q$-values dense in $(2,3)$. Totally geodesic inclusions yield counterexamples in every codimension $q\ge3$. Together with the hypersurface case, these results show that the conjecture for minimal surfaces holds precisely in codimensions one and two. A recent refinement of Lu's first pinching theorem was obtained by Liu and Yang \cite{LiuYang26}. They proved that the condition
\[
\sum_{\alpha=1}^{\min\{n,q\}}\lambda_\alpha+\lambda_2\le n
\] forces the expression on the left to be identically zero or identically $n$, and classified the equality cases in dimensions two and three. The expression agrees with $Q$ when $q\le n$ and is no larger than $Q$ in general.\par
In this paper, we determine the sharp codimension threshold for Lu's second-gap conjecture in every dimension $n\ge3$. We first give an explicit family whose constant $Q$-values are dense in $(n,2n)$ in codimensions $q\ge n+1$, even with constant scalar curvature and fixed topology. Let $\ell_0,\ell_1$ be coprime positive integers, with $\ell_0$ odd and
\[
\rho=\frac{\ell_1^2}{\ell_0^2}\in(0,\frac{1}{n}),\qquad
r^2=\frac{1-n\rho}{n(1-\rho)},\qquad
s^2=\frac{n-1}{n(1-\rho)}.
\]
Consider the map
\begin{equation}\label{eq:intro-family}
F_{\ell_0,\ell_1}(t,x)=\left(re^{i\ell_0t},se^{i\ell_1t}x\right),\qquad
(t,x)\in\R/(2\pi\mathbb Z)\times\Sph^{n-1},
\end{equation}
with values in $\Sph^{2n+1}\subset\C\oplus\C^n$. For $n=3$, this family coincides with that of Firester and Tsiamis \cite[equation (9)]{FT26} with $(a,b)=(\ell_0,\ell_1)$, up to interchanging the complex coordinate blocks. In every dimension, these maps are steady spiral products in the terminology of Li and Zhang \cite[Section~3.2]{LiZhangSpiral26}, with factor dimensions $0$ and $n-1$.
\begin{theorem}\label{thm:explicit-sequence}
For every $n\ge3$, the maps \eqref{eq:intro-family} are closed connected homogeneous minimal embeddings, linearly full in $\Sph^{2n+1}$. Their fundamental matrices have constant spectrum
\[
\left\{n(1-n\rho),(2n\rho)^{[n-1]},0\right\},
\]
where the bracketed superscript denotes multiplicity. In particular,
\[
S=n+n(n-2)\rho,\qquad Q=n+n^2\rho.
\]
As the admissible integer pairs $(\ell_0,\ell_1)$ vary, their $Q$-values are dense in $(n,2n)$. All the domains have the same topology $\Sph^1\times\Sph^{n-1}$. For every fixed $q>n+1$, composition with a totally geodesic inclusion into $\Sph^{n+q}$ preserves these conclusions except for linear fullness. The additional eigenvalues of the fundamental matrix are zero.
\end{theorem}
Consequently, Lu's second-gap conjecture fails for every $n\ge3$
and $q\ge n+1$, even within the constant-scalar-curvature class.
Our positive results require constancy of $Q$ alone. We retain the
codimension-two statement as Theorem~\ref{thm:q2-onlyQ-low-dimension-gap}
and make its constant explicit.

\begin{theorem}\label{thm:q2-onlyQ-low-dimension-gap}
For every $n\ge3$, set
\begin{equation}\label{eq:explicit-main-gap}
 \gamma_n=\exp(-10^{16}n^2).
\end{equation}
Every closed connected minimal immersion $M^n\to\Sph^{n+2}(1)$
with constant $Q=S+\lambda_2$ satisfies
\[
 Q>n\quad\Longrightarrow\quad Q\ge n+\gamma_n.
\]
\end{theorem}

The next theorem establishes the positive gap throughout the remaining
codimensions below the counterexample threshold.

\begin{theorem}\label{thm:gap-low-codimension}\label{higher:gap}
For every $n\ge3$ and $2\le q\le n$, there exists
$\gamma_{n,q}>0$ such that every closed connected minimal immersion
$M^n\to\Sph^{n+q}(1)$ with constant $Q=S+\lambda_2$ satisfies
\[
 Q>n\quad\Longrightarrow\quad Q\ge n+\gamma_{n,q}.
\]
For $q=2$, one may take $\gamma_{n,2}=\gamma_n$ from
\eqref{eq:explicit-main-gap}.
\end{theorem}

Together with the classical hypersurface second gap
\cite{PengTerng,PengTerng83}, Theorems~\ref{thm:explicit-sequence}
and~\ref{thm:gap-low-codimension} show that, for $n\ge3$, Lu's
second-gap conjecture holds precisely when $1\le q\le n$.
The explicit value in \eqref{eq:explicit-main-gap} is a conservative
admissible constant; no optimality is asserted.

We describe the geometric argument for
Theorem~\ref{thm:gap-low-codimension}. A bound for $Q$ gives smooth
compactness for complete pointed minimal immersions. We extend Lu's
first pinching classification to complete domains in arbitrary
codimension. The curvature inequality $Q\le n$ implies nonnegative
Ricci curvature and a positive lower bound for the sum of its two
lowest eigenvalues. The linear-volume-growth theorem of Cucinotta and
Mondino \cite{CucinottaMondino26}, combined with a regularized
Schatten-power inequality, makes the largest normal eigenvalue
constant and yields the Clifford classification.

For a hypothetical sequence with constant $Q_j>n$ tending to $n$,
we normalize the $q-1$ complementary shape operators by
$\sqrt{\max_{M_j}\lambda_{2,j}}$. A vector-valued Codazzi system
gives uniform local estimates. A quantitative choice of normal frame
preserves exactness of the limiting normal-field differentials and
excludes compact pointed limits. The remaining limits are complete
Clifford cylinders. A polynomial estimate for the first shape operator
then passes the constant-$Q$ equation to the normalized limit.
The limiting Jacobi fields define a positive semidefinite matrix
$H=\sum_{\alpha=2}^q w_\alpha w_\alpha^T$. A two-point maximum-principle
argument forces $H$ to be a nonzero scalar matrix. Thus
$n=\rank H\le q-1$, which contradicts $q\le n$.

For Theorem~\ref{thm:q2-onlyQ-low-dimension-gap},
Section~\ref{sec:explicit-q2} supplies explicit global entry estimates
and quantitative rigidity in a specified neighborhood of the Clifford
orbits. All numerical constants needed to obtain
\eqref{eq:explicit-main-gap} are tracked there.

The paper is organized as follows. Section~\ref{sec:preliminaries}
records the basic identities, and Section~\ref{sec:explicit-examples}
proves Theorem~\ref{thm:explicit-sequence}.
Section~\ref{sec:q2-gap} develops the compactness and scalar Jacobi
tools. Section~\ref{sec:low-codimension-gap} proves the complete
pinching theorem and Theorem~\ref{thm:gap-low-codimension}.
Finally, Section~\ref{sec:explicit-q2} proves the explicit
codimension-two estimate in
Theorem~\ref{thm:q2-onlyQ-low-dimension-gap}.

\section{Preliminaries}\label{sec:preliminaries}
Unless completeness is explicitly assumed instead, $F:M^n\to\Sph^{n+q}(1)$ denotes an isometric minimal immersion of a closed connected manifold into the unit sphere and $g$ is the metric induced by the immersion. Tangent indices range over $1,\cdots,n$ and normal indices over $1,\cdots,q$, with local relabeling stated when needed. Repeated tangent indices are summed unless otherwise indicated; sequence labels, such as the subscript in $A_j$, are never summed. We identify symmetric two-tensors with self-adjoint endomorphisms using $g$, and tangent vectors with their images under $dF$ when convenient. For a normal vector $\xi$, the shape operator is defined by
\[
\langle A_\xi X,Y\rangle=\langle h(X,Y),\xi\rangle.
\]
Given an orthonormal normal frame $\nu_1,\ldots,\nu_q$, set $A^\alpha=A_{\nu_\alpha}$. The fundamental matrix represents the self-adjoint normal-bundle endomorphism determined by
\[
\langle\mathcal A\xi,\eta\rangle=\langle A_\xi,A_\eta\rangle,
\qquad
\mathcal A_{\alpha\beta}=\langle A^\alpha,A^\beta\rangle.
\]
For matrices, $\langle U,V\rangle=\Tr(U^TV)$, $|U|^2=\langle U,U\rangle$, and $[U,V]=UV-VU$. The notation $\|U\|_{\mathrm{op}}=\sup_{|v|=1}|Uv|$ denotes the operator norm. All other tensor norms are induced by the tangent and normal metrics. A superscript $T$ denotes transpose, and $I$ denotes the identity on the indicated space. Let $D$ be the Euclidean connection, $\nabla$ the Levi-Civita connection and its induced tensor connections, and $\nabla^\perp$ the normal connection in the sphere. The Gauss and Weingarten formulas and the minimality equation are
\begin{equation}\label{eq:gauss}
\begin{aligned}
D_XdF(Y)&=dF(\nabla_XY)+h(X,Y)-g(X,Y)F,\\
D_X\xi&=-dF(A_\xi X)+\nabla_X^\perp\xi,\\
\Tr_g h&=0,\\
\Delta_gF+nF&=0.
\end{aligned}
\end{equation}
Thus the Euclidean second fundamental form is $\mathrm{II}^E=h-gF$ and satisfies $|\mathrm{II}^E|^2=S+n$. The covariant derivative of $h$ uses both the tangent and normal connections. For a normal field $\xi$, define the three-tensor $T_\xi=\langle\nabla h,\xi\rangle$, and write
\[
T_\alpha=T_{\nu_\alpha},\qquad
T_{\alpha;i}(X,Y)=\langle(\nabla_{e_i}h)(X,Y),\nu_\alpha\rangle.
\]
The derivative of an individual shape operator uses only the tangent connection; explicitly,
\[
\nabla_X A_\xi=T_\xi(X,\cdot,\cdot)+A_{\nabla_X^\perp\xi}.
\]
Here the two-tensor on the right is identified with an endomorphism. The notation $\nabla^\perp\mathcal A$ denotes the connection induced on normal-bundle endomorphisms, so that
\[
(\nabla_X^\perp\mathcal A)\xi
=\nabla_X^\perp(\mathcal A\xi)-\mathcal A(\nabla_X^\perp\xi).
\]
We use $\Delta=\operatorname{div}\nabla$ on functions. On tensor fields and bundle sections, $\Delta$ is the rough Laplacian with the appropriate induced connection:
\[
\Delta U=\sum_i\bigl(\nabla_{e_i}\nabla_{e_i}U
                 -\nabla_{\nabla_{e_i}e_i}U\bigr).
\]
In particular, $\Delta^\perp\mathcal A$ uses the induced normal-bundle connection, whereas $\Delta A_\xi$ uses the tangent tensor connection. Eigenvalues of the scalar Laplacian refer to $-\Delta$. We set $\operatorname{Hess}f=\nabla df$ and, for a symmetric two-tensor $B$ and a one-form $\theta$, use
\[
(\operatorname{div}B)_j=\sum_i\nabla_iB_{ij},\qquad
(B\theta)_i=\sum_jB_{ij}\theta_j.
\]
The latter convention uses the metric to identify one-forms and vectors. Our curvature convention is $R(X,Y)=\nabla_X\nabla_Y-\nabla_Y\nabla_X-\nabla_{[X,Y]}$, and $R^\perp$ uses the same convention. The Gauss, Codazzi and Ricci equations are
\begin{align*}
\langle R(X,Y)Z,W\rangle&=\langle X,W\rangle\langle Y,Z\rangle-\langle X,Z\rangle\langle Y,W\rangle\\
&\quad+\langle h(X,W),h(Y,Z)\rangle-\langle h(X,Z),h(Y,W)\rangle,\\
(\nabla_Xh)(Y,Z)&=(\nabla_Yh)(X,Z),\\
\langle R^\perp(X,Y)\xi,\eta\rangle&=\langle[A_\xi,A_\eta]X,Y\rangle.
\end{align*}
Consequently, each $T_\alpha$ is symmetric in all three tangent indices and trace-free in every pair, and
\[
\operatorname{Ric}(X,Y)=(n-1)g(X,Y)
-\sum_\alpha\langle A^\alpha X,A^\alpha Y\rangle,
\qquad \Scal_M=n(n-1)-S.
\]
Every trace-free symmetric endomorphism $A$ satisfies the elementary bound
\[
\|A\|_{\mathrm{op}}^2\le\frac{n-1}{n}|A|^2
\qquad(\Tr A=0).
\]
Indeed, if $\kappa_i$ are the eigenvalues of $A$, then $\kappa_i^2=(\sum_{j\ne i}\kappa_j)^2\le(n-1)(|A|^2-\kappa_i^2)$ by Cauchy--Schwarz. The tensor and scalar Simons identities for a minimal immersion in the unit sphere are \cite{Lu11,Simons}
\begin{equation}\label{eq:simons-tensor}
(\Delta h)^\alpha=nA^\alpha-\sum_\beta\langle A^\alpha,A^\beta\rangle A^\beta-\sum_\beta[A^\beta,[A^\beta,A^\alpha]],
\end{equation}
and
\begin{equation}\label{eq:simons-scalar}
\frac12\Delta S=|\nabla h|^2+nS-\sum_{\alpha,\beta}\langle A^\alpha,A^\beta\rangle^2-\sum_{\alpha,\beta}|[A^\alpha,A^\beta]|^2.
\end{equation}
In \eqref{eq:simons-tensor}, $(\Delta h)^\alpha=\langle\Delta h,\nu_\alpha\rangle$ includes the normal connection and is not, in general, equal to $\Delta A^\alpha$. All integrals use the Riemannian measure $d\mu_g$, which is suppressed when no ambiguity arises. Unless a different norm or domain is specified, $O(\tau_j)$ and $o(\tau_j)$ for tensor sequences mean uniform pointwise bounds on the domains under discussion: respectively, $\sup|U_j|\le C\tau_j$ and $\frac{\sup|U_j|}{\tau_j}\to0$, for positive $\tau_j$. Estimates in $C^k$ specify the norm and the region on which derivatives are controlled.\par
We first give the reduction of codimension when the fundamental matrix has rank one.
\begin{proposition}\label{prop:rank-one}
Let $n\ge3$. If a connected minimal immersion has $\rank\mathcal A=1$ everywhere, its image lies in a fixed totally geodesic $\Sph^{n+1}$. If the immersion is complete and $S\equiv n$, it is a covering immersion of a standard Clifford hypersurface
\[
C_p=\Sph^p\left(\sqrt{\frac{p}{n}}\right)\times
   \Sph^{n-p}\left(\sqrt{\frac{n-p}{n}}\right),\qquad 1\le p\le n-1.
\]
In particular, if $M$ is closed, $q\ge1$, $\lambda_2\equiv0$, and $Q>0$ is constant, the classical constant-$S$ hypersurface second gap theorem applies.
\end{proposition}
\begin{proof}
Write $h=A\nu$ locally, where $A\ne0$ is trace-free and $\nu$ spans the first normal line. If $\eta\perp\nu$ and $\theta(X)=\ip{\nabla_X^\perp\nu}{\eta}$, the Codazzi equation in the $\eta$ direction gives $A_{ij}\theta_k=A_{ik}\theta_j$. If $\theta\ne0$ at a point, choose the first coframe vector in its direction. All columns of $A$ except the first then vanish. Symmetry and the trace-free condition give $A=0$, a contradiction. Thus the first normal line is parallel. The rank-$(n+2)$ subbundle $E=\operatorname{span}\{F,dF(TM),\nu\}$ of the trivial Euclidean ambient bundle is preserved by $D$, by the Gauss and Weingarten formulas. Its orthogonal projection is therefore constant. Since the domain is connected, $E$ is a fixed ambient subspace. Assume now that the immersion is complete and $S\equiv n$. Since the fundamental matrix has rank one, the scalar Simons identity reduces to
\[
0=|\nabla h|^2+nS-S^2=|\nabla h|^2.
\]
Thus $h$ is parallel. On the complete simply connected cover choose a unit normal $\eta$. Its shape operator $A$ is parallel and trace-free, with squared norm $n$. The eigendistributions of $A$ are parallel, so a plane spanned by vectors in distinct eigendistributions has zero sectional curvature. The Gauss equation gives $\lambda\mu=-1$ for every pair of distinct eigenvalues. There must be exactly two, with multiplicities $p,m=n-p$ and values $\lambda=-\sqrt{\frac{m}{p}}$, $\mu=\sqrt{\frac{p}{m}}$, after choosing the sign of $\eta$. Let $E_\lambda,E_\mu$ denote the corresponding eigendistributions. Define
\[
F_1=\frac{\eta+\mu F}{\mu-\lambda},\qquad
F_2=\frac{-\eta-\lambda F}{\mu-\lambda}.
\]
Their differentials are the tangent projections onto the two eigenbundles. Moreover, the two maps are pointwise orthogonal, with $|F_1|^2=\frac{p}{n}$, $|F_2|^2=\frac{m}{n}$. The spaces $\operatorname{span}\{F_1,E_\lambda\}$ and $\operatorname{span}\{F_2,E_\mu\}$ are fixed orthogonal Euclidean subspaces: for example,
\[
D_UZ=\nabla_UZ+\ip UZ(\lambda\eta-F),\qquad
\lambda\eta-F=-(1+\lambda^2)F_1
\quad (Z\in E_\lambda).
\]
Thus the lifted immersion is a local isometry onto $C_p$. A local isometry from a complete connected manifold into a connected Riemannian manifold is a surjective covering, as follows by extending geodesic and path lifts. Since the original immersion has the same image and is also a complete local isometry into $C_p$, it too is a surjective covering. The simply connected lifted source is therefore the universal cover of $C_p$. If both factors have dimension at least two, this is the compact product itself. Otherwise it is the full cylinder $\R\times\Sph^{n-1}(\sqrt{(n-1)/n})$.
\end{proof}

We next work locally in codimension two on the open set where the eigenvalues of the fundamental matrix are distinct. Write
\begin{equation}\label{def:a,b,delta,varepsilon}
a=\lambda_1=2S-Q,\quad b=\lambda_2=Q-S,\quad
\delta=a-b>0,\quad \varepsilon=Q-n=a+2b-n.
\end{equation}
These quantities need not be constant in the identities immediately below. Choose an orthonormal normal eigenframe $\nu_1,\nu_2$, with corresponding shape operators $A,B$. Then $|A|^2=a$, $|B|^2=b$, and $\ip AB=0$. Set
\[
\nabla^\perp\nu_1=\omega\nu_2,\qquad
\nabla^\perp\nu_2=-\omega\nu_1,\qquad
T_\alpha=(\nabla h)^\alpha.
\]
In particular,
\begin{equation}\label{nablaAB}
\nabla A=T_1+\omega\otimes B,\qquad
\nabla B=T_2-\omega\otimes A.
\end{equation}
The derivatives on the left use only the tangent tensor connection. The following two formulas keep this distinction explicit:
\begin{lemma}
	With the notation above, let $\operatorname{div}\omega=\sum_i\nabla_i\omega_i$. Then
	\begin{align}
		\frac12\Delta a&=|T_1|^2+(a-b)|\omega|^2+a(n-a)-|[A,B]|^2,\label{eq:eigenvalue-laplacian-general}\\
		\Delta A&=(n-a)A-[B,[B,A]]+2\sum_i\omega_i\nabla_iB+(\operatorname{div}\omega)B+|\omega|^2A,\label{eq:tangent-simons-general}\\
		\Delta a&\ge-2a\varepsilon,\label{eq:Delta a ge}
	\end{align}
where $a,b,\varepsilon$ are defined in \eqref{def:a,b,delta,varepsilon}. If $Q$ is constant, inequality \eqref{eq:Delta a ge} also holds at coincident normal eigenvalues. 
\end{lemma}
\begin{proof}
 We begin by deriving \eqref{eq:eigenvalue-laplacian-general}. Differentiating the diagonal fundamental matrix with respect to its normal-bundle connection yields \[(\nabla^\perp\mathcal{A})_{12}=(\nabla^\perp\mathcal{A})_{21}=(a-b)\omega,\] hence
\begin{equation}\label{eq1}
(\Delta^\perp\mathcal A)_{11}
 =\Delta a-2(a-b)|\omega|^2,
\end{equation}
where the second term arises from the commutator of the normal connection and the off-diagonal components of $\mathcal{A}$.
On the other hand, applying \eqref{eq:simons-tensor} to the first normal direction gives
\begin{equation}\label{eq2}
\frac12(\Delta^\perp\mathcal A)_{11}
   =|T_1|^2+na-a^2-|[A,B]|^2.
\end{equation}
Combining \eqref{eq1} with \eqref{eq2} yields \eqref{eq:eigenvalue-laplacian-general}.

For \eqref{eq:tangent-simons-general}, calculate at a tangent geodesic frame:
\begin{equation}\label{eq:Delta h^1}
(\Delta h)^1=\sum_i\nabla_iT_{1;i}-\sum_i\omega_iT_{2;i}.
\end{equation}
Differentiating \eqref{nablaAB}, and substituting $T_{1;i}=\nabla_iA-\omega_iB$  and
$T_{2;i}=\nabla_iB+\omega_iA$ into \eqref{eq:Delta h^1}, we obtain
\[(\Delta h)^1=\Delta A-2\sum_i\omega_i\nabla_iB-(\operatorname{div}\omega)B-|\omega|^2A,\] 
and then applying
\eqref{eq:simons-tensor} gives \eqref{eq:tangent-simons-general}. Diagonalizing $A$ also gives the commutator estimate
\begin{equation}\label{eq:commutator-bound}
|[A,B]|^2=\sum_{i,j}(A_{ii}-A_{jj})^2B_{ij}^2\le\operatorname{osc}(A)^2|B|^2\le2ab,
\end{equation}
where $\operatorname{osc}(A)$ is the difference between the largest and smallest eigenvalues of $A$. If $Q=n+\varepsilon$ is constant, the ordered eigenvalues $a=2S-Q$ and $b=Q-S$ are smooth, even where they coincide. At any point, choose a unit top normal eigenvector $\nu$ and extend it locally as a unit field with $\nabla^\perp\nu=0$ at that point. The function $a-\langle\mathcal A\nu,\nu\rangle$ has a local minimum equal to zero. Moreover, $\langle\Delta^\perp\nu,\nu\rangle=0$ at the point, so the eigenvector property removes the second derivatives of $\nu$ from the Laplacian of its Rayleigh quotient. Writing $A=A_\nu$ and $B$ for the other shape operator at that point, \eqref{eq:eigenvalue-laplacian-general} and \eqref{eq:commutator-bound} yield
\begin{equation}\label{eq:q2-onlyQ-subharmonic}
\Delta a\ge2|T_\nu|^2+2na-2a^2-2|[A,B]|^2\ge-2a\varepsilon.
\end{equation}
\end{proof}
\section{Counterexamples for \texorpdfstring{$n\ge3$ and $q\ge n+1$}{n >= 3 and q >= n+1}}\label{sec:explicit-examples}
We prove Theorem \ref{thm:explicit-sequence} by a direct computation. Fix $n\ge3$ and coprime positive integers $\ell_0,\ell_1$, with $\ell_0$ odd and $\rho=\frac{\ell_1^2}{\ell_0^2}\in(0,\frac{1}{n})$. Put
\begin{equation}\label{eq:baseline-radii}
r^2=\frac{1-n\rho}{n(1-\rho)},\qquad
s^2=\frac{n-1}{n(1-\rho)}.
\end{equation}
Both radii are positive and their squares sum to one. Define
\begin{equation}\label{eq:baseline}
F(t,x)=(re^{i\ell_0t},se^{i\ell_1t}x),\qquad
(t,x)\in\R/(2\pi\mathbb Z)\times\Sph^{n-1},
\end{equation}
with values in $\Sph^{2n+1}(1)\subset\C\oplus\C^n$. The ambient Euclidean space has dimension $2n+2$, so the immersion lies in $\Sph^{2n+1}$ and has codimension $n+1$. The cross terms of the induced metric vanish, and
\begin{equation}\label{eq:baseline-metric}
g=\frac{\ell_0^2}{n}\,dt^2+s^2g_{\Sph^{n-1}},
\qquad \ell_0^2r^2+\ell_1^2s^2=\frac{\ell_0^2}{n}.
\end{equation}
The metric is positive definite. Since the frequencies are integers, the map is well defined on the closed connected domain. For the two components of the immersion, we have
\[
\Delta(re^{i\ell_0t})=-nre^{i\ell_0t},\qquad
\Delta(se^{i\ell_1t}x)
=-\left(n\rho+\frac{n-1}{s^2}\right)se^{i\ell_1t}x=-nse^{i\ell_1t}x.
\]
\begin{proof}[Proof of Theorem \ref{thm:explicit-sequence}]
By \eqref{eq:gauss}, the immersion is minimal. To prove injectivity, suppose that two image points agree and put $\delta=t-t'$. Then $e^{i\ell_0\delta}=1$ and $e^{i\ell_1\delta}x=x'$. The latter equality implies that $e^{i\ell_1\delta}=\pm1$, and hence $e^{2i\ell_1\delta}=1$. Since $\gcd(\ell_0,2\ell_1)=1$, $\delta\in2\pi\mathbb Z$ and $x=x'$. The immersion is therefore an embedding. The ambient action
\[
(R,z)(w_0,w)=(z^{\ell_0}w_0,z^{\ell_1}Rw),\qquad (R,z)\in SO(n)\times\Sph^1,
\]
is transitive on the image, so all scalar invariants are constant. For linear fullness, suppose that a real linear functional vanishes on the image. Applying it to $F(t,x)$ and $F(t,-x)$, we see that it vanishes on the first complex line and on $\C^n$. Hence it vanishes on the whole ambient space. We now compute the second fundamental form and the fundamental matrix. Write 
\[u=e^{i\ell_0t},\quad v=e^{i\ell_1t}x,\quad d=\frac{r}{s}, \quad c=\sqrt n\frac{\ell_1}{\ell_0}.\]
 For a local unit frame $Y_j$ on $\Sph^{n-1}$, put
\[
e_0=(i\sqrt n r u,i\sqrt n\frac{\ell_1}{\ell_0}s v),\quad e_j=(0,e^{i\ell_1t}Y_j).
\]
These vectors are orthonormal since $(\sqrt n r)^2+(\sqrt n\frac{\ell_1}{\ell_0}s)^2=1$. An orthonormal spherical normal frame is
\[
\eta=(su,-rv),\quad \zeta=(-i\sqrt n\frac{\ell_1}{\ell_0}s u,i\sqrt n r v),\quad
\xi_j=(0,ie^{i\ell_1t}Y_j),\qquad 1\le j\le n-1.
\]
Choose $Y_j$ to be geodesic at the point under consideration. Direct differentiation and the Gauss formula \eqref{eq:gauss} give
\[
\begin{aligned}
h(e_j,e_k)&=d\delta_{jk}\eta, &h(e_0,e_j)&=c\xi_j, &
h(e_0,e_0)&=-(n-1)d\eta.
\end{aligned}
\]
Indeed, $D_{e_j}e_k=-(0,\delta_{jk}\frac{v}{s})$ and $D_{e_0}e_0=(-nru,-n\rho sv)$. Adding the spherical correction $g_{ij}F$ gives the above normal components. Their trace is zero, as required by minimality. Let $E_{ij}$ be the matrix unit with $(E_{ij})_{k\ell}=\delta_{ik}\delta_{j\ell}$, where $i,j,k,\ell\in\{0,\cdots,n-1\}$. Consequently
\[
A_\eta=d\diag(-(n-1),1,\cdots,1),\quad
A_{\xi_j}=c(E_{0j}+E_{j0}),\quad A_\zeta=0.
\]
These shape operators are mutually orthogonal. Since $d^2=\frac{1-n\rho}{n-1}$, the spectrum of $\mathcal{A}$ is
\[
\spec\mathcal A=\{n(1-n\rho),(2n\rho)^{[n-1]},0\}.
\]
Since $2n\rho$ has multiplicity at least two, the second largest eigenvalue is $2n\rho$ throughout the parameter interval, also when $n(1-n\rho)=2n\rho$. Therefore
\[
S=n+n(n-2)\rho,\qquad Q=n+n^2\rho.
\]
The normal bundle is nonflat because $|[A_\eta,A_{\xi_j}]|^2=2n^2c^2d^2>0$. Although one shape operator vanishes, its normal line bundle is not parallel. Indeed,
\[
\nabla^\perp_{e_j}\zeta=d\sqrt n\xi_j\ne0.
\]
To prove density, fix $Q_*\in(n,2n)$ and put
\[
\tau=\frac{\sqrt{Q_*-n}}n\in(0,\frac{1}{\sqrt n}).
\]
Choose odd primes $p_j\to\infty$ and set $k_j=\lfloor p_j\tau\rfloor$. For all sufficiently large $j$,
\[
1\le k_j<\frac{p_j}{\sqrt n}<p_j,\qquad \gcd(p_j,k_j)=1.
\]
Thus $(\ell_0,\ell_1)=(p_j,k_j)$ is admissible, and
\[
Q_j=n+n^2\left(\frac{k_j}{p_j}\right)^2\rightarrow Q_*.
\]
This proves density in $(n,2n)$ while preserving the fixed domain $\Sph^1\times\Sph^{n-1}$. A totally geodesic inclusion preserves the second fundamental form and adds zero eigenvalues, which proves the assertion for every $q>n+1$.
\end{proof}
In particular, the subfamily $(\ell_0,\ell_1)=(N,1)$, with $N$ odd and $N^2>n$, satisfies
\[Q_N=n+\frac{n^2}{N^2}\downarrow n.\]
For any $\varepsilon>0$, choosing $N^2>\max\{n,\frac{n^2}{\varepsilon}\}$ gives $n<Q_N<n+\varepsilon$.

\section{Compactness and Clifford Jacobi fields}\label{sec:q2-gap}
We collect compactness and scalar Clifford-cylinder tools for the proof of Theorem~\ref{thm:gap-low-codimension}. The scalar normalized system will later be applied to each complementary normal component. We begin with a compactness lemma for complete minimal immersions into a fixed sphere, without assuming that the domains are compact.
 By a pointed immersion $F:(M,g,x)\to\Sph^{n+q}(1)$, we mean an immersion together with a distinguished basepoint $x\in M$.
\begin{lemma}\label{lem:immersion-compactness}
Fix $n\ge2$, $q\ge1$, and $C_0<\infty$. Let $F_j:(M_j,g_j,x_j)\to\Sph^{n+q}(1)$ be complete connected pointed minimal immersions with $|h_j|^2\le C_0$, where $g_j$ is the induced metric. After passing to a subsequence, the immersions converge smoothly on an exhaustion to a complete connected minimal immersion $F_\infty:(M_\infty,g_\infty,x_\infty)\to\Sph^{n+q}(1)$. More precisely, there are embeddings $\phi_j:W_j\to M_j$ of relatively compact open sets exhausting $M_\infty$, with $\phi_j(x_\infty)=x_j$, and
\[
\phi_j^*g_j\longrightarrow g_\infty,
\qquad F_j\circ\phi_j\longrightarrow F_\infty
\quad\hbox{in }C^\infty_{\mathrm{loc}}.
\]
For each fixed $R$, the images of these embeddings contain $B_{g_j}(x_j,R)$ for all sufficiently large $j$. In local coordinates obtained by projection onto the tangent spaces, the immersions have uniform derivative estimates of every order, and
\begin{equation}\label{eq:intrinsic-injectivity-bound}
\operatorname{inj}(M_j,g_j)\ge(C_0+n)^{-\frac{1}{2}}.
\end{equation}
If $M_\infty$ is compact, the convergence is eventually global through diffeomorphisms. If the intrinsic diameters of the $M_j$ are uniformly bounded, the limit is compact and the same conclusion holds.
\end{lemma}
\begin{proof}
Put $\Lambda=\sqrt{C_0+n}$. For the immersion into Euclidean space, we have $|\mathrm{II}^E_j|^2=|h_j|^2+n\le C_0+n$, hence $|\mathrm{II}^E_j|\le\Lambda$. The Euclidean Gauss equation implies $\sec_{g_j}\le\Lambda^2$, by the Rauch comparison theorem, the conjugate radius is at least $\frac{\pi}{\Lambda}$. If $\operatorname{inj}_{g_j}(x_0)<\frac{\pi}{\Lambda}$ at a point $x_0\in M_j$, the nearest-cut-point alternative gives a unit-speed geodesic loop $\gamma:[0,L]\to M_j$, smooth away from its basepoint, with $L=2\operatorname{inj}_{g_j}(x_0)$; see \cite[Chapter~5]{Petersen}. For $z=F_j\circ\gamma$ in Euclidean space, since $\gamma$ is a unit-speed geodesic, $|z'| = 1$. Moreover, $z'' = \Pi_j^E(\gamma', \gamma')$ is purely normal, so $|z''| \le \Lambda$ on $(0, L)$. The identity $z(L)=z(0)$ gives
\[
0=\ip{z(L)-z(0)}{z'(0)}
=\int_0^L\ip{z'(t)}{z'(0)}dt.
\]
Using the fundamental theorem of calculus, $z'(t) = z'(0) + \int_0^t z''(s) ds$. Substituting this and swapping the order of integration (Fubini's theorem) yields:
\[
0 = L + \int_0^L (L-s) \langle z''(s), z'(0) \rangle ds.
\]
By the Cauchy-Schwarz inequality, $\langle z''(s), z'(0) \rangle \ge -|z''(s)||z'(0)| \ge -\Lambda$. Therefore, 
\[0 \ge L - \Lambda \int_0^L (L-s) ds = L - \frac{\Lambda L^2}{2},\]
 which implies $L \ge \frac{2}{\Lambda}$. Substituting $L = 2 \text{inj}_{g_j}(x_0)$ gives $\text{inj}_{g_j}(x_0) \ge \frac{1}{\Lambda}= (C_0+n)^{-\frac{1}{2}}$, proving \eqref{eq:intrinsic-injectivity-bound} at every point. This argument uses neither compactness nor embeddedness.

 We next construct graph coordinates of uniform size. Fix $x_0\in M_j$ and identify $T=dF_j(T_{x_0}M_j)$ with $\R^n$. Let $\pi:\R^{n+q+1}\to T$ be orthogonal projection and set $P=\pi(F_j-F_j(x_0))$. Along a minimizing geodesic of length at most $R_0=(4\Lambda)^{-1}$, a parallel unit tangent vector, viewed in Euclidean space, changes by at most $\Lambda R_0$. It follows that $dP$ is invertible on $B_{g_j}(x_0,R_0)$ and $\|(dP)^{-1}\|_{\mathrm{op}}\le2$. Put $\rho=\frac{R_0}{4}$ and $D_\rho=\{x\in T:|x|<\rho\}$. For $x\in D_\rho$, solve
\[
\dot\gamma_x(t)=(dP_{\gamma_x(t)})^{-1}x,
\qquad \gamma_x(0)=x_0,\qquad 0\le t\le1.
\]
Its length is at most $2|x|<\frac{R_0}{2}$. By completeness, the local solution extends to $t=1$. Smooth dependence on the parameter gives $\Psi(x)=\gamma_x(1)$, with $P\circ\Psi=x$. Thus $\Psi$ is a smooth embedding, and $F_j\circ\Psi$ is a graph with uniformly bounded slope. Its image in $M_j$ contains $B_{g_j}(x_0,\rho)$: project a minimizing path of length less than $\rho$ and use uniqueness of its lift through the local inverse of $P$. Thus each graph chart contains an intrinsic neighborhood, even if other sheets of the immersion meet its image. After a rotation about the sphere center, write this graph as $F_j(x)=(x,c_j+u_j(x))$, with $u_j(0)=Du_j(0)=0$ and $|c_j|=1$. Here $c_j,u_j$ take values in $\R^{q+1}$, and $\alpha=1,\ldots,q+1$ in the graph equation below labels Euclidean vertical coordinates. If $g_{ik}=\delta_{ik}+\partial_i u_j\cdot\partial_k u_j$, the Euclidean graph formula gives
\begin{equation}\label{graph estimate}
|D^2u_j|^2\le(1+|Du_j|^2)^3|\mathrm{II}^E_j|^2.
\end{equation}
Indeed, for the graph frames $T_i=(e_i,\partial_i u_j)$ and $N_\alpha=(-Du_j^Te_\alpha,e_\alpha)$, we have
\[
\ip{\mathrm{II}^E_j(T_i,T_k)}{N_\alpha}=\partial_{ik}u_j^\alpha.
\]
To derive the estimate, let
\[
g=(\langle T_i,T_k\rangle)_{i,k}=I_n+(Du_j)^TDu_j,
\qquad
G=(\langle N_\alpha,N_\beta\rangle)_{\alpha,\beta}
=I_{q+1}+Du_j(Du_j)^T.
\]
Writing $(g^{ik})=g^{-1}$ and $(G^{\alpha\beta})=G^{-1}$, we obtain
\[
|\mathrm{II}^E_j|^2
=\sum_{i,k,r,s=1}^n\sum_{\alpha,\beta=1}^{q+1}
g^{ir}g^{ks}G^{\alpha\beta}
(\partial_{ik}u_j^\alpha)(\partial_{rs}u_j^\beta).
\]
Both $g^{-1}$ and $G^{-1}$ are bounded below by $(1+|Du_j|^2)^{-1}I$ as quadratic forms. Since $g^{-1}$ occurs twice and $G^{-1}$ once in this contraction, it follows that
\[
|\mathrm{II}^E_j|^2\ge(1+|Du_j|^2)^{-3}|D^2u_j|^2,
\]
which proves \eqref{graph estimate}. Subtracting $Du_j$ times the horizontal components of $\Delta_gF_j=-nF_j$ from its vertical components yields
\[
g^{ik}(Du_j)\partial_{ik}u_j^\alpha
=-n\big(c_j^\alpha+u_j^\alpha-x^k\partial_k u_j^\alpha\big).
\]
The bounds of the first-order derivative and the second-order derivative give uniform ellipticity and uniform H\"older bounds for the coefficients and the right-hand side. Interior Schauder estimates, applied successively after differentiation, give bounds of every order on a smaller graph ball of fixed radius. They also bound every covariant derivative of $h_j$ and the coordinate derivatives of $g_j$. The transition maps are obtained by tangent projection of the Euclidean graphs. They therefore satisfy the same derivative estimates on the overlaps. The metric estimates give smooth pointed convergence to a complete limit by \cite[Chapter 11]{Petersen}. The local coordinate construction and a diagonal subsequence give embeddings on an exhaustion, with convergence in every derivative order. In these coordinates the components of $F_j\circ\phi_j$ are bounded by one and solve
\[
\Delta_{\phi_j^*g_j}(F_j\circ\phi_j)+n(F_j\circ\phi_j)=0.
\]
Interior elliptic estimates and a further diagonal subsequence give smooth convergence to $F_\infty$. Passing to the limit in $(F_j\circ\phi_j)^*g_{\Sph}=\phi_j^*g_j$ proves $F_\infty^*g_{\Sph}=g_\infty$, so the limit is an isometric immersion. The limiting equation then proves its minimality. 

Finally, we verify the assertion about intrinsic balls. For fixed $R>0$, choose a precompact open $\Omega\subset M_\infty$ containing $\overline B_{g_\infty}(x_\infty,3R)$, with its closure inside an exhaustion region. For large $j$, $\phi_j^*g_j\ge \frac{g_\infty}{4}$ on $\overline\Omega$. A path from $x_j$ first exiting $\phi_j(\Omega)$ has length at least $\frac{3R}{2}$, by lifting it up to its first exit. Every minimizing path of length less than $R$ therefore stays in $\phi_j(\Omega)$, proving the claimed containment. If $M_\infty$ is compact, the exhaustion eventually contains the whole limit. Its image under $\phi_j$ is both compact and open in the connected manifold $M_j$, and hence equals $M_j$. These maps are global diffeomorphisms and the metric convergence bounds the intrinsic diameters. Conversely, if $\operatorname{diam}(M_j)\le D$, choose $R>D$. The containment of the intrinsic balls implies that $\phi_j|_\Omega$ is surjective onto $M_j$, and hence a diffeomorphism. By completeness and the diameter bound, $M_j$ is compact. Thus $\Omega$ is compact and open in connected $M_\infty$, so $\Omega=M_\infty$.
\end{proof}
Let $m=n-1\ge2$, $s^2=\frac{m}{n}$, and $g_s=s^2g_{\Sph^m}$. Consider the Clifford cylinder
\[
F_0(t,x)=\left(n^{-\frac{1}{2}}e^{i\sqrt n t},sx\right),\qquad
(t,x)\in\R\times\Sph^m,
\qquad g_0=dt^2+g_s.
\]
We identify $\Sph^m(s)$ with $\Sph^m(1)$ equipped with $g_s$ and denote its volume form by $d\mu_s$. Thus $x$ is always a unit vector, while sphere gradients, Hessians, and Laplacians use $g_s$. Choose the normal so that $A=\diag(-\sqrt m,\frac{1}{\sqrt m},\ldots,\frac{1}{\sqrt{m}})$ and $|A|^2=n$. Consider an infinitesimal variation along a fixed, extra ambient normal direction $\zeta$, given by the map $F_{\tau}=\cos(\tau f)F_0+\sin(\tau f)\zeta$. Differentiating at $\tau=0$, the first variation of the metric vanishes because the shape operator in the $\zeta$ direction is zero before the variation. Consequently, the first variation of the shape operator is $B_0=\operatorname{Hess}f+fg_0$, and the linearized minimality equation is $(\Delta_{g_0}+n)f=0$. A rotation of the normal frame removes the component of $B_0$ parallel to $A$. We thus obtain
\begin{equation}\label{B}
B=B_0-\frac{\ip{B_0}{A}}n A.
\end{equation}

\begin{lemma}\label{thm:jacobi}
Let $f$ be a smooth scalar solution of $(\Delta_{g_0}+n)f=0$ on the entire cylinder. If
\[
\sup_{t\in\R}\int_{\Sph^m(s)}|B(t,x)|^2\,d\mu_s(x)<\infty,
\]
then for constants $\alpha,\beta\in\R$ and $v,w\in\R^{m+1}$,
\begin{equation}\label{eq:jacobi-classification}
f(t,x)=\alpha\cos(\sqrt n t)+\beta\sin(\sqrt n t)+(v+tw)\cdot x.
\end{equation}
The projected tensor satisfies
\begin{equation}\label{eq:jacobi-norm}
|B|^2=\frac2{s^2}\left(|w|^2-(w\cdot x)^2\right).
\end{equation}
\end{lemma}
\begin{proof}
Use $e_0=\partial_t$ and a local orthonormal frame $e_1,\cdots,e_m$ for $g_s$. The indices $a,b$ range from $1$ to $m$ and refer to sphere directions. All sphere derivatives in this proof use $g_s$. Writing $B_0=\operatorname{Hess}_{g_0}f+fg_0$, we have
\[
(B_0)_{00}=f_{tt}+f,\qquad
(B_0)_{0a}=\nabla_a f_t,\qquad
(B_0)_{ab}=(\operatorname{Hess}_{g_s}f)_{ab}+f\delta_{ab}.
\]
Taking the inner product with $A$ and using the Jacobi equation, we obtain
\[
\begin{aligned}
\langle B_0,A\rangle
&=-\sqrt m(f_{tt}+f)
+\frac1{\sqrt m}\bigl(\Delta_{g_s}f+mf\bigr)\\
&=-\frac n{\sqrt m}(f_{tt}+f).
\end{aligned}
\]
Substituting into \eqref{B}, we have
\[
B_{00}=0,\qquad B_{0a}=\nabla_a f_t,\qquad
B_{ab}=(\operatorname{Hess}^0 f)_{ab}.
\]
Here, for any smooth function $Y$ on $\Sph^m(s)$, the trace-free sphere Hessian is defined by
\[
\operatorname{Hess}^0Y
\coloneqq\operatorname{Hess}_{g_s}Y-\frac{\Delta_{g_s}Y}{m}\,g_s,
\qquad m=n-1.
\]
For $f=f(t,x)$ this definition is applied to $f(t,\cdot)$ at each fixed $t$. Expanding $f$ in spherical harmonics orthonormal in $L^2(\Sph^m(s),d\mu_s)$, we write 
$f(t,x) = \sum_{d=0}^\infty f_d(t)Y_d(x)$,
where $Y_d$ are spherical harmonics of degree $d$ satisfying $-\Delta_{g_s}Y_d = \kappa_d Y_d$ with eigenvalues $\kappa_d = \frac{d(d+m-1)}{s^2}$. Recall that the product metric \(g_0 = dt^2 + g_s\) allows the Laplace operator to decompose as \(\Delta_{g_0} = \partial_t^2 + \Delta_{g_s}\). Substituting this expansion into the Jacobi equation $(\Delta_{g_0} + n)f = 0$, and using the orthogonality of the harmonics, we obtain the ordinary differential equation for each coefficient:
\begin{equation}\label{ODE}
	f_d''+(n-\kappa_d)f_d=0.
\end{equation} 
For a spherical eigenfunction $Y$ satisfying $-\Delta_{g_s}Y=\kappa Y$, applying the classical Bochner formula on the sphere, 
\[
\frac{1}{2}\Delta_{g_s}|\nabla Y|^2 = |\text{Hess} Y|^2 + \langle \nabla \Delta_{g_s}Y, \nabla Y \rangle + \text{Ric}_{g_s}(\nabla Y, \nabla Y),
\]
and integrating over the closed manifold $\mathbb{S}^m(s)$ yields:
\[
0 = \int_{\mathbb{S}^m(s)} |\text{Hess} Y|^2 d\mu_s - \kappa \int_{\mathbb{S}^m(s)} |\nabla Y|^2 d\mu_s + (m-1)s^{-2} \int_{\mathbb{S}^m(s)} |\nabla Y|^2 d\mu_s.
\]
Since $\int_{\mathbb{S}^m(s)} |\nabla Y|^2 \,d\mu_s= \kappa \int_{\mathbb{S}^m(s)} |Y|^2\,d\mu_s$, this reduces to \[\int_{\mathbb{S}^m(s)} |\text{Hess} Y|^2\,d\mu_s = \left(\kappa^2 - \frac{(m-1)\kappa}{s^2}\right)\int_{\mathbb{S}^m(s)} |Y|^2\,d\mu_s.\] Subtracting $m^{-1}\int_{\Sph^m(s)} (\Delta_{g_s}Y)^2\,d\mu_s = \frac{\kappa^2}{m}\int |Y|^2$ from this integrated squared Hessian norm gives the trace-free estimate:
\[
\int_{\mathbb{S}^m(s)} |\text{Hess}^0 Y|^2_{g_s} d\mu_s = \frac{m-1}{m}\kappa(\kappa - n) \int_{\mathbb{S}^m(s)} |Y|^2 d\mu_s,
\]
where we used $s^2 = \frac{m}{n}$.
 By polarization, the corresponding Hessian tensors are orthogonal whenever the spherical harmonics are orthogonal. For $d\ge2$, we have $\kappa_d>n$. The integral bound for $B$ bounds the trace-free Hessian in the sphere directions and hence each harmonic coefficient on $\R$. Since the solutions to \eqref{ODE} are exponential, $f_d(t) = c_1 e^{\sqrt{\kappa_d - n}t} + c_2 e^{-\sqrt{\kappa_d - n}t}$, boundedness on the whole line implies $c_1 = c_2 = 0$, so $f_d \equiv 0$ for $d \ge 2$. For $d=0$, we have $\kappa_0=0$ and $f_0''+nf_0=0$, which gives the sine and cosine terms in \eqref{eq:jacobi-classification}. For $d=1$, we have $\kappa_1=n$, so the coefficients are affine functions of $t$. This proves the expansion \eqref{eq:jacobi-classification}. The only possibly nonzero components of $B$ are the mixed components, which are given by $\nabla(w\cdot x)$. The squared norm of this gradient is $s^{-2}(|w|^2-(w\cdot x)^2)$, which proves \eqref{eq:jacobi-norm}.
\end{proof}

We recall the normalized equations without imposing constancy of the normal eigenvalues. On a region where $a>b$, let $A,B$ be the normal eigenframe shape operators and let $\omega$ be their normal connection form. For any positive constant $\sigma$, put $C=\frac{B}{\sigma}$ and $\theta=\frac{\omega}{\sigma}$.
Minimality, orthogonality, and the Codazzi and Ricci equations give
\begin{equation}\label{eq:normalized-constraints}
\Tr C=0,\qquad\ip AC=0,
\end{equation}
and
\begin{align}
(\nabla_XC)Y-(\nabla_YC)X
+\theta(X)AY-\theta(Y)AX&=0,\label{eq:normalized-codazzi}\\
d\theta(X,Y)&=\ip{[A,C]X}{Y}.\label{eq:normalized-ricci}
\end{align}
Tracing the Codazzi equation yields
\begin{equation}\label{eq:normalized-div}
\operatorname{div}C=-A\theta.
\end{equation}
No constancy or positive lower bound for $|C|$ is required. The eigenframe coefficients are also controlled under smooth immersion convergence whenever $a-b$ stays bounded away from zero. Indeed, extend $\mathcal A$ by zero to the ambient Euclidean bundle, and let $N$ be the orthogonal projection onto the spherical normal bundle. The normal eigenline projectors are
\[
P_1=\frac{\mathcal A-bN}{a-b},\qquad P_2=N-P_1.
\]
For a local unit section $\nu$ of a rank-one projector $P$, differentiation of $P\nu=\nu$ and $|\nu|=1$ yields $D\nu=(DP)\nu$. Thus all local derivatives of the eigenframe shape operators are controlled by the immersion and the normal spectral gap.

\begin{lemma}\label{lem:local-elliptic}
Let $n\ge3$. On a fixed coordinate domain $U$, suppose that $g_j\to g$ and $A_j\to A$ smoothly, where each $A_j$ is symmetric and trace-free with respect to $g_j$, and $A$ is invertible. For the $j$-th system, all covariant derivatives, contractions, and tensor norms are taken with respect to $g_j$. Let symmetric tensors $C_j$ and one-forms $\theta_j$ satisfy
\[
\Tr_{g_j}C_j=\ip{A_j}{C_j}=0,
\]
\[
(\nabla_XC_j)Y-(\nabla_YC_j)X
   +\theta_j(X)A_jY-\theta_j(Y)A_jX=0.
\]
If $\sup_U|C_j|$ is bounded independently of $j$, then all local derivatives of $C_j$ and $\theta_j$ are uniformly bounded. More precisely, for $U'\Subset U$, $k\ge0$, and all sufficiently large $j$,
\[
\|C_j\|_{C^{k+1}(U')}+\|\theta_j\|_{C^k(U')}
   \le C_k\|C_j\|_{L^2(U)}.
\]
The constants depend on the regions, uniform metric and coefficient bounds, and a uniform bound for $A_j^{-1}$, but not on a positive lower bound for $|C_j|$.
\end{lemma}
\begin{proof}
Tracing the first-order equation gives $\operatorname{div}C_j=-A_j\theta_j$. Eliminating $\theta_j$, consider the operator on
\[
E_A=\{C\in\operatorname{Sym}^2T^*U:\Tr C=\ip AC=0\}
\]
given by
\[
\begin{aligned}
(P_AC)_{ikl}={}&\nabla_iC_{kl}-\nabla_kC_{il}
-(A^{-1}\operatorname{div}C)_iA_{kl}+(A^{-1}\operatorname{div}C)_kA_{il}.
\end{aligned}
\]
Its principal symbol is injective for every nonzero covector when $n\ge3$ and $A$ is invertible. To verify this, we argue by contradiction. After rescaling the covector, choose an orthonormal coframe in which it is $e^1$, without also diagonalizing $A$. Suppose there exists a tensor $C$ in the symbol kernel. Together with the corresponding one-form $\theta$, they must satisfy
\[
\xi_iC_{kl}-\xi_kC_{il}+\theta_iA_{kl}-\theta_kA_{il}=0.
\]
For $i,k\ge2$, a nonzero $\theta_i$ would make the $n-1\ge2$ rows of $A$ with indices at least two proportional, which contradicts the invertibility of $A$. Thus $\theta_i=0$ for $i\ge2$. The equations with $i=1$, $k\ge2$ imply that $C+\theta_1A$ is supported only in entry $11$. Since its trace vanishes, it follows that $C=-\theta_1A$, the constraint $\ip AC=0$ then gives $\theta_1=0$ and $C=0$. This establishes the injectivity of the principal symbol.
\par
Let $P_A^*$ denote the formal $L^2$ adjoint for the induced bundle metrics and the Riemannian volume measure. The injectivity of the principal symbol of $P_A$ implies that $P_A^*P_A$ is strongly elliptic on $E_A$. On relatively compact subdomains, choose smooth local frames of $E_{A_j}$ converging to frames of $E_A$. Because \(A_j \to A\) and \(A\) is invertible, the operators \(A_j\) are uniformly invertible for large \(j\). Combined with the uniform coefficient bounds, this ensures that the strong ellipticity of \(P_{A_j}^* P_{A_j}\) is uniform; crucially, the ellipticity constant does not depend on a positive lower bound for \(|C_j|\). Since $P_{A_j}C_j=0$, interior estimates for elliptic systems \cite{ADN} bound every local Sobolev norm of $C_j$ by $\|C_j\|_{L^2(U)}$. Sobolev embedding gives the asserted bounds for $C_j$, and $\theta_j=-A_j^{-1}\operatorname{div}C_j$ gives those for $\theta_j$.
\end{proof}

The normalized equations admit the following scalar potential formulation.
\begin{lemma}\label{lem:q2-global-potential}
Let $F:M^n\to C_p\subset\Sph^{n+1}$, $n\ge3$, be a connected covering of a standard Clifford hypersurface, with global unit normal $\eta$ and shape operator $A$. Suppose that a symmetric tensor $C$ and a one-form $\theta$ satisfy \eqref{eq:normalized-constraints}-\eqref{eq:normalized-ricci}. Then the $\R^{n+2}$-valued one-form
\[
\Omega(X)=-dF(CX)-\theta(X)\eta
\]
is closed. If $\Omega$ is exact, in particular if $M$ is simply connected, there is a smooth function $f$ such that
\begin{equation}\label{eq:global-jacobi-potential}
\begin{aligned}
(\Delta+n)f&=0,\\
C&=\operatorname{Hess}f+fg-\psi A,\\
\psi&=\frac{\ip{\operatorname{Hess}f+fg}{A}}n,\\
\theta&=d\psi+A(\nabla f,\cdot).
\end{aligned}
\end{equation}
If $M$ is compact and $\Omega$ is exact, then $C=0$.
\end{lemma}
\begin{proof}
The tangential component of $d\Omega$ vanishes by \eqref{eq:normalized-codazzi}, its radial component by symmetry of $C$, and its $\eta$ component is $\ip{[A,C]X}{Y}-d\theta(X,Y)=0$. If $\Omega$ is exact, choose $W:M\to\R^{n+2}$ with $dW=\Omega$ and put $f=-\ip WF$, $\psi=-\ip W\eta$. Differentiating $f$ gives $df(X)=-\langle W,dF(X)\rangle$, so
\[
W=-fF-dF(\nabla f)-\psi\eta.
\]
Comparing the tangential and normal components of $dW=\Omega$, and then using $\Tr C=\ip AC=0$ and $|A|^2=n$, gives \eqref{eq:global-jacobi-potential}. Suppose now that $M$ is compact. If both factors of $C_p$ have dimension at least two, $C_p$ is simply connected and $M=C_p$ as a Riemannian covering. Its Laplace eigenvalues are
\[
\frac np d(d+p-1)+\frac n{n-p}e(e+n-p-1),
\qquad d,e\in\mathbb Z_{\ge0}.
\]
The eigenvalue $n$ occurs only for $(d,e)=(1,0),(0,1)$. If one factor is a circle, a compact connected cover has some degree $k$. Its circle eigenvalues are $\frac{nj^2}{k^2}$, $j\in\mathbb Z$, and its transverse sphere eigenvalues are $\frac{n\ell(\ell+n-2)}{n-1}$, $\ell\ge0$. Their sum equals $n$ only for $(\ell,j)=(0,\pm k)$ or $(1,0)$. Thus in both cases, the eigenfunctions correspond exactly to the restriction of an ambient linear function, i.e., \(f = \langle v, F \rangle\) for a constant vector \(v\). By the Gauss formula on the sphere, \(\text{Hess} f = \langle v, \eta \rangle A - fg\). Substituting this into the definition of \(\psi\) yields \(\psi = \frac{\langle \langle v, \eta \rangle A, A \rangle}{n} = \langle v, \eta \rangle\). 
Consequently, plugging both expressions into \(C = \text{Hess} f + fg - \psi A\) gives \(C = \langle v, \eta \rangle A - \langle v, \eta \rangle A = 0\). This completes the proof.
\end{proof}

\section{The gap in codimensions \texorpdfstring{$2\le q\le n$}{2 <= q <= n}}
\label{sec:low-codimension-gap}

\subsection{Complete pinching and equality limits in arbitrary codimension}
\label{higher:complete-endpoint-section}

The complete pinching classification holds in every codimension.  We
use the linear-volume-growth theorem of Cucinotta and Mondino to handle
complete domains, and a regularized Schatten-power approximation to
avoid differentiability assumptions on the largest normal eigenvalue.

\begin{theorem}\label{higher:endpoint}
Let $n\ge3$ and $q\ge2$.  Every complete connected minimal immersion
$F:M^n\to\mathbb S^{n+q}(1)$ satisfying $Q\le n$ is a covering
immersion of a totally geodesic sphere or a standard Clifford
hypersurface in a fixed totally geodesic $\mathbb S^{n+1}$.
In particular, if $Q\equiv n$, then
\[
 \lambda_2\equiv0,\qquad S\equiv n,\qquad \nabla h\equiv0,
\]
and the immersion is a Clifford covering.
\end{theorem}
\begin{proof}
If $M$ is compact, this is Lu's pinching classification
\cite[Theorem~6]{Lu11}.  Suppose henceforth that $M$ is noncompact.
Write $a=\lambda_1$, $b=\lambda_2$, and
$T=\sum_{\alpha=2}^q\lambda_\alpha$.  We have
\begin{equation}\label{higher:endpoint-elementary}
 S=a+T,\qquad Q=S+b\le n,\qquad
 b\le T\le(q-1)b,\qquad 0\le a\le S\le n.
\end{equation}

We first verify the curvature conditions for linear volume growth.
For any orthonormal normal frame set
$\mathcal C=\sum_\alpha(A^\alpha)^2$, a positive semidefinite
tangent endomorphism with $\operatorname{Tr}\mathcal C=S$.
The Gauss equation gives
\[
 \operatorname{Ric}=(n-1)I-\mathcal C.
\]
The trace-free estimate
$\|A^\alpha\|_{\mathrm{op}}^2\le\frac{n-1}{n}|A^\alpha|^2$
yields
\begin{equation}\label{higher:endpoint-Ricci}
 \operatorname{Ric}\ge\frac{n-1}{n}(n-S)g
 \ge\frac{n-1}{n}b\,g\ge0.
\end{equation}
Let $\rho_1\le\rho_2\le\cdots\le\rho_n$ be the eigenvalues of
the Ricci tensor and let $c_1\ge c_2\ge\cdots\ge c_n\ge0$ be
the eigenvalues of $\mathcal C$.  Then
\begin{equation}\label{higher:endpoint-two-Ricci}
 \rho_1+\rho_2
 =2(n-1)-(c_1+c_2)
 \ge2(n-1)-\operatorname{Tr}\mathcal C
 \ge n-2\ge1.
\end{equation}
Consequently condition~(i) of Cucinotta and Mondino
\cite[Theorem~1.3(1)]{CucinottaMondino26} holds: taking their
truncation level to be $L=1$, the normalized integral on every unit
ball of the truncated sum of the lowest two Ricci eigenvalues is
exactly one.  Their lower-Ricci hypothesis holds because
$\operatorname{Ric}\ge0$.  Hence there is a constant $C$ such that,
for a fixed $o\in M$,
\begin{equation}\label{higher:endpoint-linear-volume}
 \operatorname{Vol}(B_R(o))\le CR\qquad(R\ge1).
\end{equation}
No restriction on the codimension enters these curvature estimates.

We next show that the locally Lipschitz function $a$ satisfies
\begin{equation}\label{higher:top-weak-subharmonic}
 \Delta a\ge2a(n-Q)\ge0
 \quad\text{in the distributional sense}.
\end{equation}
At a point, choose a normal eigenbasis of the fundamental matrix
$\mathcal A$ and denote the shape operators by $A_1,\ldots,A_q$.
Set $C_{\alpha\beta}=|[A_\alpha,A_\beta]|^2$.
Lu's matrix inequality \cite[Lemma~2]{Lu11}, applied after placing
any chosen normal index first, gives
\begin{equation}\label{higher:Lu-every-index}
 \sum_{\beta\ne\alpha}C_{\alpha\beta}
 \le\lambda_\alpha
 \left(S-\lambda_\alpha+
       \max_{\beta\ne\alpha}\lambda_\beta\right).
\end{equation}
Only the inequality is used; no assertion about its equality cases
is needed.

For an even integer $p\ge2$ and $\tau>0$, put
\[
 f_p=(\operatorname{Tr}\mathcal A^p)^{1/p},\qquad
 f_{p,\tau}=(\operatorname{Tr}\mathcal A^p+\tau^p)^{1/p}.
\]
The map $X\mapsto(\operatorname{Tr}X^p+\tau^p)^{1/p}$ is the
Schatten $p$-norm of the block diagonal matrix $\operatorname{diag}(X,\tau)$.
It is convex in $X$ and smooth because $\tau>0$.
The diagonal entries of the matrix Simons identity are
\[
 (\Delta^\perp\mathcal A)_{\alpha\alpha}
 =2|T_\alpha|^2+2n\lambda_\alpha-2\lambda_\alpha^2
       -2\sum_\beta C_{\alpha\beta}.
\]
Convexity and the chain rule with the induced normal-bundle
connection therefore imply
\begin{align}
 \Delta f_{p,\tau}
 &\ge2f_{p,\tau}^{1-p}\sum_\alpha\lambda_\alpha^p
       \left(n-\lambda_\alpha-
        \frac{\sum_\beta C_{\alpha\beta}}{\lambda_\alpha}\right)
       \notag\\
 &\ge
 2\frac{f_p^p}{f_{p,\tau}^{p-1}}(n-Q)
 -2f_{p,\tau}^{1-p}(a-b)
                   \sum_{\alpha=2}^q\lambda_\alpha^p.
 \label{higher:regularized-power-laplacian}
\end{align}
Terms containing $\lambda_\alpha=0$ in the first line are
interpreted as zero.  Before division they have the factor
$\lambda_\alpha^{p-1}$, and $A_\alpha=0$ also gives
$C_{\alpha\beta}=0$.  The second line follows from
\eqref{higher:Lu-every-index}: the expression in parentheses is
at least $n-Q$ for $\alpha=1$, and at least
$(n-Q)-(a-b)$ for every $\alpha\ge2$.

Where $a>0$, write $r=b/a\in[0,1]$ and
$w_1=a^p/\sum_\alpha\lambda_\alpha^p$.  Since
$f_{p,\tau}\ge f_p$, $f_p\le S\le n$, and
$1-w_1\le(q-1)r^p$, the magnitude of the last term in
\eqref{higher:regularized-power-laplacian} is at most
\[
 2f_p(a-b)(1-w_1)
 \le2n^2(q-1)r^p(1-r)
 \le\frac{2n^2(q-1)}{p+1}.
\]
The last inequality follows by maximizing $r^p(1-r)$ on $[0,1]$.
At a point where $a=0$, the same estimate is immediate because
all normal eigenvalues vanish.

Letting $\tau\downarrow0$ in the distributional inequality,
using $f_{p,\tau}\to f_p$ uniformly and
$0\le f_p^p/f_{p,\tau}^{p-1}\le f_p\le n$, gives
\begin{equation}\label{higher:power-almost-subharmonic}
 \Delta f_p\ge2f_p(n-Q)-\frac{2n^2(q-1)}{p+1}
 \quad\text{weakly}.
\end{equation}
Finally $a\le f_p\le q^{1/p}a$ and $a\le n$ show that
$f_p\to a$ uniformly as $p\to\infty$ through even integers.
Testing \eqref{higher:power-almost-subharmonic} against any
nonnegative compactly supported smooth function and passing to the
limit proves \eqref{higher:top-weak-subharmonic}.

We include the cutoff argument that forces $a$ to be constant.
Set $v=n+1-a\ge1$.  It is a positive locally Lipschitz weakly
superharmonic function.  Let $\chi_R$ be a compactly supported
Lipschitz cutoff equal to one on $B_R(o)$, zero outside
$B_{2R}(o)$, and satisfying $|\nabla\chi_R|\le1/R$.
Testing the weak inequality for $v$ with $\chi_R^2/v$ and applying
Cauchy--Schwarz gives
\[
 \int_M\chi_R^2|\nabla\log v|^2
 \le4\int_M|\nabla\chi_R|^2
 \le\frac{4}{R^2}\operatorname{Vol}(B_{2R}(o))
 \le\frac{8C}{R}.
\]
The Lipschitz test function is justified by approximation on its
compact support.  Letting $R\to\infty$ shows that $v$, and hence
$a$, is constant on the connected manifold.

If $a=0$, then $h=0$ and completeness identifies the immersion
as a covering of a totally geodesic sphere, as asserted.
Suppose instead that $a>0$.  Inequality
\eqref{higher:top-weak-subharmonic} now forces $Q=n$ everywhere.
If the constant $a$ were smaller than $n$, then
\eqref{higher:endpoint-elementary} would imply
\[
 n-a=T+b\le qb,\qquad b\ge\frac{n-a}{q}>0.
\]
By \eqref{higher:endpoint-Ricci}, this would give a uniform
positive Ricci lower bound.  The Bonnet--Myers theorem would make
$M$ compact, contradicting our standing assumption.  Therefore
$a=n$, and \eqref{higher:endpoint-elementary} gives $T=b=0$ and
$S=n$.  Proposition~\ref{prop:rank-one} now gives the codimension
reduction and the complete Clifford covering classification.
\end{proof}

\begin{corollary}\label{higher:uniform-endpoint-entry}
Fix $n\ge3$ and $q\ge2$.  Suppose that
$F_j:M_j^n\to\mathbb S^{n+q}(1)$ is a sequence of complete
connected minimal immersions with constant $Q_j>n$ and $Q_j\to n$.
Then
\[
 \sup_{M_j}\lambda_{2,j}\longrightarrow0,\qquad
 \sup_{M_j}|S_j-n|\longrightarrow0,
 \qquad
 \sup_{M_j}|\nabla^k h_j|\longrightarrow0
 \quad\text{for every }k\ge1.
\]
Every sequence of basepoints admits a smooth complete pointed
subsequential limit covering a Clifford hypersurface.
\end{corollary}
\begin{proof}
The bound $S_j\le Q_j$ allows us to apply
Lemma~\ref{lem:immersion-compactness} to arbitrary basepoints.
Every complete pointed limit satisfies $Q=n$ and is classified by
Theorem~\ref{higher:endpoint}.  Choosing basepoints at which any of
the asserted uniform convergences fails contradicts smooth pointed
convergence and the parallel rank-one second fundamental form of
the limit.  Finally $S_j=Q_j-\lambda_{2,j}$ gives the asserted
uniform convergence of $S_j$.
\end{proof}

\subsection{A vector-valued normalized Codazzi system}
\label{higher:normalization-section}

We first record why the normalized estimates used below remain uniform
when the complement of the first normal line has rank greater than one.
The argument uses the unnormalized geometry to fix its coefficients;
it does not assume an estimate for a normalized normal connection.

Let $\nu_1$ be a local unit section of the eigenspace corresponding to
the largest eigenvalue $a$ of the fundamental matrix. Assume that this
eigenvalue is simple. Denote its shape operator by $A$, its orthogonal
normal complement by $\mathcal E$, and the induced metric connection on
$\mathcal E$ by $D$. In a local orthonormal frame
$\nu_2,\ldots,\nu_q$ of $\mathcal E$, write
\[
 h=A\nu_1+\sum_{\alpha=2}^q B_\alpha\nu_\alpha,
 \qquad
 \nabla^\perp\nu_1=\sum_{\alpha=2}^q\omega_\alpha\nu_\alpha.
\]
The tensors $B=(B_\alpha)$ and $\omega=(\omega_\alpha)$ are intrinsically
$\mathcal E$-valued. They satisfy
$\Tr B_\alpha=\ip{A}{B_\alpha}=0$. For a positive constant $\sigma$, set
\[
 C_\alpha=\sigma^{-1}B_\alpha,
 \qquad \theta_\alpha=\sigma^{-1}\omega_\alpha.
\]
Minimality, Codazzi, and Ricci give
\begin{align}
 \Tr C_\alpha=\ip{A}{C_\alpha}&=0,
 \label{higher:normalized-constraints}\\
 (D_iC_\alpha)_{kl}-(D_kC_\alpha)_{il}
       +\theta_{\alpha i}A_{kl}-\theta_{\alpha k}A_{il}&=0,
 \label{higher:normalized-codazzi}\\
 \operatorname{div}^{D}C_\alpha&=-A\theta_\alpha,
 \label{higher:normalized-div}\\
 (d^D\theta)_\alpha(X,Y)&=\ip{[A,C_\alpha]X}{Y}.
 \label{higher:normalized-ricci}
\end{align}
Here $D$ includes the Levi--Civita connection on all tangent slots.
In particular these equations do not require a local eigenframe for
the lower normal block, or a positive lower bound for any $|C_\alpha|$.

\begin{lemma}\label{higher:vector-elliptic}
The assertion of Lemma~\ref{lem:local-elliptic} holds for the system
\eqref{higher:normalized-constraints}--\eqref{higher:normalized-codazzi}
with values in a fixed-rank metric vector bundle, provided that local
connection coefficients and all their derivatives are uniformly
bounded. More precisely, if the metrics, the invertible tensors $A_j$,
and the bundle connections $D_j$ have uniformly bounded smooth
coefficients on $U$, with uniform metric ellipticity and uniformly
bounded $A_j^{-1}$, then for $U'\Subset U$ and $k\ge0$,
\[
 \|C_j\|_{C^{k+1}(U')}+\|\theta_j\|_{C^k(U')}
 \le C_k\|C_j\|_{L^2(U)}.
\]
The constants do not depend on a positive lower bound for $|C_j|$.
\end{lemma}

\begin{proof}
Use \eqref{higher:normalized-div} to eliminate $\theta_j$. The resulting
first-order operator is the operator $P_{A_j}$ in the proof of
Lemma~\ref{lem:local-elliptic}, with each derivative replaced by $D_j$.
Its principal symbol is $q-1$ copies of the same injective symbol.
For completeness, let $\xi=e^1$ and let $(C,\theta)$ belong to its
symbol kernel. For every normal component,
\[
 \xi_iC_{kl}-\xi_kC_{il}
       +\theta_iA_{kl}-\theta_kA_{il}=0.
\]
If one of $\theta_i$, $i\ge2$, is nonzero, the equations with
$i,k\ge2$ make the $n-1$ rows of $A$ indexed by $2,\ldots,n$
proportional. Since $n-1\ge2$, this contradicts invertibility.
Thus $\theta_i=0$ for $i\ge2$. The remaining equations show that
$C+\theta_1A$ has only its $11$ entry possibly nonzero. Its trace is
zero, so $C=-\theta_1A$; orthogonality to $A$ gives $\theta_1=0$.
The argument applies separately to every normal component.

Consequently $P_{A_j,D_j}^*P_{A_j,D_j}$ is uniformly strongly elliptic
on the bundle defined by the two constraints in
\eqref{higher:normalized-constraints}. Choose smooth local frames of
that bundle, whose coefficient bounds follow from those of $A_j$.
Interior elliptic estimates and Sobolev embedding, as in
Lemma~\ref{lem:local-elliptic}, give the asserted estimates; see
\cite{ADN}. The formula for $\theta_j$ then gives its estimates.
\end{proof}

There is no gauge assumption hidden in this lemma's application to
smoothly converging immersions. The spectral projector onto the first
normal line is a smooth function of the fundamental matrix when its
largest eigenvalue is uniformly separated. The projector onto
$\mathcal E_j$ is the difference of the spherical normal projector
and this spectral projector. On a fixed immersion chart, project a
fixed ambient frame onto $\mathcal E_j$ and orthonormalize on a smaller
chart. The resulting coefficients of $D_j$, with every derivative,
are bounded by the unnormalized immersion and its spectral gap.
Lemma~\ref{lem:immersion-compactness} supplies such charts uniformly.
Thus bounds for the normalized fields are consequences of
Lemma~\ref{higher:vector-elliptic}, rather than assumptions on its
coefficients.

\subsection{Exclusion of compact pointed limits}
\label{higher:compact-limits-section}

\begin{proposition}\label{higher:no-compact-limit}
Fix $n\ge3$ and $q\ge2$. Let $M_j^n\to\Sph^{n+q}(1)$ be closed
connected minimal immersions with constant $Q_j>n$ and $Q_j\to n$.
No smooth complete pointed subsequential limit is compact.
\end{proposition}

\begin{proof}
Suppose a limit $F_0:M_0\to\Sph^{n+q}(1)$ is compact. By
Lemma~\ref{lem:immersion-compactness}, convergence is global through
diffeomorphisms, and the diameters are uniformly bounded. Lu's closed
first pinching theorem identifies $F_0$ as a compact covering of a
Clifford hypersurface in a fixed great $\Sph^{n+1}$; see
\cite[Theorem~6]{Lu11} and Proposition~\ref{prop:rank-one}. Let $\eta$
be its global first normal. The first normal eigenspace of the nearby
immersions is globally oriented by projecting $\eta$, so we can use
the notation above globally as bundle-valued notation.

Put $b_j=\lambda_{2,j}$ and $\beta_j=\max_{M_j}b_j$. If
$\beta_j=0$ along a subsequence, Proposition~\ref{prop:rank-one}
reduces the sequence to hypersurfaces, contradicting the hypersurface
second gap theorem \cite{PengTerng}. We may assume $\beta_j>0$ and
normalize with $\sigma_j=\sqrt{\beta_j}$. The largest eigenvalue of
the Gram matrix $(\ip{C_{\alpha,j}}{C_{\gamma,j}})$ is at most one,
and has global maximum one. In particular
\[
 |C_j|^2=\sum_{\alpha=2}^q|C_{\alpha,j}|^2\le q-1.
\]
Lemma~\ref{higher:vector-elliptic} gives uniform global bounds on all
derivatives of $C_j$ and $\theta_j$.

We next choose a global complementary normal frame whose internal
connection has the required rate. Let $P_j$ be the orthogonal
projection, in the trivial Euclidean ambient bundle, onto
$\mathcal E_j$. Differentiating $P_j$ using the Weingarten formula
gives, for any local orthonormal frame of $\mathcal E_j$,
\[
 dP_j(X)=-\sum_{\alpha=2}^q
 \bigl(dF_j(B_{\alpha,j}X)+\omega_{\alpha,j}(X)\nu_{1,j}\bigr)
                 \mathbin{\odot}\nu_{\alpha,j},
\]
where $u\mathbin{\odot}v=u\otimes v+v\otimes u$. Internal connection
terms cancel in this formula. Hence
\begin{equation}\label{higher:projection-derivative}
 \|dP_j\|_{L^\infty}\le C\sqrt{\beta_j}.
\end{equation}
Choose a point $p_j$ and an ambient constant orthonormal matrix $Z_j$
whose columns span $\mathcal E_j(p_j)$. The uniform diameter bound and
\eqref{higher:projection-derivative} imply
\[
 \sup_{M_j}|P_j-P_j(p_j)|\le C\sqrt{\beta_j}.
\]
For large $j$, the columns of
\[
 V_j=P_jZ_j(Z_j^TP_jZ_j)^{-1/2}
\]
therefore give a global orthonormal frame of $\mathcal E_j$.
Its internal connection matrix $\Sigma_j=V_j^TdV_j$ satisfies
\begin{equation}\label{higher:small-internal-connection}
 |\Sigma_j|\le C|P_j-P_j(p_j)|\,|dP_j|
                  \le C\beta_j.
\end{equation}
Indeed, $P_j(dP_j)P_j=0$, so both
$Z_j^T(dP_j)Z_j$ and $V_j^T(dP_j)Z_j$ contain a factor bounded by
$|P_j-P_j(p_j)|$. Differentiating the displayed formula for $V_j$
then proves \eqref{higher:small-internal-connection}; the matrix
inverse square root and its derivative stay uniformly bounded.

Use the columns of $V_j$ as $\nu_{\alpha,j}$. They converge, after a
subsequence, to an ambient constant orthonormal frame
$\zeta_2,\ldots,\zeta_q$ of the complement of the fixed
$\R^{n+2}$ containing the limiting immersion. The normalized fields
converge smoothly globally in this frame. Their limits
$(C_\alpha,\theta_\alpha)$ satisfy the scalar equations
\eqref{eq:normalized-constraints}--\eqref{eq:normalized-ricci}
for each $\alpha$. Define
\[
 \Omega_{\alpha,j}(X)
   =-dF_j(C_{\alpha,j}X)-\theta_{\alpha,j}(X)\nu_{1,j}.
\]
The Weingarten formula and
\eqref{higher:small-internal-connection} give
\[
 \Omega_{\alpha,j}
 =d\left(\frac{\nu_{\alpha,j}}{\sqrt{\beta_j}}\right)
   -\frac1{\sqrt{\beta_j}}
        \sum_{\gamma=2}^q(\Sigma_j)_{\gamma\alpha}\nu_{\gamma,j},
 \qquad
 \left|\frac{\Sigma_j}{\sqrt{\beta_j}}\right|
       \longrightarrow0.
\]
For every fixed closed loop in $M_0$, smooth global metric convergence bounds the lengths of its corresponding loops in $M_j$. The error in its period is therefore $O(\sqrt{\beta_j})$. Consequently, the integral of the limit $\Omega_\alpha$ along every closed loop is zero. It is an exact $\R^{n+2}$-valued one-form on
$M_0$. The compact-cover conclusion of
Lemma~\ref{lem:q2-global-potential} gives $C_\alpha=0$ for every
$\alpha$. This contradicts the normalization: smooth global
convergence preserves the maximum one of the largest Gram
eigenvalue. The contradiction proves the proposition.
\end{proof}

\subsection{A rate for the first shape operator}
\label{higher:rate-section}

Consider a sequence of closed connected minimal immersions with
constant $Q_j=n+\varepsilon_j$, where $\varepsilon_j>0$ and
$\varepsilon_j\to0$. Theorem~\ref{higher:endpoint} and
Lemma~\ref{lem:immersion-compactness} imply that every choice of
basepoints has a complete pointed Clifford limit. The usual
contradiction by a choice of basepoints therefore gives, uniformly
on $M_j$,
\[
 \lambda_{2,j}\to0,
 \qquad \lambda_{1,j}\to n,
 \qquad \nabla^k h_j\to0\quad(k\ge1).
\]
All geometric derivatives are uniformly bounded before normalization.
Proposition~\ref{higher:no-compact-limit} excludes compact limits.
A Clifford product whose two factors have dimension at least two is
simply connected and compact. For a product with a circle factor,
every nontrivial subgroup of its fundamental group $\mathbb Z$ has
finite index. Thus every noncompact connected covering of such a
product is its universal cover. It follows that every complete
pointed limit in our sequence is the full cylinder
\begin{equation}\label{higher:cylinder}
 \R\times\Sph^{n-1}\left(\sqrt{\frac{n-1}{n}}\right).
\end{equation}

Discarding the hypersurface case $\beta_j:=\max\lambda_{2,j}=0$
as above, assume $\beta_j>0$. Put $m=n-1$. The absolute value of the
cubic trace of the first shape operator tends uniformly to
$n(n-2)/\sqrt m>0$. We therefore orient the first normal line
globally by $\Tr(A_j^3)<0$. Its limit, in a product frame, is
\[
 A_0=\diag\left(-\sqrt m,\frac1{\sqrt m},\ldots,
                    \frac1{\sqrt m}\right).
\]
No orientability or triviality of $\mathcal E_j$ is needed.
Define $C_j=B_j/\sqrt{\beta_j}$ and
$\theta_j=\omega_j/\sqrt{\beta_j}$ intrinsically as
$\mathcal E_j$-valued fields. In local orthonormal frames, set
\begin{equation}\label{higher:gram-definitions}
 G_j=(\ip{C_{\alpha,j}}{C_{\gamma,j}})_{\alpha,\gamma=2}^q,
 \quad T_j=\Tr G_j,
 \quad \rho_j=\lambda_{\max}G_j.
\end{equation}
Then
\begin{equation}\label{higher:normalization-identities}
 0\le\rho_j\le1,
 \quad \max_{M_j}\rho_j=1,
 \quad 0\le T_j\le q-1,
 \quad
 a_j:=|A_j|^2=n+\varepsilon_j-\beta_j(T_j+\rho_j).
\end{equation}
The last identity is exactly $Q_j=S_j+\lambda_{2,j}$.
Although a largest eigenvalue of a general smooth matrix need not
be smooth, here $\rho_j=(Q_j-S_j)/\beta_j$ is smooth for each $j$.
We will not assume uniform higher derivative bounds for it.

The invertibility of $A_j$ is uniform, since every pointed limit
has the invertible operator $A_0$. The normal spectral gap tends
uniformly to $n$. Hence Lemma~\ref{higher:vector-elliptic}, on the
fixed-size graph charts, gives uniform bounds for every derivative
of $C_j$ and $\theta_j$. The constants are independent of the
basepoints. In particular $G_j$ is uniformly Lipschitz in local
orthonormal trivializations, and the Rayleigh quotient inequality
\[
 |\lambda_{\max}(G)-\lambda_{\max}(\widetilde G)|
              \le|G-\widetilde G|
\]
gives a uniform local Lipschitz bound for $\rho_j$. Thus
$T_j+\rho_j$ has a uniform $C^{0,\alpha}$ bound for any fixed
$0<\alpha<1$. These are the only spectral regularity estimates
needed in the following lemma.

\begin{lemma}\label{higher:first-shape-rate}
With the notation above, put
\[
 c=\frac{2-n}{\sqrt m},\qquad
 E_j=A_j^2-cA_j-I,\qquad d_j=\max_{M_j}|E_j|.
\]
Then
\begin{equation}\label{higher:first-shape-rate-conclusion}
 d_j+\varepsilon_j=O(\beta_j),
 \qquad \|\nabla A_j\|_{C^{1,\alpha}(U')}=O(\beta_j)
\end{equation}
on uniformly sized smaller immersion charts $U'$, with constants
independent of the basepoint. In particular
\[
 T_{1,j}:=\nabla A_j-\beta_j
            \sum_{\alpha=2}^q\theta_{\alpha,j}\otimes C_{\alpha,j}
                     =O(\beta_j).
\]
\end{lemma}

\begin{proof}
The arbitrary-basepoint cylinder convergence gives $d_j\to0$.
Taking a trace in the definition of $E_j$ and using
\eqref{higher:normalization-identities}, we obtain
\begin{equation}\label{higher:trace-rate}
 \Tr E_j=\varepsilon_j-\beta_j(T_j+\rho_j),
 \qquad \varepsilon_j\le\sqrt n\,d_j+q\beta_j.
\end{equation}

We recall the pointwise left inverse used for the first shape
operator. On tensors symmetric in their last two slots, define
\[
 \mathcal M_A(Z)_{ikr}
 =\left(A_{k\ell}Z_{i\ell r}
              +Z_{ik\ell}A_{\ell r}-cZ_{ikr},
            \ Z_{ikr}-Z_{kir}\right).
\]
At $A_0$, the first component can vanish only on entries whose last
two indices belong to different eigenspaces. The second component,
together with the symmetry in the last two slots, makes a tensor
in the kernel fully symmetric. Among three indices, two belong
to the same eigenspace; moving them to the last two positions
shows that each entry vanishes. Thus $\mathcal M_{A_0}$ is
injective. Compactness of the orthogonal orbit and continuity
give a smooth uniformly bounded left inverse
\[
 J_A=(\mathcal M_A^*\mathcal M_A)^{-1}\mathcal M_A^*
\]
for all $A$ sufficiently close to that orbit.

Codazzi in the first normal direction gives
\[
 (\nabla_iA_j)_{kr}-(\nabla_kA_j)_{ir}
   =\beta_j\sum_{\alpha=2}^q
       (\theta_{\alpha i,j}C_{\alpha kr,j}
                  -\theta_{\alpha k,j}C_{\alpha ir,j}).
\]
Applying $J_{A_j}$ to this identity and the derivative of $E_j$
therefore yields
\begin{equation}\label{higher:left-inverse}
 \nabla A_j=L_{A_j}(\nabla E_j)+\beta_j R_j,
\end{equation}
where $L_A(Z)=J_A(Z,0)$ and every local derivative of $R_j$ is
uniformly bounded. The contractions over normal indices are
invariant under changes of local frame.

The first normal component of the tensor Simons identity is
\begin{equation}\label{higher:top-tensor-simons}
 \Delta A_j=-\varepsilon_j A_j
        +\beta_j\bigl((T_j+\rho_j)A_j+K_j\bigr),
\end{equation}
where the exact expression for the remainder is
\begin{align*}
 K_j={}&-\sum_{\alpha=2}^q
             [C_{\alpha,j},[C_{\alpha,j},A_j]]
       +2\sum_{i,\alpha}\theta_{\alpha i,j}D_iC_{\alpha,j}\\
     &+\sum_{\alpha=2}^q
             (\operatorname{div}^{D}\theta_j)_\alpha C_{\alpha,j}
       +|\theta_j|^2A_j.
\end{align*}
All derivatives of $K_j$ are uniformly bounded. To check the
connection terms directly, the normal covariant derivatives of
$h_j$ are
\[
 T_{1,j}=\nabla A_j-\sqrt{\beta_j}
                 \sum_\alpha\omega_{\alpha,j}C_{\alpha,j},
 \qquad
 T_{\alpha,j}=\sqrt{\beta_j}D C_{\alpha,j}
                       +\omega_{\alpha,j}A_j.
\]
Taking one more normal covariant derivative in the first component
produces $-2\beta_j\theta_j\cdot DC_j$,
$-\beta_j(\operatorname{div}^{D}\theta_j)\cdot C_j$,
and $-\beta_j|\theta_j|^2A_j$ on the covariant-Laplacian side.
Moving them to the tangent-Laplacian side gives precisely the
terms displayed above. In particular there is no isolated term
of order $\sqrt{\beta_j}$.

Differentiate the polynomial defining $E_j$ twice:
\[
 \Delta E_j=A_j\Delta A_j+(\Delta A_j)A_j-c\Delta A_j
                  +2\sum_i(\nabla_iA_j)^2.
\]
Using \eqref{higher:left-inverse}, write the last term exactly as
\begin{align*}
 2(\nabla_iA_j)^2={}&
  (\nabla_iA_j)(L_{A_j}\nabla E_j)_i
 +(L_{A_j}\nabla E_j)_i(\nabla_iA_j)\\
 &+\beta_j\bigl((\nabla_iA_j)R_{j,i}
                         +R_{j,i}(\nabla_iA_j)\bigr).
\end{align*}
Consequently
\begin{equation}\label{higher:E-equation}
 \Delta E_j-\mathcal B_j(\nabla E_j)
  =-\varepsilon_j(2A_j^2-cA_j)+\beta_j\mathcal F_j.
\end{equation}
The coefficient $\mathcal B_j$ is smooth, uniformly bounded with
its derivatives, and tends to zero on all pointed charts, since
it has a factor $\nabla A_j$. The remainder $\mathcal F_j$ is
uniformly $C^{0,\alpha}$: its only factor for which higher
derivative bounds were not proved is $T_j+\rho_j$, and its uniform
Lipschitz bound was established before the lemma. Thus interior
Schauder estimates for the linear system
\eqref{higher:E-equation}, followed by
\eqref{higher:left-inverse} and \eqref{higher:trace-rate}, give
\begin{equation}\label{higher:E-interior}
 \|E_j\|_{C^{2,\alpha}(U')}
      +\|\nabla A_j\|_{C^{1,\alpha}(U')}
                     \le C(d_j+\beta_j).
\end{equation}
All coefficients here are supplied by the already bounded
unnormalized geometry and the normalized Codazzi estimates.
No estimate asserted in \eqref{higher:first-shape-rate-conclusion}
has been used to derive \eqref{higher:E-interior}.

Suppose $d_j/\beta_j\to\infty$ along a subsequence. Choose basepoints
where $|E_j|=d_j$, and divide \eqref{higher:E-equation} by $d_j$.
By \eqref{higher:trace-rate}, after a subsequence
$\varepsilon_j/d_j\to\kappa_0\ge0$. The estimate
\eqref{higher:E-interior} gives a complete pointed limit on the
full cylinder, with $U_j=E_j/d_j$ converging locally in
$C^{2,\alpha'}$, $0<\alpha'<\alpha$, to a tensor $U$ such that
\[
 |U|\le1,\quad |U(x_\infty)|=1,\quad
 \Delta U=-\kappa_0(2I+cA_0),\quad \Tr U=\kappa_0.
\]
Taking the trace of the differential equation gives
$0=-2n\kappa_0$. Therefore $\kappa_0=0$, $\Delta U=0$, and
$\Tr U=0$.

Let $P_j$ be the tangent spectral projection onto the eigenline of
$A_j$ close to its eigenvalue $-\sqrt m$. Since $E_j$ is a
polynomial in $A_j$, it has no mixed components for this splitting.
Write the tangent eigenvalues as $-\sqrt m+\delta_0$ and
$m^{-1/2}+\delta_a$, $1\le a\le m$. Their trace is zero, so
$\delta_0+\sum_a\delta_a=0$. The two roots of $t^2-ct-1$ are
simple, whence $\sum_i\delta_i^2\le C|E_j|^2$. Direct substitution
then gives
\[
 2\Tr(P_jE_j)-\Tr E_j
       =\delta_0^2-\sum_a\delta_a^2=O(|E_j|^2).
\]
After division by $d_j$ and passage to the limit, we obtain
$2\Tr(PU)-\Tr U=0$. Thus $U_{00}=0$, and $U$ is a bounded
trace-free symmetric tensor entirely in the sphere directions.

Choose a cutoff $\chi_R(t)$ equal to one on $[-R,R]$, supported
in $[-2R,2R]$, with $|\chi_R'|\le C/R$. Integration by parts in
$\Delta U=0$ gives
\[
 \int\chi_R^2|\nabla U|^2
 \le4\int|d\chi_R|^2|U|^2\le\frac CR.
\]
Letting $R\to\infty$ shows that $U$ is parallel. A parallel
self-adjoint endomorphism on a round sphere of dimension $m\ge2$
commutes with all its curvature endomorphisms and is a scalar
multiple of the identity. Since its trace is zero, $U=0$.
This contradicts $|U(x_\infty)|=1$.

It follows that $d_j=O(\beta_j)$. Now
\eqref{higher:trace-rate} gives $\varepsilon_j=O(\beta_j)$,
and \eqref{higher:E-interior} proves the remaining assertions.
\end{proof}

\subsection{The normal-rank obstruction and the second gap}
\label{higher:gap-section}

\begin{lemma}\label{higher:rank}
Let $n\ge3$ and let $H$ be a positive semidefinite symmetric
$n\times n$ matrix. For $x\in\mathbb S^{n-1}(1)$, set
\[
 h_H(x)=x^THx,\qquad
 r_H(x)=\lambda_{\max}\!
 \left(\left.(I-xx^T)H(I-xx^T)\right|_{x^\perp}\right).
\]
Suppose that, for a constant $K\in\mathbb R$,
\begin{equation}\label{higher:rank-master}
 \Delta_{\mathbb S^{n-1}}r_H
 +2(n-1)r_H+(n^2+n+2)h_H=K
\end{equation}
holds in the sense of distributions. Then $H=cI_n$ for some $c\ge0$.
In particular, if $H\ne0$, then $\operatorname{rank}H=n$.
\end{lemma}

\begin{proof}
The function $r_H$ is Lipschitz. Indeed, since $H\ge0$, its value
is also the largest eigenvalue of the full symmetric matrix
$(I-xx^T)H(I-xx^T)$; the latter matrix depends smoothly on $x$,
and the largest eigenvalue is Lipschitz with respect to the
operator norm. Equation~\eqref{higher:rank-master} is a linear
elliptic equation with smooth coefficients and smooth right-hand
side for $r_H$. Elliptic regularity therefore gives
$r_H\in C^\infty(\mathbb S^{n-1})$. Thus the pointwise maximum
principle can be applied without assuming differentiability of
the largest eigenvalue in advance.

Write the eigenvalues of $H$ as
\[
 h_1\ge h_2\ge\cdots\ge h_n\ge0
\]
and choose an orthonormal eigenbasis $e_1,\ldots,e_n$.
For every unit $x$, the two-dimensional space
$\operatorname{span}\{e_1,e_2\}$ has a nonzero intersection with
$x^\perp$. The Rayleigh quotient on this intersection, together
with the maximality of $h_1$, gives
\begin{equation}\label{higher:rank-interlacing}
 h_2\le r_H(x)\le h_1.
\end{equation}
Moreover,
\[
 h_H(e_1)=h_1,\qquad r_H(e_1)=h_2,\qquad
 h_H(e_n)=h_n,\qquad r_H(e_n)=h_1.
\]
These identities remain valid when eigenvalues have multiplicity.
Consequently $e_1$ is a global minimum point of $r_H$ and $e_n$
is a global maximum point. Evaluating
\eqref{higher:rank-master} at these points yields
\[
 K\ge2(n-1)h_2+(n^2+n+2)h_1,
 \qquad
 K\le2(n-1)h_1+(n^2+n+2)h_n.
\]
It follows that
\[
 (n^2+n+2)(h_1-h_n)
 \le2(n-1)(h_1-h_2)
 \le2(n-1)(h_1-h_n).
\]
Since $n^2+n+2-2(n-1)=n^2-n+4>0$, we obtain
$h_1=h_n$. Thus all eigenvalues of $H$ are equal, which proves
the assertion.
\end{proof}

\begin{remark}\label{higher:rank-quantitative}
The same two-point argument has an explicit stability estimate.
Suppose that $r_H\in C^2(\mathbb S^{n-1})$ and that, for constants
$K\in\mathbb R$ and $e\ge0$,
\begin{equation}\label{higher:rank-residual}
 \left\|
 \Delta_{\mathbb S^{n-1}}r_H+2(n-1)r_H
 +(n^2+n+2)h_H-K
 \right\|_{L^\infty(\mathbb S^{n-1})}\le e.
\end{equation}
At the minimum point $e_1$ and the maximum point $e_n$ used in
the proof, the corresponding inequalities now have errors at
most $e$. Subtracting them gives
\begin{equation}\label{higher:rank-stability}
 (n^2-n+4)(h_1-h_n)\le2e.
\end{equation}
If, in addition, $\operatorname{rank}H\le n-1$ and
\[
 \max_{x\in\mathbb S^{n-1}}
 \frac{2n}{n-1}r_H(x)=1,
\]
then $h_n=0$ and $h_1=(n-1)/(2n)$ by
\eqref{higher:rank-interlacing}. Hence
\begin{equation}\label{higher:rank-margin}
 e\ge\frac{(n^2-n+4)(n-1)}{4n}\ge\frac53.
\end{equation}
In particular, a residual smaller than $1$ is impossible under
these rank and normalization assumptions. The $C^2$ hypothesis
in this quantitative statement is explicit; it is not inferred
merely from the fact that $r_H$ is a largest eigenvalue.
\end{remark}

\begin{proof}[Proof of Theorem~\ref{thm:gap-low-codimension}]
If the assertion fails, choose a sequence with constant
$Q_j=n+\varepsilon_j>n$ and $\varepsilon_j\to0$, passing to connected
components if necessary. Apply the preceding
arguments. In particular $\beta_j>0$, every pointed limit is the
full Clifford cylinder, and Lemma~\ref{higher:first-shape-rate}
applies. Choose basepoints with $\rho_j=1$. On an exhaustion of the limiting cylinder, project the same fixed ambient orthonormal basis of its extra normal space onto $\mathcal E_j$ and orthonormalize. Using this single ambient basis on all exhaustion regions makes the limiting frames compatible. Use the same
fixed ambient frame on every member of the exhaustion, so these
local constructions agree on overlaps. These frames
converge smoothly to the constant frame, so $D_j$ converges to the
trivial connection. Lemma~\ref{higher:vector-elliptic} and a diagonal
subsequence give smooth limits $C_\alpha,\theta_\alpha$ on the
entire cylinder. They satisfy the scalar normalized equations
\eqref{eq:normalized-constraints}--\eqref{eq:normalized-ricci}
separately for $\alpha=2,\ldots,q$. The limiting Gram matrix $G$
satisfies
\[
 \lambda_{\max}G\le1,
 \qquad \lambda_{\max}G(x_\infty)=1.
\]

The cylinder is simply connected because $m=n-1\ge2$.
Lemma~\ref{lem:q2-global-potential} therefore gives a global
Jacobi potential $f_\alpha$ for every component. The global bound
$\sum_\alpha|C_\alpha|^2\le q-1$ supplies the transverse $L^2$
bound in Lemma~\ref{thm:jacobi}. Consequently, for constants
$a_\alpha,b_\alpha\in\R$ and $v_\alpha,w_\alpha\in\R^n$,
\[
 f_\alpha(t,x)=a_\alpha\cos(\sqrt n\,t)
        +b_\alpha\sin(\sqrt n\,t)+(v_\alpha+tw_\alpha)\cdot x.
\]
The terms without $tw_\alpha$ have zero projected shape tensor.
Using the product metric
$g_0=dt^2+(m/n)g_{\Sph^{n-1}(1)}$, the remaining tensors are
\begin{equation}\label{higher:limit-fields}
 (C_\alpha)_{00}=0,
 \quad (C_\alpha)_{0a}=\nabla_a(w_\alpha\cdot x),
 \quad (C_\alpha)_{ab}=0,
 \quad
 \theta_\alpha=-\frac n{\sqrt m}(w_\alpha\cdot x)\,dt.
\end{equation}
In particular the normalized limiting tensors are independent
of $t$. At least one $w_\alpha$ is nonzero by the normalization.

Form the $n\times(q-1)$ matrix $W$ with columns $w_\alpha$ and set
\[
 H=WW^T=\sum_{\alpha=2}^q w_\alpha w_\alpha^T,
 \qquad h_H(x)=x^THx,
 \qquad t_H=\Tr H.
\]
Then $H\ge0$, $H\ne0$, and $\rank H\le q-1$. Formula
\eqref{higher:limit-fields} gives
\begin{equation}\label{higher:limit-gram}
 G(x)=\frac{2n}{m}W^T(I-xx^T)W,
 \quad
 T:=\Tr G=\frac{2n}{m}(t_H-h_H),
 \quad
 V:=\sum_\alpha|\theta_\alpha|^2=\frac{n^2}{m}h_H.
\end{equation}
The nonzero eigenvalues of $W^T(I-xx^T)W$ coincide, with
multiplicity, with those of
$(I-xx^T)H(I-xx^T)|_{x^\perp}$. Thus, with the notation of
Lemma~\ref{higher:rank},
\begin{equation}\label{higher:rho-compression}
 \rho:=\lambda_{\max}G=\frac{2n}{m}r_H.
\end{equation}
The formula also holds when this largest eigenvalue is zero.

We derive the equation forced by constancy of $Q$ before applying
the rank obstruction. After a subsequence,
\[
 \frac{\varepsilon_j}{\beta_j}\longrightarrow\kappa
\]
by Lemma~\ref{higher:first-shape-rate}. The exact Laplacian formula
for the simple largest normal eigenvalue, written without
diagonalizing the lower block, is
\begin{equation}\label{higher:top-eigenvalue-simons}
 \frac12\Delta a_j
  =|T_{1,j}|^2+a_j(n-a_j)
    -\beta_j\sum_{\alpha=2}^q|[A_j,C_{\alpha,j}]|^2
    +\beta_j\sum_i
       \ip{(a_jI-\beta_jG_j)\theta_{j,i}}{\theta_{j,i}}.
\end{equation}
For example, it follows by choosing an eigenframe of the lower
block at the point in question and applying the simple-eigenvalue
second derivative formula. The off-diagonal derivative of the
fundamental matrix is $(a_jI-\beta_jG_j)\omega_j$; the resulting
eigenvector correction is the last term above. This also shows
that the expression is unchanged at lower-eigenvalue crossings.

Use $a_j=n+\varepsilon_j-\beta_j(T_j+\rho_j)$ and divide
\eqref{higher:top-eigenvalue-simons} by $\beta_j$. The first term
tends to zero, since $|T_{1,j}|^2/\beta_j=O(\beta_j)$.
The difference between the two tangent eigenvalues of $A_0$ is
$-n/\sqrt m$, and $C_\alpha$ has only mixed components, so
\[
 \sum_\alpha|[A_0,C_\alpha]|^2=\frac{n^2}{m}T.
\]
All terms on the right therefore have locally uniform limits.
On the left, $T_j\to T$ smoothly and $\rho_j\to\rho$ uniformly,
the latter by the eigenvalue Lipschitz inequality. Testing
against a compactly supported smooth function and integrating
by parts transfers the Laplacian to the test function. Smooth
metric convergence then passes the left-hand side to
$-\frac12\Delta_{g_0}(T+\rho)$ in the sense of distributions.
No uniform second derivative bound for $\rho_j$ is used.
We obtain
\[
 -\frac12\Delta_{g_0}(T+\rho)
      =n(T+\rho-\kappa)-\frac{n^2}{m}T+nV,
\]
or equivalently
\begin{equation}\label{higher:constant-Q-master}
 \kappa=\frac{\Delta_{g_0}(T+\rho)}{2n}
                    -\frac{T}{m}+\rho+V.
\end{equation}
In particular, both the trace $T$ of the lower Gram matrix and
its largest eigenvalue $\rho$ enter the equation.

All the quantities in \eqref{higher:constant-Q-master} are
independent of $t$, and
$\Delta_{g_0}=(n/m)\Delta_{\Sph^{n-1}(1)}$ on such functions.
Substitute \eqref{higher:limit-gram} and
\eqref{higher:rho-compression}, and use
\[
 \Delta_{\Sph^{n-1}(1)}h_H=2t_H-2n h_H.
\]
After multiplying by $m^2/n$, the result is
\begin{equation}\label{higher:compression-equation}
 \Delta_{\Sph^{n-1}(1)}r_H
       +2m r_H+(n^2+n+2)h_H
                   =4t_H+\frac{m^2}{n}\kappa.
\end{equation}
The equation holds distributionally. Lemma~\ref{higher:rank}
first makes $r_H$ smooth by elliptic regularity and then gives
$H=c_0I_n$. Since $H\ne0$, $c_0>0$ and
\[
 n=\rank H\le q-1.
\]
This contradicts $q\le n$. Thus the sequence cannot exist, and
the asserted positive gap follows.
\end{proof}

\begin{remark}\label{higher:threshold-sharp}
The rank calculation identifies the codimension threshold.
A nonzero normalized cylindrical limit requires at least $n$
extra normal directions, hence $q\ge n+1$. In the isotropic case
$H=c_0I_n$, the normalization $\max\rho=1$ gives
$c_0=(n-1)/(2n)$, and
\eqref{higher:compression-equation} gives $\kappa=n/2$.
Thus the obstruction is consistent with the examples of
Theorem~\ref{thm:explicit-sequence}, which show failure of the
second gap in every codimension $q\ge n+1$. The argument above proves existence of $\gamma_{n,q}$; the explicit
codimension-two value is established in
Section~\ref{sec:explicit-q2}.
\end{remark}

\section{An explicit gap in codimension two}\label{sec:explicit-q2}

Throughout this section, the immersion is closed and connected,
$n\ge3$, $q=2$, and $Q=n+\varepsilon$ is constant with
$\varepsilon\ge0$. We use $a=\lambda_1$, $b=\lambda_2$,
$\delta=a-b$, and $\beta=\max_M b$.
The proof has three parts: an explicit bound for $\beta$ in terms
of $\varepsilon$, entry into a prescribed Clifford neighborhood,
and rigidity in that neighborhood. All constants are deliberately
rounded in the direction that weakens an estimate.

\subsection{An algebraic spectral estimate}

\begin{lemma}\label{lem:joint-spectral-bound}
For $n \ge 3$, let $A$ be a trace-free symmetric endomorphism of an $n$-dimensional Euclidean space,
\begin{equation}\label{eq:joint-spectral-bound}
\|A\|_{\mathrm{op}}^2+
\frac{\operatorname{osc}(A)^2}{2n}
\le\frac{n+\sqrt{n^2-2n+4}}{2n}|A|^2,
\end{equation}
where $\operatorname{osc}(A)$ is the difference between the largest and smallest eigenvalues of $A$.
\end{lemma}
\begin{proof}
Let $\lambda=(\lambda_1,\ldots,\lambda_n)$ be the eigenvalue vector of $A$, so $\lambda\perp\mathbf1=(1,\ldots,1)$. Let $e_1,\ldots,e_n$ be the standard basis of $\R^n$.
For indices $i,j,k$ with $j\ne k$, put
$v=e_i-\frac{\mathbf1}{n}$ and $d=e_j-e_k$.
Then a direct computation yields 
\[|v|^2=\frac{n-1}{n}, \quad |d|^2=2, \quad |\langle v,d\rangle|\le1.\]
Consider the symmetric operator \(\hat{S} = v \otimes v + \frac{1}{2n} d \otimes d\), which acts as \(\hat{S} z = \langle v, z \rangle v + \frac{1}{2n} \langle d, z \rangle d\). Thus, we have
\[
\lambda_i^2 + \frac{1}{2n}(\lambda_j - \lambda_k)^2 = \langle \lambda, \hat{S} \lambda \rangle \le \|\hat{S}\|_{op} |\lambda|^2 = \|\hat{S}\|_{op} |A|^2.
\]
 Since \(\hat{S}\) acts nontrivially only on the subspace spanned by \(v\) and \(d\), its matrix representation in this basis is:
\[
\begin{pmatrix} |v|^2 & \langle v, d \rangle \\ \frac{1}{2n}\langle v, d \rangle & \frac{1}{2n}|d|^2 \end{pmatrix} = \begin{pmatrix} \frac{n-1}{n} & \langle v, d \rangle \\ \frac{1}{2n}\langle v, d \rangle & \frac{1}{n} \end{pmatrix},
\]
where \(|\langle v, d \rangle|\le 1\). Hence, the largest eigenvalue of
$\hat{S}$ is at most
\[
\frac12\left(1+
\sqrt{\left(\frac{n-2}{n}\right)^2+\frac2n}\right)
=\frac{n+\sqrt{n^2-2n+4}}{2n}.
\] 
Apply this bound to
$\lambda_i^2+(2n)^{-1}(\lambda_j-\lambda_k)^2$
and maximize over $i,j,k$ yields the desired result \eqref{eq:joint-spectral-bound}.
\end{proof}
\subsection{An explicit global estimate for the second normal eigenvalue}
\label{quant:global-entry-section}

We first obtain a numerical estimate which does not assume an a priori
Clifford neighbourhood.  Throughout this subsection, the immersion is
closed and connected, has codimension two, and satisfies
\[
 Q=a+2b=n+\varepsilon,\qquad
 a=\lambda _1,\quad b=\lambda _2,\quad
 S=a+b=Q-b,\quad \beta=\max_M b.
\]
In particular, \(a\) and \(b\) are smooth even at a repeated normal
eigenvalue.  The following constants will be used:
\begin{equation}\label{quant:entry-constants}
 \begin{split}
 \eta _0&=10^{-6},\\
 \Theta _n&=\min\left\{
 \frac{(\eta _0/3)^{n+2}}
 {2^{11n+20}n^{3n+4}},\ 
 \frac{1}{10000n^2}\right\},\\
 B_n&=\exp(5000n^{3/2}).
 \end{split}
\end{equation}

\begin{proposition}\label{quant:global-entry}
Let \(n\ge3\).  If \(0\le\varepsilon\le\Theta _n\), then
\begin{equation}\label{quant:entry-beta-rate}
 \max_M\lambda _2=\beta\le B_n\varepsilon.
\end{equation}
No bound for the volume, the injectivity radius, the diameter, or the
degree of the immersion is assumed.
\end{proposition}

We give the auxiliary estimates and the local radius argument before
proving the proposition.  For a positive function \(u\), the curvature
condition in the next lemma is the special case of the conformal
method of \cite{CatinoRoncoroni24} that is needed here.  Its local
conclusion is included explicitly.

\begin{lemma}[Local conformal radius]\label{quant:local-conformal-radius}
Let \((M^n,g)\) be closed and connected, and let \(\Omega\) be a
nonempty open subset, possibly equal to \(M\).  Suppose that \(u\) is
smooth and positive on a neighbourhood of \(\overline\Omega\), and that
\begin{equation}\label{quant:local-weighted-ricci}
 \operatorname{Ric}-k\frac{\Delta u}{u}g\ge c_0g
 \quad\hbox{on }\Omega,
 \qquad
 a_k:=k-\frac{n-1}{4}k^2>0,
 \qquad c_0>0.
\end{equation}
Put
\begin{equation}\label{quant:local-radius-value}
 R(k,c_0,n)=\pi\left[
 \frac{(n-1)+(n-3)^2k^2/(4a_k)}{c_0}
 \right]^{1/2}.
\end{equation}
If \(\Omega\ne M\), every point of \(\Omega\) can be joined to
\(\partial\Omega\) by a curve of \(g\)-length at most
\(R(k,c_0,n)\) whose interior is in \(\Omega\).  If \(\Omega=M\),
then \(\operatorname{diam}(M,g)\le R(k,c_0,n)\).
In particular, for \(k=1/(2n)\) and \(c_0=1/12\), both conclusions
hold with the strict upper bound \(12\sqrt n\).
\end{lemma}

\begin{proof}
Write \(f=\log u\) and \(\widetilde g=e^{2kf}g\).  In the proper-domain
case, extend \(f\) smoothly from a neighbourhood of
\(\overline\Omega\) to \(M\).  The resulting metric is positive and
complete because \(M\) is compact.  A shortest \(\widetilde g\)-curve
from a fixed point of \(\Omega\) to the closed set
\(\partial\Omega\) exists.  Its interior lies in \(\Omega\), since
otherwise its first intersection with the boundary would shorten it.
Thus no regularity of the level set \(\partial\Omega\) is required.
If \(\Omega=M\), take instead a shortest \(\widetilde g\)-curve
between any two given points.  Denote either curve by \(\gamma\),
parametrize it by \(g\)-arclength \(s\in[0,\ell]\), and put
\(T=\gamma'(s)\).

The conformal geodesic equation and the conformal Ricci formula give
\[
 \nabla_TT=k(\nabla f-f'T),\qquad
 \operatorname{Hess}f(T,T)=f''-k(|\nabla f|^2-(f')^2),
\]
and, using \eqref{quant:local-weighted-ricci},
\begin{equation}\label{quant:conformal-ricci-along-curve}
 \operatorname{Ric}_{\widetilde g}(T,T)
 \ge c_0-(n-2)k f''+k|\nabla f|^2.
\end{equation}
Choose a \(\widetilde g\)-parallel orthonormal frame normal to the
curve.  Sum the index forms of the normal variation fields with
coefficient \(\phi=e^{kf/2}\psi\), where
\(\psi(0)=\psi(\ell)=0\).  Since
\(d\widetilde s=e^{kf}ds\), nonnegativity of the index form implies
\[
 0\le\int_0^\ell e^{-kf}
 \left((n-1)(\phi')^2
       -\operatorname{Ric}_{\widetilde g}(T,T)\phi^2\right)ds.
\]
Substitute \eqref{quant:conformal-ricci-along-curve}, integrate the
\(f''\)-term by parts, and use \(|\nabla f|^2\ge(f')^2\).  This gives
\begin{align*}
 0&\le\int_0^\ell
 \left((n-1)(\psi')^2-c_0\psi^2
       -(n-3)k f'\psi\psi'-a_k(f'\psi)^2\right)ds\\
 &\le\int_0^\ell
 \left(\left(n-1+\frac{(n-3)^2k^2}{4a_k}\right)(\psi')^2
             -c_0\psi^2\right)ds.
\end{align*}
The last inequality is completion of a square, with positive
coefficient \(a_k\).  The choice
\(\psi(s)=\sin(\pi s/\ell)\) proves
\(\ell\le R(k,c_0,n)\).  The original \(g\)-distance is no larger
than this \(g\)-length, proving both assertions.

For the stated numerical specialization,
\[
 a_k=\frac{7n+1}{16n^2},\qquad
 R(k,c_0,n)^2
 =12\pi^2\left(n-1+\frac{(n-3)^2}{7n+1}\right)<144n.
\]
Indeed, \((n-3)^2/(7n+1)\le n/7\) for \(n\ge3\), and
\(\pi^2<10\).
\end{proof}

\begin{lemma}[Global logarithmic gradient]\label{quant:entry-log-gradient}
If \(0<\varepsilon\le1\), then
\begin{equation}\label{quant:entry-coarse-gradient}
 |\nabla h|^2\le200n^3,\qquad
 |\nabla b|\le34n^2,
\end{equation}
and
\begin{equation}\label{quant:entry-log-bound}
 |\nabla\log(b+\varepsilon)|\le400n.
\end{equation}
\end{lemma}

\begin{proof}
We record explicitly the coarse differentiated Simons estimate used
here.  Set \(G=|\nabla h|^2\) and \(H=|\nabla^2h|^2\).
Since \(S\le Q\le n+1\), the tensor and scalar Simons formulas give
\begin{equation}\label{quant:entry-coarse-bochner}
 \Delta G\ge2H-100n^2G,\qquad
 \Delta S\ge2G-6S^2.
\end{equation}
For completeness, these numerical coefficients follow directly from
the tensor formulas as follows.  The derivatives of the Gram cubic
and double-commutator cubic in \eqref{eq:simons-tensor} have norm at
most \(3S\sqrt G\) and \(12S\sqrt G\), respectively.  Hence
\[
 |\nabla\Delta h|\le(n+15S)\sqrt G\le21n\sqrt G.
\]
The Gauss and Ricci equations give
\[
 |\operatorname{Rm}|\le5n,\quad
 |R^\perp|\le3n,\quad
 |\nabla\operatorname{Rm}|, |\nabla R^\perp|
       \le4\sqrt S\sqrt G,\quad
 |\operatorname{Ric}|_{\mathrm{op}}\le2n.
\]
The bundle containing \(h\) has two covariant tangent slots and one
normal slot.  Thus its curvature action has norm at most
\(2|\operatorname{Rm}|+|R^\perp|\le13n\).  In the commutator
\([\Delta,\nabla]h\), the undifferentiated curvature terms are
therefore bounded by \((2n+26n)\sqrt G\); the divergence of the
bundle curvature contributes at most \(12\sqrt n\,S\sqrt G\).
Consequently
\[
 \tfrac12\Delta G
 \ge H-\bigl(49n+12\sqrt n(n+1)\bigr)G
 \ge H-50n^2G.
\]
The second inequality in \eqref{quant:entry-coarse-bochner} follows
because the squared Gram sum is at most \(S^2\), and the squared
commutator sum is at most \(2S^2\).

At a maximum point of \(G+100n^2S\),
\[
 0\ge2H+100n^2G-600n^2S^2,
\]
so \(G\le6(n+1)^2\) at that point.  At every point of \(M\) it
follows that
\[
 G\le6(n+1)^2+100n^2(n+1)\le200n^3.
\]
As \(b=Q-S\), this implies
\[
 |\nabla b|=|\nabla S|\le2\sqrt S\sqrt G
       \le2\sqrt{n+1}\sqrt{200n^3}\le34n^2.
\]

It remains to use the improved estimate where \(b\) is small.  On the
set \(b\le a/8\), which is equivalently \(b\le Q/10\), the normal
eigenvalues are distinct.  Use a local normal eigenframe as in
\eqref{nablaAB}, and put
\[
 \delta=a-b,\qquad
 \mathscr D=2ab-|[A,B]|^2\ge0,\qquad
 U=|T_1|^2+2|T_2|^2.
\]
Differentiating \(\langle A,B\rangle=0\) gives
\begin{equation}\label{quant:entry-connection-identity}
 \delta\omega_i
 =\langle B,T_{1;i}\rangle+\langle A,T_{2;i}\rangle.
\end{equation}
Moreover, the two normal components of the eigenvalue Simons identity,
combined with \(\Delta a+2\Delta b=0\), give the exact identity
\begin{equation}\label{quant:entry-weighted-identity}
 U-\delta|\omega|^2
 =Q\varepsilon+2\delta b-3\mathscr D.
\end{equation}
Weighted Cauchy--Schwarz in
\eqref{quant:entry-connection-identity} yields
\[
 \delta|\omega|^2
 \le\frac{b+a/2}{a-b}U\le\frac57U.
\]
It follows from \eqref{quant:entry-weighted-identity} that
\[
 U\le7Q(b+\varepsilon),\qquad
 |\nabla b|^2
 =4\sum_i\langle B,T_{2;i}\rangle^2
 \le2bU\le14Qb(b+\varepsilon).
\]
Hence \(|\nabla\log(b+\varepsilon)|\le\sqrt{14Q}\) on this set.
On the complementary set \(b\ge Q/10\),
\eqref{quant:entry-coarse-gradient} gives
\[
 |\nabla\log(b+\varepsilon)|
 \le\frac{34n^2}{Q/10}\le340n.
\]
These two estimates prove \eqref{quant:entry-log-bound} everywhere.
The second argument uses only the smooth scalar function \(b\), so
it also covers points where the two normal eigenvalues coincide.
\end{proof}

\begin{proof}[Proof of Proposition~\ref{quant:global-entry}]
If \(\varepsilon=0\), Lu's first pinching theorem
\cite[Theorem~6]{Lu11}, with \(n\ge3\), gives \(b\equiv0\).
We therefore assume \(0<\varepsilon\le\Theta_n\).

\smallskip
\noindent\emph{Step 1: a uniformly controlled distance to the small-
\(b\) region.}
The Rayleigh-quotient version of
\eqref{eq:q2-onlyQ-subharmonic}, together with
\(\Delta a=-2\Delta b\), gives, at every point,
\begin{equation}\label{quant:entry-rayleigh}
 -\Delta b\ge2ab-|[A,B]|^2-a\varepsilon,
 \qquad \Delta b\le a\varepsilon.
\end{equation}
These inequalities remain true when \(a=b\).  On \(b>0\), put
\(V=-\Delta b/b\).  Then
\[
 V\ge2a-\operatorname{osc}(A)^2-\frac{a\varepsilon}{b}.
\]
Let
\[
 r_n=\sqrt{n^2-2n+4},\qquad
 \alpha_n=\frac{n+r_n}{2n},\qquad
 c_n=\frac{n-r_n}{2}.
\]
The joint spectral estimate of
Lemma~\ref{lem:joint-spectral-bound}, the Gauss equation, and
\(\|B\|_{\mathrm{op}}^2\le(n-1)b/n\) imply
\begin{equation}\label{quant:entry-weighted-ricci}
 \begin{split}
 \operatorname{Ric}+\frac{V}{2n}g
 &\ge\left(n-1-\left(\alpha_n-\frac1n\right)a
                 -\frac{n-1}{n}b
                 -\frac{a\varepsilon}{2nb}\right)g\\
 &=\left(c_n+\frac{r_n-1}{n}b
          -\left(\alpha_n-\frac1n\right)\varepsilon
          -\frac{a\varepsilon}{2nb}\right)g.
 \end{split}
\end{equation}
Here \(a=n+\varepsilon-2b\).  We have \(c_n\ge1/6\): indeed,
\((n-1/3)^2-r_n^2=(12n-35)/9>0\) for \(n\ge3\).
Also \(0<\alpha_n-1/n\le1\).

Consider the open set
\[
 \Omega=\{x\in M:b(x)>24\varepsilon\}.
\]
Since \(\varepsilon\le1/24\) and \(a\le n+1\),
\eqref{quant:entry-weighted-ricci} gives on \(\Omega\)
\begin{equation}\label{quant:entry-ricci-twelfth}
 \operatorname{Ric}+\frac{V}{2n}g
 \ge\left(\frac16-\frac1{24}-\frac1{36}\right)g
 =\frac7{72}g\ge\frac1{12}g.
\end{equation}
In particular, Lemma~\ref{quant:local-conformal-radius} applies with
\(u=b\), \(k=1/(2n)\), and \(c_0=1/12\).  Its positivity
hypothesis up to the boundary holds because
\(b\ge24\varepsilon>0\) on \(\overline\Omega\).

If \(\Omega\) is empty, \(\beta\le24\varepsilon\) already proves
the proposition.  If \(\Omega\) is proper and nonempty, start at a
maximum point of \(b\), and use the curve in
Lemma~\ref{quant:local-conformal-radius} to a point where
\(b=24\varepsilon\).  Its length is less than \(12\sqrt n\).
Integration of \eqref{quant:entry-log-bound} along this curve yields
\begin{equation}\label{quant:entry-proper-domain}
 \beta+\varepsilon
 \le25\exp(4800n^{3/2})\varepsilon
 \le B_n\varepsilon.
\end{equation}
The last inequality uses \(25<\exp(200n^{3/2})\).

\smallskip
\noindent\emph{Step 2: the everywhere-positive region has a normal
spectral gap.}
It remains to consider \(\Omega=M\).  The same local conformal
radius lemma now gives
\begin{equation}\label{quant:entry-diameter}
 \operatorname{diam}(M,g)<12\sqrt n=:D_0.
\end{equation}
Also \(S=Q-b<n-23\varepsilon\).  The trace-free operator norm bound
in the Gauss equation gives
\[
 \operatorname{Ric}
 \ge\left(n-1-\frac{n-1}{n}S\right)g\ge0.
\]

We will use the following elementary eigenvalue consequence, including
its constant:
\begin{equation}\label{quant:entry-scalar-poincare}
 \lambda_1(-\Delta)\ge\frac1{4D_0^2}
 =\frac1{576n}\ge\frac1{10000n^2}.
\end{equation}
To see this, let \(-\Delta v=\lambda_1v\), normalize
\(\max|v|=1\), and set \(J=|\nabla v|^2\).
The Bochner formula with nonnegative Ricci curvature gives
\(\Delta J\ge-2\lambda_1J\).  At a maximum of
\(J/(2-v^2)\), the gradient terms in the quotient formula cancel,
so
\[
 0\ge(2-v^2)\Delta J+J\Delta(v^2)
 \ge2J^2-4\lambda_1J.
\]
Thus \(|\nabla v|^2\le4\lambda_1\) everywhere.  An eigenfunction
for \(\lambda_1>0\) has mean zero and hence a zero.  Integrating
this gradient bound from a zero to a point where \(|v|=1\), along a
curve of length at most \(D_0\), proves
\eqref{quant:entry-scalar-poincare}.

Multiply the second inequality of \eqref{quant:entry-rayleigh} by
the nonnegative function \(\beta-b\).  Integration by parts gives
\begin{equation}\label{quant:entry-b-energy}
 \int_M|\nabla b|^2
 =\int_M(\beta-b)\Delta b
 \le\varepsilon\int_M(\beta-b)a
 \le Q^2\varepsilon\operatorname{Vol}(M).
\end{equation}
Write \(\bar b=\operatorname{Vol}(M)^{-1}\int_Mb\).
Equations \eqref{quant:entry-scalar-poincare} and
\eqref{quant:entry-b-energy}, with \(Q^2\le2n^2\), give
\begin{equation}\label{quant:entry-b-variance}
 \frac1{\operatorname{Vol}(M)}\int_M(b-\bar b)^2
 \le20000n^4\varepsilon.
\end{equation}

Put \(\sigma=\max b-\min b\).  If \(\sigma>0\), choose a point
where \(|b-\bar b|\ge\sigma/2\).  The gradient estimate
\eqref{quant:entry-coarse-gradient} ensures
\(|b-\bar b|\ge\sigma/4\) on the ball of radius
\(\rho=\sigma/(136n^2)\) about that point.  We have
\(\rho<D_0\).  Bishop--Gromov volume comparison for
\(\operatorname{Ric}\ge0\), applied up to radius \(D_0\), gives
\[
 \frac{\operatorname{Vol}(B_\rho)}{\operatorname{Vol}(M)}
 \ge\left(\frac{\rho}{D_0}\right)^n.
\]
Together with \eqref{quant:entry-b-variance}, this yields
\begin{equation}\label{quant:entry-b-oscillation}
 \begin{split}
 \sigma^{n+2}
 &\le320000(1632)^n n^{(5/2)n+4}\varepsilon\\
 &\le2^{11n+20}n^{3n+4}\varepsilon.
 \end{split}
\end{equation}
The same conclusion is immediate if \(\sigma=0\).  The choice of
\(\Theta_n\) therefore implies
\begin{equation}\label{quant:entry-delta-oscillation}
 \operatorname{osc}_M b\le\frac{\eta_0}{3},\qquad
 \operatorname{osc}_M\delta\le\eta_0,
 \qquad \delta=a-b=Q-3b.
\end{equation}

The immersion is not totally geodesic because \(Q=n+\varepsilon>0\).
The explicit first curvature gap of
\cite[Theorem~1.1]{LiZhaoGap} states, for \(n\ge3\), that
\[
 S_{\max}>
 \frac{2n}{3}+
 \frac{31-9\sqrt6}{75}(n-2)
 >\frac{2n}{3}+\eta_0.
\]
Since \(S=(2Q+\delta)/3\), this gives
\(\delta_{\max}>3\eta_0-2\varepsilon\).
Our threshold satisfies \(\varepsilon\le\Theta_n\le\eta_0/2\).
Combining with \eqref{quant:entry-delta-oscillation}, we conclude
\begin{equation}\label{quant:entry-positive-delta}
 \delta\ge\eta_0\quad\hbox{on }M.
\end{equation}

\smallskip
\noindent\emph{Step 3: a Hodge identity controls the minimum of \(b\).}
The normal eigenlines are now smooth and separated everywhere.
If necessary, pass to a connected finite cover which orients both
lines, and choose a global orthonormal eigenframe there.  All scalar
extrema and all preceding pointwise inequalities are unchanged.  The
conformal radius lemma applies on this closed cover as well, so that
\eqref{quant:entry-diameter} remains valid on it.  We use the same
notation for the cover.  Alternatively, the following integrations
may be performed on the original manifold with the associated flat
real line bundle of signs.

Let \(\omega\), \(T_1\), and \(T_2\) be defined by
\eqref{nablaAB}.  The Ricci equation and the Codazzi symmetries give
the exact one-form identity
\begin{equation}\label{quant:entry-hodge-identity}
 d^*d\omega=\delta\omega-2\langle B,T_1\rangle,
 \qquad
 (\langle B,T_1\rangle)_j=\langle B,T_{1;j}\rangle.
\end{equation}
Here is an index verification, including the normalization of
two-form norms.  For \(K=[A,B]\), our curvature convention gives
\((d\omega)_{ij}=K_{ji}=-K_{ij}\), and hence
\[
 (d^*d\omega)_j
 =\sum_i\nabla_iK_{ij}
 =\langle A,T_{2;j}\rangle-\langle B,T_{1;j}\rangle.
\]
In this computation, the terms from
\(\operatorname{div}A=B\omega\) and
\(\operatorname{div}B=-A\omega\) cancel the terms from
\(\nabla A=T_1+\omega\otimes B\) and
\(\nabla B=T_2-\omega\otimes A\).
Equation \eqref{quant:entry-connection-identity} now gives
\eqref{quant:entry-hodge-identity}.  With the usual two-form norm,
\(|d\omega|^2=|K|^2/2\); integrating therefore gives
\begin{equation}\label{quant:entry-integrated-hodge}
 \int_M|[A,B]|^2
 =2\int_M\delta|\omega|^2
 -4\int_M\omega_i\langle B,T_{1;i}\rangle.
\end{equation}

On the other hand, integration of
\eqref{eq:eigenvalue-laplacian-general} gives
\begin{equation}\label{quant:entry-integrated-energy}
 \int_M\bigl(|T_1|^2+\delta|\omega|^2+\mathscr D\bigr)
 =\varepsilon\int_Ma,
 \qquad \mathscr D=2ab-|[A,B]|^2\ge0.
\end{equation}
Set
\[
 X=\int_M|T_1|^2,\qquad
 Y=\int_M\delta|\omega|^2,\qquad
 Z=\int_M\mathscr D,
 \qquad s=\sqrt{\frac{\beta}{\eta_0}}.
\]
By \eqref{quant:entry-positive-delta} and Cauchy--Schwarz,
\[
 \left|\int_M\omega_i\langle B,T_{1;i}\rangle\right|
 \le s\sqrt{XY}.
\]
Consequently \eqref{quant:entry-integrated-hodge} implies
\[
 2\int_Mab\le Z+2Y+4s\sqrt{XY}.
\]
The largest eigenvalue of
\(\left(\begin{smallmatrix}0&2s\\2s&2\end{smallmatrix}\right)\)
is \(1+\sqrt{1+4s^2}\), which is at least \(2\).
Applying this two-dimensional quadratic-form estimate to
\((\sqrt X,\sqrt Y)\), and then using
\eqref{quant:entry-integrated-energy}, gives
\begin{equation}\label{quant:entry-minimum-b}
 \begin{split}
 2\int_Mab
 &\le\left(1+\sqrt{1+4\beta/\eta_0}\right)
       \varepsilon\int_Ma,\\
 \min_M b
 &\le\frac12\left(1+\sqrt{1+4\beta/\eta_0}\right)
       \varepsilon.
 \end{split}
\end{equation}
The last division is valid because \(a\ge Q/3>0\).

Finally, integrate \eqref{quant:entry-log-bound} along a curve
between a minimum point and a maximum point of \(b\), using
\eqref{quant:entry-diameter}.  Since
\(\beta\le Q/3\le n\), \(\eta_0=10^{-6}\), and
\[
 \frac32+\frac12\sqrt{1+4\beta/\eta_0}
 \le1002\sqrt n\le\exp(200n^{3/2}),
\]
we obtain
\[
 \begin{split}
 \beta+\varepsilon
 &\le\exp(4800n^{3/2})(\min_Mb+\varepsilon)\\
 &\le\exp(5000n^{3/2})\varepsilon.
 \end{split}
\]
This proves \eqref{quant:entry-beta-rate} also in the final case, and
completes the proof.
\end{proof}

\subsection{Explicit derivative bounds and Clifford proximity}
\label{quant:jets-entry-section}

\begin{lemma}\label{quant:jets-clifford-entry}
Let $n\ge3$, set $x=100n$ and $K=x^{10000}$, and suppose that a
closed connected minimal immersion into $\mathbb S^{n+2}(1)$ has
constant $Q=n+\varepsilon$, where $\varepsilon\ge0$.
Write $b=\lambda_2$ and $\beta=\max_Mb$.  Suppose that
\begin{equation}\label{quant:jets-small-input}
 0<t\le x^{-30000},\qquad \varepsilon\le t,\qquad\beta\le t.
\end{equation}
After passing to a connected covering of degree at most four, the
normal eigenlines have unit sections and the first shape operator
$A$ is globally defined.  There is an integer $p\in\{1,\ldots,n-1\}$
such that
\begin{equation}\label{quant:jets-clifford-output}
 \sqrt\beta+
 \sup_M\operatorname{dist}(A,\mathcal O_p)
 +\sum_{r=1}^6\sup_M|\nabla^rA|
 \le K\sqrt t,
\end{equation}
where $\mathcal O_p$ is the orthogonal orbit of the standard
Clifford shape operator with multiplicities $p,n-p$.
More explicitly,
\begin{equation}\label{quant:jets-intermediate-bounds}
 \sup_M|\nabla^rh|\le x^{7001}\sqrt t,
 \qquad
 \sup_M|\nabla^rA|\le x^{7009}\sqrt t
 \quad(1\le r\le6).
\end{equation}
\end{lemma}

\begin{proof}
Put $a=\lambda_1=n+\varepsilon-2b$ and $\delta=a-b$.
The numerical assumptions give
\[
 \frac{3n}{4}\le a\le2n,\qquad
 \delta\ge\frac n2,\qquad b\le\frac a8.
\]
Thus the two normal eigenline bundles are smooth.  Orienting each
requires a double covering; a connected component of the resulting
covering has degree at most four.  All the assumptions are preserved.
Write the corresponding unit normals as $\nu_1,\nu_2$, with shape
operators $A,B$, and let
$\nabla^\perp\nu_1=\omega\nu_2$.

\smallskip
\noindent\emph{The first derivative of $h$.}
Set
\[
 U=|T_1|^2+2|T_2|^2,
 \qquad D=2ab-|[A,B]|^2\ge0.
\]
The two normal eigenvalue Simons identities and the constancy of
$a+2b$ give the exact equality
\[
 U-\delta|\omega|^2
 =Q\varepsilon+2\delta b-3D.
\]
Differentiating $\langle A,B\rangle=0$ gives
\[
 \delta\omega=\langle T_1,B\rangle+\langle A,T_2\rangle.
\]
Weighted Cauchy--Schwarz and $b\le a/8$ imply
\[
 \delta|\omega|^2
 \le\frac{b+a/2}{a-b}U\le\frac57U.
\]
It follows that
\begin{equation}\label{quant:first-h-jet-exact}
 |\nabla h|^2\le U
 \le\frac72Q\varepsilon+7ab
 \le21nt.
\end{equation}
In what follows put $s=\sqrt{21nt}$, so that $s<1$.

\smallskip
\noindent\emph{A finite derivative count.}
We spell out a numerical bound for the differentiated tensor Simons
equation so that no elliptic regularity constant is left implicit.
Write $H_j=\nabla^jh$, with $H_0=h$.  Expanding the cubic part of
the tensor Simons identity gives five cubic contractions when
absolute values of their coefficients are counted: one comes from
the fundamental matrix and four from the double commutator.
After $j$ differentiations these contribute at most $5\cdot3^j$
terms.  To commute a derivative past the rough Laplacian, use the
curvature action on each of the at most nine bundle or covariant
indices occurring for $j\le6$.  The Gauss and Ricci equations
express each curvature action as a constant-curvature term or a
quadratic contraction in $h$.  There are at most 200 such terms
at each commutation step, including their integer coefficients.
Thus, for $1\le j\le6$, the sum of absolute coefficients is at most
\begin{equation}\label{quant:finite-jet-count}
 5\cdot3^j+200\sum_{i=0}^{j-1}3^i
 \le76445.
\end{equation}
Each contraction has at most six summed tangent or normal indices.
Bounding their norm loss by $(2n)^6$, and inserting
$|h|\le\sqrt{2n}$ wherever an undifferentiated factor occurs, is
covered by
\[
 76445(2n)^7 < x^{20}=:R.
\]
This also covers the linear term $nH_j$ and the constant-curvature
commutator terms.

Here the bound 200 can also be checked directly from the
commutation formula.  For a normal-valued covariant $r$-tensor,
with $r\le8$, one curvature commutator acts on its $r$ tangent
indices and on its normal index.  Expanding Gauss and Ricci gives
at most $4r+2$ terms.  In $\Delta\nabla-\nabla\Delta$, the
commutator acting on a first derivative contributes at most
$4r+6$ terms, and differentiating the other commutator contributes
at most $6+8r$ terms: its quadratic-curvature terms acquire three
Leibniz terms, while its constant-curvature terms acquire one.
Thus at most $12r+12\le108<200$ terms are needed.

In particular, if $2\le j\le6$ and
$|H_i|\le Ds$ for $1\le i<j$, where $D\ge1$, then
\begin{equation}\label{quant:jet-rough-equations}
 |\Delta H_j|\le R|H_j|+RD^3s^2,
 \qquad
 |\Delta H_{j-1}|\le2RD^3s.
\end{equation}
Indeed, after terms linear in $H_j$ have been separated, every
remaining cubic term has at least two differentiated factors, each
of order smaller than $j$.  Their product is bounded by $D^3s^2$
because $s\le1$; the numerical coefficient already includes every
undifferentiated factor.  The same argument at order $j-1$ gives
the second bound.

\smallskip
\noindent\emph{The global Bernstein induction.}
Let $V=|H_j|^2$ and $W=|H_{j-1}|^2$.
Equation \eqref{quant:jet-rough-equations} and Young's inequality
give
\[
 \Delta V\ge2|H_{j+1}|^2-3RV-R^2D^6s^4,
 \qquad
 \Delta W\ge2V-4RD^4s^2.
\]
At a maximum point of $V+2RW$ on the closed manifold,
\[
 0\ge RV-R^2D^6s^4-8R^2D^4s^2,
\]
and hence $V\le9RD^6s^2$ there.  Since $W\le D^2s^2$,
\[
 \sup_M|H_j|^2
 \le\max_M(V+2RW)\le11RD^6s^2.
\]
It is therefore sufficient to set
\[
 D_1=1,\qquad D_j=2x^{50}D_{j-1}^3\quad(2\le j\le6).
\]
Starting from \eqref{quant:first-h-jet-exact}, induction proves
$\sup|H_j|\le D_js$.  Since $2\le x$,
\[
 D_j\le x^{\frac{51}{2}(3^{j-1}-1)},
 \qquad D_6\le x^{6171}.
\]
Using $\sqrt{21n}\le x$ and rounding upward, we conclude that
\begin{equation}\label{quant:h-six-jets}
 \sup_M|\nabla^jh|\le x^{6172}\sqrt t
 \le x^{7001}\sqrt t=:H\qquad(1\le j\le6).
\end{equation}
Notice that $H\le x^{-7999}$ under
\eqref{quant:jets-small-input}.

\smallskip
\noindent\emph{From normal-covariant derivatives to derivatives of $A$.}
This step uses the connection on the normal bundle.  Denote the first
normal eigenprojection by $P_\nu$.  It has the exact expression
\[
 P_\nu=\frac{\mathcal A-bI}{\delta},
 \qquad b=Q-S,\qquad\delta=3S-2Q.
\]
Leibniz's rule and \eqref{quant:h-six-jets} give, for $1\le r\le6$,
\[
 |(\nabla^\perp)^r\mathcal A|,\ |\nabla^rS|\le xH,
 \qquad |\nabla^rb|\le xH,\qquad |\nabla^r\delta|\le3xH.
\]
For example, the derivatives of a quadratic contraction of $h$ are
bounded by $2\sqrt{2n}H+(2^r-2)H^2\le xH$.
Differentiate $\delta P_\nu=\mathcal A-bI$ and use
$|P_\nu|=1$, $\delta^{-1}\le2/n$.  Induction in $r$ shows
\begin{equation}\label{quant:normal-projector-six-jets}
 |(\nabla^\perp)^rP_\nu|\le x^2H\qquad(1\le r\le6).
\end{equation}
Here is a bound for all terms in this induction.  The two terms with
all $r$ derivatives on the scalar or on the numerator contribute
at most $1200H$.  There are at most $2^r-2\le62$ remaining
Leibniz terms.  Assuming the lower-order bounds, their sum is at
most $37200x^2H^2$.  Thus
\[
 |(\nabla^\perp)^rP_\nu|
 \le1200H+37200x^2H^2\le x^2H.
\]
The last inequality follows from $x\ge300$ and $H\le x^{-7999}$.

For the unit section $\nu_1$ of this line, differentiation of
$P_\nu\nu_1=\nu_1$ gives
\[
 \nabla^\perp\nu_1=(\nabla^\perp P_\nu)\nu_1.
\]
Repeated differentiation of this formula and
\eqref{quant:normal-projector-six-jets} give
\[
 |(\nabla^\perp)^r\nu_1|\le2x^2H\qquad(1\le r\le6).
\]
Indeed, the induction involves at most $2^{r-1}-1\le31$ terms
with a differentiated normal section, and is absorbed by
$62x^2H<1$.  Finally, $A=\langle h,\nu_1\rangle$ and
Leibniz's rule imply
\[
 |\nabla^rA|
 \le H+2\sqrt{2n}x^2H+124x^2H^2
 \le x^{7009}\sqrt t\qquad(1\le r\le6).
\]
These are tangent covariant derivatives of $A$; the preceding
calculation includes all contributions from the normal connection.

\smallskip
\noindent\emph{The principal-curvature polynomial.}
Diagonalize $A$ at a point, with eigenvalues
$\kappa_1,\ldots,\kappa_n$.  The Ricci commutation identity for
the tangent two-tensor $A$, together with the Gauss equation, gives
\[
 |(\kappa_i-\kappa_j)(1+\kappa_i\kappa_j)|
 \le2|\nabla^2A|+
    |\kappa_i-\kappa_j|\,|B_{ii}B_{jj}-B_{ij}^2|.
\]
Since $|\kappa_i-\kappa_j|\le2\sqrt n$ and
$|B_{ii}B_{jj}-B_{ij}^2|\le b/2$, it follows that
\begin{equation}\label{quant:principal-polynomial-small}
 |(\kappa_i-\kappa_j)(1+\kappa_i\kappa_j)|\le\tau,
 \qquad \tau=x^{7010}\sqrt t<\frac18.
\end{equation}

We now give an elementary spectral clustering argument with its
constant.  Put $\ell=\min_i\kappa_i$, $r=\max_i\kappa_i$,
and $D=r-\ell$.  The trace-free condition implies
\[
 a\le-n\ell r\le\frac n4D^2,
 \qquad D^2\ge\frac{4a}{n}\ge3.
\]
For any eigenvalue $z$, put $d_1=z-\ell$, $d_2=r-z$ and
$e=1+\ell r$.  The two polynomial expressions involving $z$ and
the extreme eigenvalues satisfy
\[
 d_1(1+\ell z)+d_2(1+rz)=D(e+d_1d_2).
\]
Equation \eqref{quant:principal-polynomial-small} also gives
$D|e|\le\tau$.  Consequently
\[
 d_1d_2\le\frac{3\tau}{D},\qquad
 \min\{d_1,d_2\}\le\frac{6\tau}{D^2}\le2\tau.
\]
Thus every eigenvalue lies in one of the disjoint intervals
$[\ell,\ell+2\tau]$ and $[r-2\tau,r]$.
Let their multiplicities be $p$ and $m=n-p$; both are positive.
The gap between the intervals is at least $D-4\tau>1$, so
these multiplicities are constant on the connected covering, and
the corresponding spectral projection $P$ is smooth.

Let $\bar\ell$ and $\bar r$ be the averages in the two clusters
and put $A_0'=\bar\ell P+\bar r(I-P)$.  Then
\[
 |A-A_0'|\le2\sqrt n\,\tau,
 \qquad p\bar\ell+m\bar r=0,
 \qquad |A_0'|^2=a-|A-A_0'|^2.
\]
Write $a_0=|A_0'|^2$.  The trace condition and the ordering of the
two cluster means imply
\[
 A_0'=\sqrt{\frac{a_0}{n}}
 \left(-\sqrt{\frac mp}P+\sqrt{\frac pm}(I-P)\right).
\]
The expression in parentheses is a Clifford shape operator.
It follows that
\begin{align}
 \operatorname{dist}(A,\mathcal O_p)
 &\le2\sqrt n\,\tau+|\sqrt{a_0}-\sqrt n|\notag\\
 &\le2\sqrt n\,\tau+\frac{|a-n|}{\sqrt n}
                 +4\sqrt n\,\tau^2\notag\\
 &\le7\sqrt n\,\tau+\frac{|a-n|}{\sqrt n}.
 \label{quant:explicit-spectral-clustering}
\end{align}
Finally $|a-n|\le\varepsilon+2b\le3t$.
Combining \eqref{quant:explicit-spectral-clustering},
\eqref{quant:principal-polynomial-small}, the six derivative bounds,
and $\sqrt\beta\le\sqrt t$, gives
\[
 \sqrt\beta+\sup_M\operatorname{dist}(A,\mathcal O_p)
       +\sum_{r=1}^6\sup_M|\nabla^rA|
 \le x^{10000}\sqrt t.
\]
This proves \eqref{quant:jets-clifford-output}.
\end{proof}

\subsection{Quantitative estimates near a Clifford shape operator}
\label{quant:core-setup}

We give the quantitative estimates used in the local rigidity argument.
Throughout this subsection,
\(F:M^n\to\mathbb S^{n+2}(1)\) is a closed connected minimal immersion
with \(Q\equiv n+\varepsilon\), where \(\varepsilon\ge0\).
Tensor norms in this subsection are induced by the Riemannian metric
and the Hilbert--Schmidt inner products.  Norms of linear and multilinear
maps between these tensor fibres are the corresponding induced operator
norms.  For their covariant derivatives, we take the Euclidean square
sum over the derivative indices of these operator norms.  Fix
\begin{equation}\label{quant:master-constants}
 x=100n,\qquad N=x^{100000},\qquad L=100nN=x^{100001},
 \qquad 0<\eta\le L^{-1000}.
\end{equation}
Let \(p,m\) be positive integers with \(p+m=n\), and put
\begin{equation}\label{quant:clifford-roots}
 \lambda=-\sqrt{\frac mp},\qquad
 \mu=\sqrt{\frac pm},\qquad
 c=\lambda+\mu,\qquad
 D=\mu-\lambda=\frac n{\sqrt{pm}}\ge2.
\end{equation}
Suppose that the first normal eigenline has been oriented and that its
shape operator \(A\) has a smooth spectral projection \(P\) of rank \(p\).
Write \(P^\perp=I-P\), and set
\[
 A_*=\lambda P+\mu P^\perp .
\]
The hypotheses needed below are
\begin{equation}\label{quant:clifford-neighborhood}
 |A-A_*|\le\eta,\qquad
 |\nabla^jA|\le\eta\quad(1\le j\le4),\qquad
 0\le\beta\le\eta^2,
\end{equation}
where \(\beta=\max_M b\), \(a=|A|^2\), \(b=|B|^2\), and
\[
 a+2b=n+\varepsilon,\qquad \varepsilon\ge0.
\]
All supremum bounds in this subsection are taken over the closed domain,
or over the finite cover on which the two normal eigenlines are oriented.

For \(\beta>0\), define
\begin{equation}\label{quant:normalized-fields}
 C=\frac{B}{\sqrt\beta},\qquad
 \theta=\frac{\omega}{\sqrt\beta},\qquad
 u=|C|^2,\qquad
 \delta=a-b=a-\beta u .
\end{equation}
When \(\beta=0\), the rank-one reduction makes the unused normal line
parallel; for formulas involving the auxiliary fields below, we then
set \(C=\theta=0\) and \(u=0\).
Thus \(0\le u\le1\), and the normal eigenframe identities give
\[
 \operatorname{tr}C=0,\qquad \langle A,C\rangle=0,\qquad
 \operatorname{div}C=-A\theta .
\]
The assumptions imply the following elementary bounds:
\begin{equation}\label{quant:elementary-neighborhood-bounds}
\begin{gathered}
 \frac n2\le a\le2n,\qquad
 \delta\ge\frac n2,\qquad
 |A|\le2\sqrt n,\qquad
 \|A^{-1}\|_{\mathrm{op}}\le2\sqrt n,\qquad |c|\le\sqrt n,\\
 0\le\varepsilon\le3\sqrt n\,\eta+2\beta,\qquad
 \varepsilon+2\beta\le4\sqrt n\,\eta .
\end{gathered}
\end{equation}
Indeed, \(|A_*|^2=n\), so
\[
 |a-n|\le2\sqrt n\,\eta+\eta^2\le3\sqrt n\,\eta.
\]
The least absolute value of a Clifford principal curvature is at least
\(1/\sqrt n\), which proves the inverse bound.  The remaining assertions
follow from \(a+2b=n+\varepsilon\), \(b\le\beta\le\eta^2\), and
\eqref{quant:master-constants}.

\begin{lemma}\label{quant:normalized-derivatives}
Under \eqref{quant:master-constants}--\eqref{quant:clifford-neighborhood},
if \(\beta>0\), then
\begin{equation}\label{quant:normalized-derivative-bound}
 |C|\le1,\qquad
 |\nabla C|,\ |\nabla^2C|,\ |\nabla^3C|,\
 |\theta|,\ |\nabla\theta|,\ |\nabla^2\theta|\le N .
\end{equation}
The bounds require no positive lower bound for \(b\).
\end{lemma}

\begin{proof}
We first record a semilinear equation and then give a global Bernstein
argument through order three.

\smallskip
\noindent\emph{The semilinear equation.}
The tangent Simons equation in the second normal direction is
\[
 \Delta B=(n-b)B-[A,[A,B]]
          -2\sum_i\omega_i\nabla_iA
          -(\operatorname{div}\omega)A+|\omega|^2B.
\]
Differentiating \(\langle A,B\rangle=0\) twice and substituting the two
tangent Simons equations gives
\[
 \delta\,\operatorname{div}\omega
 =\omega\cdot(\nabla b-\nabla a)
     +2\langle\nabla A,\nabla B\rangle.
\]
Since \(\nabla a=-2\nabla b\), division by \(\sqrt\beta\) yields
\begin{equation}\label{quant:normalized-divergence}
 \delta\,\operatorname{div}\theta
 =2\langle\nabla A,\nabla C\rangle
       +6\beta\sum_i\theta_i\langle C,\nabla_iC\rangle.
\end{equation}
For a three-tensor \(U\), symmetric in its last two indices, define
\[
 (\operatorname{ctr}U)_k=\sum_iU_{iik},
 \qquad z(U)=-A^{-1}\operatorname{ctr}U.
\]
For \(U=\nabla C\), we have \(z(U)=\theta\).
Equations \eqref{quant:normalized-divergence} and
\(\operatorname{div}C=-A\theta\) therefore give the exact equation
\begin{equation}\label{quant:normalized-system}
 \Delta C=\mathscr V+\mathscr P(\nabla C)
                     +\beta\mathscr Q(\nabla C,\nabla C),
\end{equation}
where
\begin{align*}
 \mathscr V&=(n-\beta|C|^2)C-[A,[A,C]],\\
 \mathscr P(U)
 &=-2\sum_i z_i(U)\nabla_iA
       -2\delta^{-1}\langle\nabla A,U\rangle A,\\
 \mathscr Q(U,U)
 &=|z(U)|^2C
       -6\delta^{-1}\sum_i z_i(U)\langle C,U_i\rangle A.
\end{align*}
The bilinear map \(\mathscr Q\) is obtained by polarization of its
displayed quadratic expression.

Put \(q=|\nabla C|^2\).  The coefficient estimates needed for the first
derivative are
\begin{equation}\label{quant:first-coefficients}
\begin{gathered}
 \|z\|\le2n,\qquad
 |\mathscr V|\le10n,\qquad
 \|\mathscr V_C\|\le12n,\qquad
 \|\mathscr V_A\|\le16\sqrt n,\\
 \|\mathscr P\|\le8n\eta,\qquad
 \|\mathscr Q\|\le32n^2,\\
 |\nabla\mathscr P|
       \le16n\sqrt n\,\eta+32\beta\eta\sqrt q,\qquad
 |\nabla\mathscr Q|
       \le128n^3(\eta+\sqrt q).
\end{gathered}
\end{equation}
In the last line, the tensor arguments of the fibre maps are held
parallel while their coefficients are differentiated.
For completeness, the estimates follow from
\[
 |\operatorname{ctr}U|\le\sqrt n\,|U|,\qquad
 \nabla(A^{-1})=-A^{-1}(\nabla A)A^{-1},
 \qquad
 |\nabla\delta|\le4\sqrt n\,\eta+2\beta\sqrt q .
\]
The double commutator satisfies
\[
 \|[A,[A,\cdot]]\|_{\mathrm{op}}
       =\operatorname{osc}(A)^2\le2|A|^2\le8n.
\]
For example, the two terms in \(\mathscr Q\) have norms at most
\(4n^2\) and \(48\sqrt n\), respectively.  Their coefficient derivatives
are bounded by
\[
 (16n^2\sqrt n+96n+408)\eta
 +(4n^2+48\sqrt n+192\beta/\sqrt n)\sqrt q,
\]
which is bounded by the last expression in
\eqref{quant:first-coefficients}.  Differentiating \(\mathscr P\) gives
\[
 (4n+8/\sqrt n)\eta
 +(8n\sqrt n+68/n)\eta^2
 +32\beta\eta\sqrt q/n^{3/2},
\]
and hence its stated bound.

\smallskip
\noindent\emph{The first derivative.}
The Gauss equation and \(B=\sqrt\beta\,C\) give
\begin{equation}\label{quant:curvature-derivatives-first}
 |\operatorname{Rm}|\le7n,\qquad
 |\nabla\operatorname{Rm}|
       \le8\sqrt n\,\eta+4\beta\sqrt q.
\end{equation}
The commutation formula for a symmetric two-tensor gives
\begin{equation}\label{quant:rough-commutator}
 |\Delta\nabla C-\nabla\Delta C|
 \le4\sqrt n\,|\operatorname{Rm}|\,|\nabla C|
       +2\sqrt n\,|\nabla\operatorname{Rm}|\,|C|.
\end{equation}
To verify the coefficients, the Ricci term is at most
\(\sqrt n|\operatorname{Rm}||\nabla C|\), the two curvature actions on
the tensor indices contribute at most
\(4|\operatorname{Rm}||\nabla C|\), and the derivative of these actions
contributes at most \(2\sqrt n|\nabla\operatorname{Rm}||C|\).
Since \(n\ge3\), their sum has the bound
\eqref{quant:rough-commutator}.

Set
\begin{equation}\label{quant:K-zero}
 K_0=2^{20}n^6.
\end{equation}
Differentiating \eqref{quant:normalized-system}, using
\eqref{quant:first-coefficients}--\eqref{quant:rough-commutator}, and
taking the inner product with \(\nabla C\), we obtain
\begin{equation}\label{quant:bernstein-first}
 \frac12\Delta q
 \ge\frac12|\nabla^2C|^2-K_0(q+1)-K_0\beta q^2.
\end{equation}
Here is an explicit absorption of the terms containing
\(\nabla^2C\):
\[
\begin{split}
 &8n\eta\sqrt q\,|\nabla^2C|
       +64\beta n^2q\,|\nabla^2C|\\
 &\hspace{1cm}\le
 \frac12|\nabla^2C|^2
       +64n^2\eta^2q+4096\beta^2n^4q^2.
\end{split}
\]
The remaining terms, including the curvature commutators, are bounded
by
\[
 512n^3(q+1)+256n^3\beta q^2.
\]
Since \(\beta\le1\), these expressions are covered by \(K_0\) in
\eqref{quant:bernstein-first}.

Also, \eqref{quant:normalized-system} implies
\[
 \frac12\Delta u
 \ge(1-32n^2\beta)q-8n\eta\sqrt q-10n.
\]
The inequalities
\[
 8n\eta\sqrt q\le\frac14q+64n^2\eta^2,\qquad
 \beta\le\frac1{128n^2}
\]
therefore give
\begin{equation}\label{quant:laplacian-u-lower}
 \Delta u\ge q-24n^2.
\end{equation}

Apply the maximum principle to \(q/(2-u)\).
At a maximum point put \(w=2-u\), so \(1\le w\le2\).
The gradient terms in its Laplacian cancel, and
\eqref{quant:bernstein-first}--\eqref{quant:laplacian-u-lower} give
\[
 0\ge(1-4K_0\beta)q^2-(24n^2+4K_0)q-4K_0.
\]
If \(\beta\le(8K_0)^{-1}\), then \(q<12K_0\) at that point.
Consequently \(\sup q<24K_0<32K_0\).  We may thus fix
\begin{equation}\label{quant:N-one}
 N_1=2^{13}n^3,\qquad |\nabla C|\le N_1.
\end{equation}

\smallskip
\noindent\emph{The second and third derivatives.}
We provide the finite product bounds in this step because the second
derivative equation has a quadratic highest-order term.  Put
\[
 B_0=(100n)^{100}.
\]
Expanding at most three covariant derivatives of
\eqref{quant:normalized-system} yields, for \(X_2=\nabla^2C\),
\begin{equation}\label{quant:second-derivative-structure}
\begin{split}
 |\Delta X_2-\mathscr D_2(\nabla X_2)|
 &\le B_0\bigl((1+N_1)^4|X_2|
                    +(1+N_1)^5+\beta|X_2|^2\bigr),\\
 \|\mathscr D_2\|&\le B_0(1+N_1).
\end{split}
\end{equation}
Once \(|\nabla^2C|\le N_2\) is known, the corresponding equation for
\(X_3=\nabla^3C\) satisfies
\begin{equation}\label{quant:third-derivative-structure}
\begin{split}
 |\Delta X_3-\mathscr D_3(\nabla X_3)|
 &\le B_0\bigl((1+N_2)^4|X_3|+(1+N_2)^5\bigr),\\
 \|\mathscr D_3\|&\le B_0(1+N_2).
\end{split}
\end{equation}

The following direct count verifies these bounds.  After expanding the
commutators, \(\mathscr V,\mathscr P,\mathscr Q\) contain at most sixteen
basic products.  Their longest coefficient product has the form
\(\delta^{-1}A^{-1}CA\), with four factors.  Repeated differentiation
uses
\[
 D(A^{-1})[H]=-A^{-1}HA^{-1},\qquad
 D(\delta^{-1})=-\delta^{-2}D\delta.
\]
At each of the three differentiations, allowing thirty-two descendants
per product bounds the total number by
\(16\cdot32^3=2^{19}\).  The differentiated Gauss equation and the
commutation of up to three covariant derivatives increase this to fewer
than \(2^{20}\) products.  An expanded product contains at most twenty
basic factors.  Apart from positive-order derivatives of \(C\), each
factor has norm at most \(2n\).  Allowing an additional factor \(n\) for
each contraction, as well as the integer coefficients in the equations,
bounds the resulting coefficient sum by
\[
 2^{32}(2n)^{40}\le\frac{(100n)^{100}}{16}\qquad(n\ge3).
\]
The largest total differentiation order in the \(C\)-factors is
\(r+2\le5\), which gives the fifth power in the inhomogeneous terms of
\eqref{quant:second-derivative-structure} and
\eqref{quant:third-derivative-structure}.  The first-order highest
derivative has coefficient \(\mathscr P+2\beta\mathscr Q(\nabla C,\cdot)\).
The only term quadratic in \(X_2\) is obtained by differentiating the
two \(\nabla C\) factors in
\(\beta\mathscr Q(\nabla C,\nabla C)\) once each.  It retains its factor
\(\beta\).  For \(r=3\), two factors \(X_3\) would have total order six;
the remaining nonlinear terms instead involve
\(\nabla^2C\,\nabla^3C\), and are linear in \(X_3\) after the bound
for \(N_2\).  Curvature derivatives are treated by differentiating
the Gauss equation; they contribute at most one unknown derivative
\(\nabla^rC\), with coefficient containing \(\beta\).  This proves the
two displayed structure bounds.

For every integer \(j\ge1\), write \(H_j=|\nabla^jC|^2\).
For \(r=2,3\), define successively
\begin{equation}\label{quant:bernstein-recursion}
 \begin{split}
 H_r&=|\nabla^rC|^2,\\
 L_r&=[B_0(1+N_{r-1})]^{16},\qquad
 W_r=1+N_{r-1}^2,\qquad
 N_r=16L_rW_r^2 .
 \end{split}
\end{equation}
The structure bounds and Young's inequality give
\begin{equation}\label{quant:higher-bernstein}
 \frac12\Delta H_r
 \ge\frac12H_{r+1}-L_r(H_r+1)-L_r\beta H_r^2.
\end{equation}
In deriving the \(r=2\) inequality, use
\(H_2^{3/2}\le H_2+H_2^2\); thus the term
\(\beta|\nabla^2C|^2\) in
\eqref{quant:second-derivative-structure} has not been discarded.
The factor \(L_r\) covers the squares of the first-order coefficients
and the tenth powers that arise from Young's inequality.  For \(r=3\),
the last negative term in \eqref{quant:higher-bernstein} may simply
be added.

The already established estimate at order \(r-1\) also gives
\begin{equation}\label{quant:previous-order-lower}
 \Delta H_{r-1}\ge H_r-L_r.
\end{equation}
For \(r=2\), this follows from \eqref{quant:bernstein-first}.
For \(r=3\), use \eqref{quant:higher-bernstein} at order two and
\[
 4L_2(1+N_2)^4\le L_3.
\]
The latter inequality follows directly from
\(N_2=16L_2W_2^2\).

Apply the maximum principle to
\[
 \frac{H_r}{W_r-H_{r-1}},\qquad
 1\le W_r-H_{r-1}\le W_r.
\]
If
\begin{equation}\label{quant:higher-beta-condition}
 \beta\le\frac1{8L_rW_r},
\end{equation}
then \eqref{quant:higher-bernstein} and
\eqref{quant:previous-order-lower} imply at its maximum point
\[
 0\ge\frac12H_r^2-3L_rW_rH_r-2L_rW_r.
\]
It follows that \(\sup H_r\le8L_rW_r^2\), which proves
\(|\nabla^rC|\le N_r\).

We check explicitly that the chosen master constants imply every
smallness condition used above.  With \(x=100n\),
\[
 \begin{gathered}
 N_1\le x^5,\quad 1+N_1\le x^6,\quad W_2\le x^{11},\\
 L_2\le x^{1696},\quad N_2\le x^{1719}\le x^{1720},\\
 W_3\le x^{3441},\quad L_3\le x^{29136},\quad
 N_3\le x^{36019}<x^{40000}.
 \end{gathered}
\]
In particular,
\[
 8L_2W_2\le x^{1708},\qquad
 8L_3W_3\le x^{32578}.
\]
Thus \(\beta\le x^{-50000}\) suffices for
\eqref{quant:higher-beta-condition} and both first-order smallness
conditions.  This follows from
\(\beta\le\eta^2\le L^{-2000}=x^{-200002000}\).

Finally,
\[
 |\nabla(A^{-1})|\le4n\eta,\qquad
 |\nabla^2(A^{-1})|
     \le16n^{3/2}\eta^2+4n\eta.
\]
Differentiating \(\theta=-A^{-1}\operatorname{div}C\) twice consequently
gives
\[
 |\theta|+|\nabla\theta|+|\nabla^2\theta|
       \le100n^3N_3<N.
\]
Together with the estimates for \(N_1,N_2,N_3\), this proves
\eqref{quant:normalized-derivative-bound}.
\end{proof}

\subsection{The polynomial defect and its quantitative rate}
\label{quant:polynomial-subsection}

Define
\begin{equation}\label{quant:polynomial-defect}
 E=A^2-cA-I,\qquad d=\max_M|E|,\qquad
 \operatorname{tr}E=\varepsilon-2\beta u.
\end{equation}
The neighborhood condition gives \(d\le4\sqrt n\,\eta\).

\begin{lemma}\label{quant:polynomial-coercivity}
Under \eqref{quant:master-constants}--\eqref{quant:clifford-neighborhood},
the polynomial defect satisfies the pointwise inequality
\begin{equation}\label{quant:curvature-coercivity}
 |\nabla^2E|^2
 \ge\frac n4|E|^2-n(\varepsilon+2\beta)^2.
\end{equation}
\end{lemma}

\begin{proof}
At the point in question diagonalize \(A\), and write its eigenvalues as
\(\lambda+\delta_i\) in the \(P\)-block and
\(\mu+\delta_i\) in the \(P^\perp\)-block.
Since the roots of \(t^2-ct-1\) are simple,
\[
 |\delta_i|\le\frac{2|E_i|}{D},\qquad
 \sum_i\delta_i=0.
\]
Consequently, with \(t=\operatorname{tr}E\),
\begin{equation}\label{quant:block-trace-identity}
 \operatorname{tr}(PE)-\operatorname{tr}(P^\perp E)
      =\sum_{P}\delta_i^2-\sum_{P^\perp}\delta_i^2,
 \qquad
 \left|\operatorname{tr}(PE)-\operatorname{tr}(P^\perp E)\right|
      \le|E|^2 .
\end{equation}
Here the linear terms cancel because \(\sum_i\delta_i=0\), and the last
inequality uses \(D\ge2\).
Set
\[
 W=\frac{\operatorname{tr}(PE)}pP
       +\frac{\operatorname{tr}(P^\perp E)}mP^\perp,
 \qquad Z=E-W.
\]
Thus \(Z\) has zero trace in each block.  The orthogonality of the two
blocks and \eqref{quant:block-trace-identity} give
\begin{equation}\label{quant:traceless-block-bound}
 |W|\le|t|+|E|^2,\qquad
 |Z|^2\ge(1-2d^2)|E|^2-2t^2.
\end{equation}

Let \(\operatorname{Rm}_*\) be the algebraic curvature tensor obtained
from the Gauss equation with shape operators \(A_*,0\).
Its sectional curvatures are \(n/p\) in the \(P\)-block,
\(n/m\) in the \(P^\perp\)-block, and zero in mixed planes.
Summing over ordered pairs \(i,j\), direct diagonalization of each
trace-free block yields
\begin{equation}\label{quant:model-curvature-action}
 \begin{split}
 \sum_{i,j}|[R^*_{ij},Z]|^2
 &=4p\left(\frac np\right)^2|Z_P|^2
      +4m\left(\frac nm\right)^2|Z_{P^\perp}|^2\\
 &\ge4n|Z|^2 .
 \end{split}
\end{equation}
For instance, on a block of dimension \(r\) and curvature \(k\), the
sum is
\[
 4k^2\sum_{i<j}(z_i-z_j)^2=4rk^2\sum_i z_i^2
\]
when \(\sum_i z_i=0\).
The block-scalar tensor \(W\) commutes with every \(R^*_{ij}\).

The Gauss equation and \eqref{quant:clifford-neighborhood} imply
\[
 |\operatorname{Rm}-\operatorname{Rm}_*|
      \le16\sqrt n\,\eta.
\]
Its action on \(E\) therefore has norm at most
\(32\sqrt n\,\eta|E|\).
Using \(|X+Y|^2\ge|X|^2/2-|Y|^2\), followed by
\eqref{quant:traceless-block-bound}, gives
\[
 \begin{split}
 \sum_{i,j}|R_{ij}\cdot E|^2
 &\ge2n|Z|^2-1024n\eta^2|E|^2\\
 &\ge\bigl(2n-4nd^2-1024n\eta^2\bigr)|E|^2-4nt^2\\
 &\ge n|E|^2-4nt^2.
 \end{split}
\]
For the last step it is enough that
\(\eta\le2^{-10}/n\) and \(d\le4\sqrt n\,\eta\), both consequences of
our hypotheses.
Finally, the Ricci commutation identity gives
\[
 \sum_{i,j}|R_{ij}\cdot E|^2
   =\sum_{i,j}|(\nabla_i\nabla_j-\nabla_j\nabla_i)E|^2
   \le4|\nabla^2E|^2.
\]
Since \(|t|\le\varepsilon+2\beta\), this proves
\eqref{quant:curvature-coercivity}.
\end{proof}

\begin{proposition}\label{quant:polynomial-rate}
Under \eqref{quant:master-constants}--\eqref{quant:clifford-neighborhood},
the following bounds hold:
\begin{equation}\label{quant:polynomial-rate-conclusion}
 \begin{split}
 d+\sup_M|\nabla A|&\le L^{10}(\varepsilon+\beta),\\
 \varepsilon&\le130N^2\beta,\\
 d+\sup_M|\nabla A|&\le L^{12}\beta.
 \end{split}
\end{equation}
If \(\beta=0\), the same conclusions follow after using the rank-one
reduction and setting \(C=\theta=0\); in particular \(\varepsilon=0\).
\end{proposition}

\begin{proof}
We use the finite-dimensional inverse in the polynomial identity, then
combine a Bernstein estimate for \(E\) with
Lemma~\ref{quant:polynomial-coercivity}.

\smallskip
\noindent\emph{A uniformly bounded left inverse.}
Let \(\mathcal T\) be the space of tensors \(T_{ikr}\) symmetric in
\(k,r\), and let \(\mathcal W\) be the space of tensors antisymmetric
in \(i,k\).  Define
\begin{equation}\label{quant:M-map}
 \mathcal M_A:\mathcal T\longrightarrow\mathcal T\oplus\mathcal W,
 \qquad
 \mathcal M_A(T)=
 \left(A_{k\ell}T_{i\ell r}
        +T_{ik\ell}A_{\ell r}-cT_{ikr},
       \ T_{ikr}-T_{kir}\right).
\end{equation}
At \(A_*\), its smallest singular value squared is
\[
 \frac{D^2+3-\sqrt{(D^2+1)^2+8}}2
      \ge\frac{D^2}{D^2+3}\ge\frac47.
\]
To verify this, the coefficient in the first component is \(-D\)
when its last two indices are in the first block, \(D\) when both
are in the second block, and zero for mixed last indices.
The nontrivial mixed-index blocks reduce, after using the symmetry
in the last two indices, to
\[
 \begin{pmatrix}D^2+2&-\sqrt2\\-\sqrt2&1\end{pmatrix}.
\]
For three distinct indices there is in addition an eigenvalue \(3\);
the blocks whose indices all belong to one eigenspace are bounded
below by \(D^2\).  This gives the displayed least eigenvalue.

Since
\[
 \|\mathcal M_A-\mathcal M_{A_*}\|_{\mathrm{op}}
     \le2|A-A_*|,
\]
we have \(\sigma_{\min}(\mathcal M_A)\ge1/2\) whenever
\(\eta\le1/8\).
Its left inverse
\[
 J_A=(\mathcal M_A^*\mathcal M_A)^{-1}\mathcal M_A^*
\]
therefore satisfies \(\|J_A\|_{\mathrm{op}}\le2\).
Writing \(\Pi_A=\mathcal M_AJ_A\), differentiation gives
\[
 DJ_A[H]
 =-J_A(D\mathcal M_A[H])J_A
  +(\mathcal M_A^*\mathcal M_A)^{-1}
       (D\mathcal M_A[H])^*(I-\Pi_A).
\]
As \(\|D\mathcal M_A[H]\|\le2|H|\), it follows that
\begin{equation}\label{quant:left-inverse-derivative}
 \|DJ_A[H]\|\le16|H|.
\end{equation}
Let \(L_A(Z)=J_A(Z,0)\).
Differentiating \(E=A^2-cA-I\), and using
\(\nabla A=T_1+\beta\theta\otimes C\), gives
\[
 \mathcal M_A(\nabla A)
  =\left(\nabla E,\
    \beta(\theta_iC_{kr}-\theta_kC_{ir})_{ikr}\right).
\]
Thus
\begin{equation}\label{quant:first-shape-left-inverse}
 \nabla A=L_A(\nabla E)+\beta\mathcal R,\qquad
 |\mathcal R|\le4N,\qquad |\nabla\mathcal R|\le64N^2.
\end{equation}
The last bound follows from
\eqref{quant:left-inverse-derivative} and
Lemma~\ref{quant:normalized-derivatives}.

\smallskip
\noindent\emph{The equation for \(E\).}
Write
\begin{equation}\label{quant:K-tensor}
 K=-[C,[C,A]]
      +2\sum_i\theta_i\nabla_iC
      +(\operatorname{div}\theta)C+|\theta|^2A.
\end{equation}
The first tangent Simons equation becomes
\begin{equation}\label{quant:tangent-simons-normalized}
 \Delta A=(-\varepsilon+2\beta u)A+\beta K.
\end{equation}
In particular,
\[
 |K|\le\frac{25}{2}\sqrt n\,N^2,\qquad
 |K|+|\nabla K|\le100\sqrt n\,N^2.
\]
The first estimate follows by bounding the four terms in
\eqref{quant:K-tensor} by
\(8\sqrt n\), \(2N^2\), \(\sqrt nN\), and \(2\sqrt nN^2\).
One more differentiation, using the bounds through order three for
\(C\) and through order two for \(\theta\), gives the second estimate.

Set \(T(A)=2A^2-cA\).
Using \eqref{quant:first-shape-left-inverse} in the exact identity
\[
 \Delta E=A\Delta A+(\Delta A)A-c\Delta A
                  +2\sum_i(\nabla_iA)^2
\]
gives
\begin{equation}\label{quant:E-drift-equation}
 \Delta E=\mathcal B(\nabla E)+\mathcal H,
 \qquad \mathcal H=-\varepsilon T(A)+\beta\mathcal F,
\end{equation}
where
\[
 \begin{split}
 \mathcal B(Z)
   &=\sum_i\bigl((\nabla_iA)(L_AZ)_i
                         +(L_AZ)_i(\nabla_iA)\bigr),\\
 \mathcal F
   &=2uT(A)+AK+KA-cK\\
   &\quad+\sum_i\bigl((\nabla_iA)\mathcal R_i
                              +\mathcal R_i(\nabla_iA)\bigr).
 \end{split}
\]
This rewriting uses the exact equality
\[
 2(\nabla_iA)^2
 =(\nabla_iA)(L_A\nabla E)_i
   +(L_A\nabla E)_i(\nabla_iA)
   +\beta\bigl((\nabla_iA)\mathcal R_i
                      +\mathcal R_i(\nabla_iA)\bigr).
\]
It therefore retains the entire quadratic derivative term.

Put
\begin{equation}\label{quant:E-estimate-constants}
 t=\varepsilon+\beta,\qquad M_0=L^6,\qquad R_0=L^{10}.
\end{equation}
The displayed formulas, \eqref{quant:left-inverse-derivative}, and
\eqref{quant:first-shape-left-inverse} imply
\begin{equation}\label{quant:E-coefficient-bounds}
 |\mathcal B|\le4\eta,\qquad
 |\nabla\mathcal B|\le36\eta,\qquad
 |\mathcal H|+|\nabla\mathcal H|\le M_0t.
\end{equation}
For the last inequality one can use the intermediate bounds
\[
 |T(A)|\le10n,\quad |\nabla T(A)|\le9\sqrt n\,\eta,\quad
 |\mathcal F|+|\nabla\mathcal F|\le4096n^2N^3,
\]
all of which are strictly smaller than the chosen \(L^6\) coefficient.

\smallskip
\noindent\emph{A global maximum estimate.}
The master hypotheses imply \(\beta N\le\eta\), so the Gauss equation
now gives
\begin{equation}\label{quant:E-curvature-derivative}
 |\nabla\operatorname{Rm}|
       \le8\sqrt n\,\eta+4\beta|\nabla C|
       \le12\sqrt n\,\eta .
\end{equation}
Put \(G=|\nabla E|^2\).  Applying the same tensor commutation formula as
in \eqref{quant:rough-commutator} to
\eqref{quant:E-drift-equation}, using
\eqref{quant:E-coefficient-bounds}--\eqref{quant:E-curvature-derivative},
gives
\begin{equation}\label{quant:E-gradient-laplacian}
 \frac12\Delta G
 \ge\frac12|\nabla^2E|^2
       -100n^2G-144n^2\eta^2|E|^2-\frac14M_0^2t^2.
\end{equation}
For clarity, the term \(4\eta\sqrt G|\nabla^2E|\) is bounded by
\(\frac12|\nabla^2E|^2+8\eta^2G\).
The derivative of the curvature contributes at most
\(24n\eta|E|\sqrt G\), which is bounded by
\(G+144n^2\eta^2|E|^2\).
The term \(|\nabla\mathcal H|\sqrt G\) is bounded by
\(G+M_0^2t^2/4\).
The remaining coefficient of \(G\) is at most
\(8\eta^2+36\eta+28n^{3/2}+2<100n^2\).
Likewise,
\begin{equation}\label{quant:E-norm-laplacian}
 \frac12\Delta|E|^2
 \ge\frac12G-8\eta^2|E|^2-M_0t|E|.
\end{equation}

Let
\[
 J=G+400n^2|E|^2,\qquad K_E=2^{21}n^3M_0^2.
\]
Combining \eqref{quant:E-gradient-laplacian},
\eqref{quant:E-norm-laplacian}, and
\eqref{quant:curvature-coercivity}, and using
\((\varepsilon+2\beta)^2\le4t^2\), yields
\[
 \begin{split}
 \frac12\Delta J
 &\ge100n^2G
   +\left(\frac n8-3344n^2\eta^2\right)|E|^2
   -400n^2M_0t|E|-\left(\frac14M_0^2+2n\right)t^2\\
 &\ge100n^2G+\frac n{16}|E|^2-K_Et^2\\
 &\ge\frac{J}{6400n}-K_Et^2.
 \end{split}
\]
Here \(\eta\le(1000n)^{-1}\) suffices.  To check the second line, use
\[
 400n^2M_0t|E|
 \le\frac n{32}|E|^2+1280000\,n^3M_0^2t^2.
\]
At a maximum point of \(J\), we conclude
\[
 J\le6400nK_Et^2.
\]
Equations \eqref{quant:first-shape-left-inverse} and
\eqref{quant:E-estimate-constants} then give
 \begin{equation}\label{quant:polynomial-pre-rate}
 d+\sup_M|\nabla A|
 \le2^{19}n^2M_0t
 \le R_0t=L^{10}(\varepsilon+\beta).
 \end{equation}
This proves the first line of
\eqref{quant:polynomial-rate-conclusion}.

\smallskip
\noindent\emph{Absorbing \(\varepsilon\).}
Taking the Laplacian of \(a=|A|^2=n+\varepsilon-2\beta u\) and using
\eqref{quant:tangent-simons-normalized} gives the exact identity
\begin{equation}\label{quant:trace-identity}
 \varepsilon a
   =|\nabla A|^2
      +\beta\bigl(\Delta u+2au+\langle A,K\rangle\bigr).
\end{equation}
Since \(|\Delta C|\le\sqrt n|\nabla^2C|\), the preceding bounds imply
\[
 |\Delta u|+2au+|\langle A,K\rangle|\le32nN^2.
\]
Together with \(a\ge n/2\), this gives
\[
 \varepsilon
 \le64N^2\beta+\frac{2R_0^2}{n}(\varepsilon+\beta)^2
 \le64N^2\beta
       +\frac{4R_0^2}{n}\varepsilon^2
       +\frac{4R_0^2}{n}\beta^2.
\]
The master hypotheses and
\eqref{quant:elementary-neighborhood-bounds} imply
\[
 \varepsilon\le\frac n{8R_0^2},\qquad
 \beta\le\frac n{16R_0^2}.
\]
Indeed, it is enough to insert
\(\varepsilon\le4\sqrt n\,\eta\),
\(\beta\le\eta^2\), \(R_0=L^{10}\), and \(\eta\le L^{-1000}\).
Absorbing the squared terms therefore yields
\[
 \varepsilon\le(128N^2+\tfrac12)\beta\le130N^2\beta.
\]
Consequently
\[
 d+\sup_M|\nabla A|
 \le132N^2L^{10}\beta\le L^{12}\beta,
\]
because \(132N^2\le L^2\).

If \(\beta=0\), the rank-one reduction makes the unused normal line
parallel, so \(B=\omega=0\).  Set \(C=\theta=0\) in the preceding
formulas.  The estimates for the normalized data are then automatic,
and the proof through the first line of
\eqref{quant:polynomial-rate-conclusion} applies with \(t=\varepsilon\).
Identity \eqref{quant:trace-identity} reduces to
\(\varepsilon a=|\nabla A|^2\), and the same absorption forces
\(\varepsilon=0\).  This also proves the stated conclusion in that case.
\end{proof}

\begin{lemma}[Uniform rigidity near a Clifford orbit]
\label{quant:near-clifford-rigidity}
Let \(n\ge3\), and put
\[
 N=(100n)^{100000},\qquad L=(100n)^{100001}=100nN.
\]
For \(1\le p\le n-1\), set \(m=n-p\) and
\[
 \lambda=-\sqrt{\frac mp},\qquad
 \mu=\sqrt{\frac pm},\qquad
 D=\mu-\lambda=\frac n{\sqrt{pm}},\qquad c=\lambda+\mu.
\]
Let \(\mathcal O_p\) be the orthogonal orbit of
\(\operatorname{diag}(\lambda I_p,\mu I_m)\).
Suppose that a closed connected minimal immersion into
\(\mathbb S^{n+2}(1)\) has constant \(Q=n+\varepsilon\), where
\(\varepsilon\ge0\).  Write \(a=\lambda_1\), \(b=\lambda_2\), and
\(\beta=\max_M b\).  After passing to a finite covering, choose a normal
eigenframe with shape operators \(A,B\).  Suppose that, for some
\(0<\eta\le L^{-1000}\),
\begin{equation}\label{quant:near-clifford-hypotheses}
 \sup_M\operatorname{dist}(A,\mathcal O_p)\le\eta,\qquad
 \sup_M|\nabla^rA|\le\eta\quad(1\le r\le5),\qquad
 \sqrt\beta\le\eta .
\end{equation}
Here distance and all tensor norms are induced by the Hilbert--Schmidt
inner product.  Then \(\varepsilon=\beta=0\).  In particular, the
immersion is a covering of a standard Clifford hypersurface in a fixed
great sphere.
\end{lemma}

\begin{proof}
The hypotheses imply \(a-b>n/2\), so the normal eigenlines are smooth.
Orienting the two real eigenline bundles requires a covering of degree
at most four and preserves every numerical hypothesis.  We work on
this covering.  Let \(P\) be the tangent spectral projection of \(A\)
onto the cluster near \(\lambda\), and put \(P^\perp=I-P\).  Its rank is
\(p\).  Define
\[
 A_*=\lambda P+\mu P^\perp,\qquad
 H=A-A_*,\qquad E=A^2-cA-I,\qquad d=\sup_M|E|.
\]
Since \(D\ge2\) and the two clusters lie within \(\eta\) of the roots,
\begin{equation}\label{quant:root-comparison-end}
 |H|\le |E|,\qquad |A|\le2\sqrt n,\qquad
 |\lambda|,\mu\ge n^{-1/2}.
\end{equation}
Indeed, for an eigenvalue \(t\) in the first cluster,
\(|t-\mu|\ge D-\eta\ge1\), and hence
\(|t-\lambda|\le|(t-\lambda)(t-\mu)|\); the other cluster is identical.

We first dispose of \(\beta=0\).  The rank-one reduction shows that
\(B=\omega=0\).  In the polynomial argument one may consequently set
\(C=\theta=0\); the estimate
\eqref{quant:polynomial-pre-rate} gives
\(|\nabla A|\le L^{10}\varepsilon\).  Since \(a=n+\varepsilon\)
is constant, the scalar Simons identity is
\[
 a\varepsilon=|\nabla A|^2\le L^{20}\varepsilon^2 .
\]
On the other hand, \eqref{quant:near-clifford-hypotheses} implies
\[
 \varepsilon=a-n\le(2\sqrt n+\eta)\eta\le4\sqrt n\,\eta.
\]
If \(\varepsilon>0\), these inequalities give
\(n\le4\sqrt n\,L^{20}\eta< n\), a contradiction.
Thus \(\varepsilon=0\); the scalar Simons identity and the rank-one
reduction identify the Clifford immersion.

Henceforth \(\beta>0\).  Put
\[
 C=\frac B{\sqrt\beta},\qquad
 \theta=\frac\omega{\sqrt\beta},\qquad
 u=|C|^2,\qquad v=|\theta|^2 .
\]
In particular, \(\max_M u=1\).  Lemma~\ref{quant:normalized-derivatives}
and Proposition~\ref{quant:polynomial-rate} yield
\begin{equation}\label{quant:terminal-inputs}
\begin{gathered}
 |C|\le1,\qquad
 |\nabla^r C|\le N\ (1\le r\le3),\qquad
 |\nabla^r\theta|\le N\ (0\le r\le2),\\
 \varepsilon\le130N^2\beta,\qquad
 d+\sup_M|\nabla A|\le L^{12}\beta .
\end{gathered}
\end{equation}
All estimates in the rest of the proof hold on this same closed
manifold and are independent of its volume, diameter, and covering
degree.

\medskip
\noindent\emph{The diagonal blocks of \(C\).}
For completeness, the derivative constants for \(P\) can be read
directly from its Riesz projection.  On the circle of radius \(D/3\)
about \(\lambda\), the resolvent has operator norm at most \(4/D\).
Consequently,
\[
 \|DP_A\|\le\frac{16}{3D}\le3,\qquad
 \|D^2P_A\|\le\frac{128}{3D^2}\le11 .
\]
The derivative norms here use the Hilbert--Schmidt norm for the input
and output matrices.  Thus
\begin{equation}\label{quant:projector-derivatives-end}
 |\nabla P|\le3L^{12}\beta,\qquad
 |\nabla^2P|\le3\eta+11\eta^2\le4\eta .
\end{equation}
The tangent Simons equation has the form
\[
 \Delta A=(-\varepsilon+2\beta u)A+\beta K,\qquad
 |K|\le16\sqrt n\,N^2 .
\]
Since \(DP_A(A)=0\), it follows that
\begin{equation}\label{quant:projector-laplacian-end}
 |\Delta P|
 \le48\sqrt n\,N^2\beta+11L^{24}\beta^2
 \le L^5\beta .
\end{equation}

For a symmetric endomorphism \(T\), define the orthogonal projection
\[
 \Pi_P(T)=PTP+P^\perp TP^\perp
 -\frac{\operatorname{tr}(PT)}pP
 -\frac{\operatorname{tr}(P^\perp T)}mP^\perp .
\]
Let
\[
 Y=\Pi_P(C),\qquad C_{\mathrm{diag}}=PCP+P^\perp CP^\perp .
\]
The product rule gives
\[
 \|D\Pi_P(U)\|\le6|U|,\qquad
 \|D^2\Pi_P(U,V)\|\le8|U||V|.
\]
These are operator norms on symmetric endomorphisms.  The exact
normalized tangent Simons equation for \(C\) is
\begin{equation}\label{quant:projected-c-equation}
 \Delta C=(n-\beta u)C-[A,[A,C]]
 -2\sum_i\theta_i\nabla_iA
 -(\operatorname{div}\theta)A+\beta vC .
\end{equation}
Since \(A\) commutes with \(P\), its scalar parts on the two blocks
make no contribution to
\(\Pi_P([A,[A,C]])\); therefore
\[
 |\Pi_P([A,[A,C]])|\le4|H|^2|C|\le4d^2 .
\]
Moreover, \(|\Pi_P(A)|\le d\), and
\eqref{quant:projector-derivatives-end}--\eqref{quant:projector-laplacian-end}
give
\[
 |\nabla\Pi_P|\le18L^{12}\beta,\qquad
 |\Delta\Pi_P|\le6L^5\beta+72L^{24}\beta^2\le L^7\beta .
\]
Applying \(\Pi_P\) to \eqref{quant:projected-c-equation}, including
the two derivatives of the projection, now gives
\begin{equation}\label{quant:diagonal-coercive-equation}
 |\Delta Y-nY|\le L^{15}\beta .
\end{equation}
For example, the derivative terms are bounded by
\(36NL^{12}\beta+L^7\beta\), the normal divergence term by
\(\sqrt n\,Nd\), and the remaining terms by
\(\beta+4d^2+2NL^{12}\beta+N^2\beta\).
Their sum is at most \(L^{15}\beta\).
At a maximum of \(|Y|^2\),
\[
 0\ge|\nabla Y|^2+n|Y|^2-L^{15}\beta|Y|,
\]
so \(|Y|\le L^{16}\beta\).

The traces discarded by \(\Pi_P\) are small as well.  Writing
\(\alpha=\operatorname{tr}(PC)\), the identities
\(\operatorname{tr}C=\langle A,C\rangle=0\) give
\[
 \operatorname{tr}(P^\perp C)=-\alpha,\qquad
 \alpha=\frac{\langle H,C\rangle}{D}.
\]
The norm of the scalar block part is exactly
\[
 \left|\frac\alpha pP-\frac\alpha mP^\perp\right|
 =|\alpha|\sqrt{\frac1p+\frac1m}
 =\frac{|\langle H,C\rangle|}{\sqrt n}\le\frac d{\sqrt n}.
\]
We conclude that
\begin{equation}\label{quant:block-c-small}
 |C_{\mathrm{diag}}|\le L^{17}\beta .
\end{equation}

We shall use the following elementary tensor interpolation estimate
on a complete manifold:
\begin{equation}\label{quant:geodesic-interpolation}
 \sup|\nabla T|
 \le2\sqrt{n\,\sup|T|\,\sup|\nabla^2T|}.
\end{equation}
To prove it, parallel transport \(T\) along a complete unit-speed
geodesic.  Taylor's estimate along this geodesic,
\(|f'(0)|\le2\sqrt{\sup|f|\,\sup|f''|}\) controls every directional
derivative; summing over an orthonormal frame gives the factor
\(\sqrt n\).
Differentiating the definition of \(C_{\mathrm{diag}}\) twice gives
\[
 |\nabla^2C_{\mathrm{diag}}|
 \le N+4|\nabla P|N+2|\nabla^2P|+4|\nabla P|^2\le4N.
\]
Consequently,
\begin{equation}\label{quant:block-c-gradient}
\begin{gathered}
 |\nabla C_{\mathrm{diag}}|\le L^{10}\sqrt\beta,\\
 |P(\nabla C)P+P^\perp(\nabla C)P^\perp|
 \le|\nabla C_{\mathrm{diag}}|+2|\nabla P||C|
 \le L^{11}\sqrt\beta .
\end{gathered}
\end{equation}

\medskip
\noindent\emph{Both multiplicities are at least two.}
Suppose first that \(p,m\ge2\).  Restrict all three tangent indices
of the normalized Codazzi equation to the \(P\) block:
\[
 \lambda(\theta_i\delta_{kr}-\theta_k\delta_{ir})
 =-(\nabla_iC_{kr}-\nabla_kC_{ir})
   -\theta_iH_{kr}+\theta_kH_{ir}.
\]
The elementary identity
\[
 \sum_{i,k,r\in P}
   (\theta_i\delta_{kr}-\theta_k\delta_{ir})^2
 =2(p-1)|P\theta|^2
\]
and \eqref{quant:block-c-gradient} imply
\[
 |\lambda|\sqrt{2(p-1)}\,|P\theta|
 \le2L^{11}\sqrt\beta+2NL^{12}\beta
 \le3L^{11}\sqrt\beta .
\]
Since \(|\lambda|\ge n^{-1/2}\), we have
\(|P\theta|\le L^{12}\sqrt\beta\).  The same argument in the
\(P^\perp\) block, using \(m\ge2\), gives
\begin{equation}\label{quant:theta-multiple-blocks}
 |\theta|\le L^{13}\sqrt\beta,\qquad
 |\nabla\theta|
 \le2\sqrt{nNL^{13}\sqrt\beta}
 \le L^{10}\beta^{1/4}.
\end{equation}
In matrix components the Ricci equation says
\([A,C]_{ij}=-(\nabla_i\theta_j-\nabla_j\theta_i)\).  Hence, with
the full matrix norm,
\[
 |[A,C]|\le2|\nabla\theta|.
\]
The two tangent spectral clusters of \(A\) are separated by
\(D-2\eta\ge1\), so
\[
 |PCP^\perp+P^\perp CP|
 \le|[A,C]|\le2L^{10}\beta^{1/4}.
\]
At a point where \(|C|=1\), this contradicts
\[
 1\le L^{17}\beta+2L^{10}\beta^{1/4}
 \le L^{-1983}+2L^{-490}<1.
\]

\medskip
\noindent\emph{One multiplicity is one.}
Changing the sign of the first normal if necessary, assume
\(p=1\), \(m=n-1\).  Thus
\[
 \lambda=-\sqrt m,\qquad \mu=\frac1{\sqrt m},\qquad
 D=\frac n{\sqrt m}.
\]
Orient the rank-one tangent bundle of \(P\) on an additional double
cover and let \(e_0\) be its unit section.  Define
\[
 z=P^\perp Ce_0,\qquad \tau=\theta(e_0),\qquad
 \theta_\perp=P^\perp\theta .
\]
We have the exact decomposition
\begin{equation}\label{quant:cylinder-decomposition}
 C=e_0\otimes z+z\otimes e_0+C_{\mathrm{diag}},\qquad
 u=2|z|^2+|C_{\mathrm{diag}}|^2,\qquad
 v=\tau^2+|\theta_\perp|^2 .
\end{equation}
The preceding pure-block argument applies to the \(m\)-dimensional
block and gives
\[
 |\theta_\perp|\le L^{12}\sqrt\beta.
\]
Since \(|\nabla^2\theta_\perp|\le4N\), interpolation gives
\[
 |\nabla\theta_\perp|\le L^{10}\beta^{1/4}.
\]
Here and below the covariant derivatives are taken on the original
manifold, rather than on a product model.

The mixed Codazzi equations and the Ricci equation now give
\begin{equation}\label{quant:cylinder-first-order}
 \begin{split}
 |\nabla z-\mu\tau P^\perp|&\le L^{15}\sqrt\beta,\\
 |P^\perp\nabla\tau+Dz|&\le L^{15}\beta^{1/4},\qquad
 |\nabla_{e_0}\tau|\le2N .
 \end{split}
\end{equation}
Here are the component details.  In an adapted orthonormal frame,
with \(a,b\) in the \(P^\perp\) block, Codazzi gives
\[
 \nabla_aC_{0b}=\nabla_0C_{ab}+\tau A_{ab},
 \qquad
 \nabla_0C_{a0}=\nabla_aC_{00}+\theta_a A_{00}.
\]
Passing from these components to the derivative of
\(z=P^\perp Ce_0\) adds at most \(2|\nabla P||C|\).
Thus the first error in \eqref{quant:cylinder-first-order} is at most
\[
 2L^{11}\sqrt\beta+Nd+
 2\sqrt n\,|\theta_\perp|+2|\nabla P|
 \le L^{15}\sqrt\beta .
\]
Also,
\[
 d\theta(e_0,e_a)=Dz_a+[H,C]_{a0}.
\]
Differentiating \(\tau=\theta(e_0)\) and separating
\(\theta_\perp\) shows that the second error is at most
\[
 |\nabla\theta_\perp|+2N|\nabla P|+2d
 \le L^{15}\beta^{1/4}.
\]
The final bound follows from
\(|\nabla\theta|\le N\) and \(|\nabla e_0|\le|\nabla P|\).

Put \(\kappa=\varepsilon/\beta\).  Constancy of \(Q=a+2b\)
gives the following exact weighted identity:
\begin{equation}\label{quant:q-weighted-identity}
 \begin{split}
 Q\kappa={}&\frac{|T_1|^2}{\beta}
 +2|\nabla C+\theta\otimes A|^2
 -(a-b)v+4au+2\beta u^2-3|[A,C]|^2 .
 \end{split}
\end{equation}
Indeed, adding the \(a\) eigenvalue equation to twice the \(b\)
eigenvalue equation gives
\[
 |T_1|^2+2|T_2|^2-(a-b)|\omega|^2
 =Q\varepsilon+2(a-b)b-3(2ab-|[A,B]|^2);
\]
divide by \(\beta\) and use
\(T_2/\sqrt\beta=\nabla C+\theta\otimes A\).

We record explicitly the leading terms and the error in
\eqref{quant:q-weighted-identity}.  By
\eqref{quant:terminal-inputs},
\[
 |T_1|=|\nabla A-\beta\theta\otimes C|
 \le2L^{12}\beta,\qquad
 \kappa\le130N^2<L^2,\qquad
 \varepsilon\kappa\le L^4\beta .
\]
Let \(\mathcal T\) be the symmetric three-tensor whose only nonzero
adapted components are
\[
 \mathcal T_{000}=\lambda\tau,\qquad
 \mathcal T_{0ab}=\mathcal T_{a0b}=\mathcal T_{ab0}
 =\mu\tau\delta_{ab}.
\]
Then \(|\mathcal T|^2=(\lambda^2+3m\mu^2)\tau^2
=(n+2)\tau^2\).
Equations \eqref{quant:cylinder-decomposition}--%
\eqref{quant:cylinder-first-order} give
\[
 |\nabla C+\theta\otimes A-\mathcal T|
 \le L^{17}\sqrt\beta,\qquad
 |\nabla C+\theta\otimes A|+|\mathcal T|\le2L.
\]
Furthermore,
\[
 \left||[A,C]|^2-\frac{n^2}{m}u\right|\le L^{15}\beta,
 \qquad |v-\tau^2|\le L^{24}\beta,
 \qquad |a-n|+|a-b-n|\le L^3\beta .
\]
For the commutator estimate, use
\([A_*,C_{\mathrm{diag}}]=0\),
\(|[H,C]|\le2d\), and
\(|[A_*,C-C_{\mathrm{diag}}]|^2
=D^2(u-|C_{\mathrm{diag}}|^2)\).
Substituting these bounds into \eqref{quant:q-weighted-identity},
and replacing \(Q\kappa\) by \(n\kappa\) using the displayed
bound on \(\varepsilon\kappa\), yields
\begin{equation}\label{quant:q-leading}
 \left|\kappa-\frac{n-4}{m}u-\frac{n+4}{n}v\right|
 \le L^{30}\sqrt\beta .
\end{equation}
In particular, this replacement does not discard an uncontrolled
factor \(\varepsilon/\beta\).

Choose \(x_0\) with \(u(x_0)=1\).  At this point
\(|z(x_0)|^2=(1-|C_{\mathrm{diag}}(x_0)|^2)/2\ge1/4\).
Write
\[
 E_z=\nabla z-\mu\tau P^\perp .
\]
Differentiating \(u\) in the unit direction
\(X=z(x_0)/|z(x_0)|\), and using \(\nabla u(x_0)=0\), gives
\[
 0=4\mu\tau(x_0)|z(x_0)|
   +4\langle E_z(X),z(x_0)\rangle
   +2\langle\nabla_XC_{\mathrm{diag}},C_{\mathrm{diag}}\rangle .
\]
Consequently,
\begin{equation}\label{quant:cylinder-start}
 |\tau(x_0)|\le L^{17}\sqrt\beta,\qquad
 v(x_0)\le L^{35}\beta,\qquad
 \left|\kappa-\frac{n-4}{m}\right|\le L^{37}\sqrt\beta .
\end{equation}

Let \(\gamma\) be the unit-speed geodesic with
\(\gamma(0)=x_0\), \(\gamma'(0)=X\), and write \(T=\gamma'\).
Along \(\gamma\), set
\[
 a_0(s)=\langle T,e_0\rangle,\qquad
 \ell(s)=\langle z,T\rangle,\qquad
 y(s)=\frac{\tau}{\sqrt n},\qquad
 \varpi=\sqrt{\frac nm}.
\]
The subscript on \(a_0\) distinguishes this scalar angle from the
normal eigenvalue \(a\).  Since
\(\nabla_Te_0=(\nabla_TP)e_0\) and \(a_0(0)=0\),
\[
 |a_0(s)|\le3L^{12}\beta s\le6L^{12}\beta
 \qquad(0\le s\le2).
\]
The first-order equations \eqref{quant:cylinder-first-order}
imply
\begin{equation}\label{quant:cylinder-oscillator}
 \ell'=\varpi y+e_1,\qquad
 y'=-\varpi\ell+e_2,\qquad
 |e_1|+|e_2|\le L^{28}\beta^{1/4},
\end{equation}
where, explicitly,
\[
 \begin{split}
 e_1&=\langle E_z(T),T\rangle-\varpi y\,a_0^2,\\
 e_2&=\frac1{\sqrt n}
 \left(\langle P^\perp\nabla\tau+Dz,P^\perp T\rangle
       +a_0\nabla_{e_0}\tau\right).
 \end{split}
\]
Put \(s_*=\pi/(2\varpi)<2\).  The unperturbed system is an
orthogonal rotation, so variation of constants gives
\[
 |(\ell(s_*),y(s_*))-(y(0),-\ell(0))|
 \le2L^{28}\beta^{1/4}.
\]
Using \eqref{quant:cylinder-start},
\begin{equation}\label{quant:cylinder-rotation-end}
 |\ell(s_*)|\le L^{30}\beta^{1/4},\qquad
 |y(s_*)+\ell(0)|\le L^{29}\beta^{1/4}.
\end{equation}

It remains to control the whole vector \(z\), rather than just its
component along \(T\).  Let \(w=z-\ell T\).  Then \(w(0)=0\), and
\[
 \nabla_Tw
 =\mu\tau(a_0^2T-a_0e_0)
   +E_z(T)-\langle E_z(T),T\rangle T.
\]
Since \(|a_0^2T-a_0e_0|\le|a_0|\), this gives
\[
 |\nabla_Tw|\le6NL^{12}\beta+L^{15}\sqrt\beta
 \le L^{16}\sqrt\beta,\qquad
 |w(s_*)|\le L^{17}\sqrt\beta .
\]
Together with \eqref{quant:cylinder-decomposition} and
\eqref{quant:cylinder-rotation-end}, it follows that
\begin{equation}\label{quant:cylinder-end}
 u(\gamma(s_*))\le L^{64}\beta^{1/2},\qquad
 \left|v(\gamma(s_*))-\frac n2\right|
 \le L^{32}\beta^{1/4}.
\end{equation}
For the second estimate, use
\(v=ny^2+|\theta_\perp|^2\) and
\(\ell(0)^2=(1-|C_{\mathrm{diag}}(x_0)|^2)/2\).
Applying \eqref{quant:q-leading} at the endpoint yields
\[
 \left|\kappa-\frac{n+4}{2}\right|
 \le L^{66}\beta^{1/4}.
\]
The two leading values of \(\kappa\) are separated by
\[
 \frac{n+4}{2}-\frac{n-4}{n-1}
 =\frac{n^2+n+4}{2(n-1)}\ge4
 \qquad(n\ge3).
\]
But their combined error is at most
\[
 L^{37}\sqrt\beta+L^{66}\beta^{1/4}
 \le L^{70}\beta^{1/4}
 \le L^{-430}<1.
\]
This final contradiction excludes \(\beta>0\) and proves the lemma.
\end{proof}

\subsection{The numerical gap constant}
\label{quant:final-constant-section}

We now combine the estimates and compare their thresholds with the
constant in Theorem~\ref{thm:q2-onlyQ-low-dimension-gap}.

\begin{proof}[Proof of Theorem~\ref{thm:q2-onlyQ-low-dimension-gap}]
For $n\ge3$, retain the constants $\Theta_n$ and $B_n$ from
\eqref{quant:entry-constants}, and put
\begin{equation}\label{quant:explicit-threshold}
 \begin{split}
 x&=100n,\qquad K=x^{10000},\qquad
 \eta_n=x^{-10^{14}},\\
 \widehat\gamma_n
 &=\min\left\{\Theta_n,
       B_n^{-1}x^{-(2\cdot10^{14}+20002)}\right\}.
 \end{split}
\end{equation}
Suppose that $Q=n+\varepsilon$ is constant and
$0<\varepsilon\le\widehat\gamma_n$.
Proposition~\ref{quant:global-entry} gives
$\beta\le B_n\varepsilon$. Set $t=B_n\varepsilon>0$.
Then $\varepsilon\le t$, $\beta\le t$, and
\[
 t\le x^{-(2\cdot10^{14}+20002)}\le x^{-30000}.
\]
Lemma~\ref{quant:jets-clifford-entry} supplies a fixed Clifford orbit
$\mathcal O_p$ on a connected finite covering, and gives
\[
 \sqrt\beta+
 \sup_M\operatorname{dist}(A,\mathcal O_p)
 +\sum_{r=1}^6\sup_M|\nabla^rA|
 \le K\sqrt t
 \le x^{-10^{14}-1}<\eta_n.
\]
With $L=x^{100001}$ as in the quantitative rigidity lemma,
\[
 \eta_n=x^{-10^{14}}
       \le x^{-100001000}=L^{-1000}.
\]
Thus every hypothesis of
Lemma~\ref{quant:near-clifford-rigidity} is satisfied.
That lemma gives $\varepsilon=\beta=0$, a contradiction.
We have proved the exclusion of all
$0<\varepsilon\le\widehat\gamma_n$.

For clarity, the comparison with the simpler formula
\eqref{eq:explicit-main-gap} is entirely numerical. Since
$\log(3\cdot10^6)<15$, $\log n\le n$, and $n\ge3$, the definition
of $\Theta_n$ gives
\begin{equation}\label{quant:threshold-log-comparison}
 \begin{split}
 -\log\Theta_n
 &=\max\left\{
 (n+2)\log(3\cdot10^6)+(11n+20)\log2
                  +(3n+4)\log n,\right.\\[-2pt]
 &\hspace{39mm}\left.\log(10000n^2)\right\}
 <50n^2.
 \end{split}
\end{equation}
Also $\log(100n)\le2n\le n^2$, so
\[
 \begin{split}
 \log B_n+(2\cdot10^{14}+20002)\log x
 &=5000n^{3/2}+(2\cdot10^{14}+20002)\log(100n)\\
 &<10^{15}n^2.
 \end{split}
\]
Consequently
\[
 -\log\widehat\gamma_n<10^{15}n^2<10^{16}n^2,
 \qquad
 \exp(-10^{16}n^2)<\widehat\gamma_n.
\]
The exclusion just proved therefore applies whenever
$0<Q-n\le\exp(-10^{16}n^2)$, and implies the stated gap.
\end{proof}

\begin{remark}
The expression $\widehat\gamma_n$ in
\eqref{quant:explicit-threshold} is also an explicit admissible
codimension-two gap. Formula~\eqref{eq:explicit-main-gap} is chosen
for its simpler form. The existence proof for $3\le q\le n$ uses
complete pointed compactness and interior elliptic estimates with
constants depending on $n$ and $q$; the present paper does not assign
an explicit numerical value to $\gamma_{n,q}$ in that range.
\end{remark}

\raggedbottom
\par\bigskip
\noindent\textbf{Acknowledgements.}

Fagui Li is partially supported by NSFC (No. 12271040 and 12501061), the Guangdong Provincial Association for Science and Technology Youth Talent Support Program (No. SKXRC2026424), and the Research Start-up Funding of Beijing Institute of Technology (No. 5640011253301). The authors are grateful to Prof. Jianquan Ge and Dr. Kai Jia
for their encouragement, support, and valuable discussions.

\raggedbottom
\par\bigskip
\noindent\textbf{AI Disclosure.}

The authors used OpenAI GPT-6 Astra  to assist with developing and checking mathematical arguments and with manuscript preparation. The authors are responsible for verifying all arguments and references and for the content and final presentation.


\begin{thebibliography}{99}
\bibitem{ADN}
S.~Agmon, A.~Douglis, and L.~Nirenberg, \emph{Estimates near the boundary for solutions of elliptic partial differential equations satisfying general boundary conditions. II}, Comm. Pure Appl. Math. \textbf{17} (1964), 35--92. \href{https://doi.org/10.1002/cpa.3160170104}{DOI}.

\bibitem{CatinoRoncoroni24}
G.~Catino and A.~Roncoroni, \emph{A closure result for globally hyperbolic spacetimes}, Proc. Amer. Math. Soc. \textbf{152} (2024), 5339--5354. \href{https://doi.org/10.1090/proc/16969}{DOI}.

\bibitem{Chang93}
S. P. Chang, \emph{On minimal hypersurfaces with constant scalar curvatures in $\Sph^4$}, J. Differential Geom. \textbf{37} (1993), 523--534. \href{https://doi.org/10.4310/jdg/1214453898}{DOI}.

\bibitem{ChenXu}
Q. Chen and S. L. Xu, \emph{Rigidity of compact minimal submanifolds in a unit sphere}, Geom. Dedicata \textbf{45} (1993), 83--88. \href{https://doi.org/10.1007/BF01667404}{DOI}.

\bibitem{CDK}
S. S. Chern, M. do Carmo, and S. Kobayashi, \emph{Minimal submanifolds of a sphere with second fundamental form of constant length}, in \emph{Functional Analysis and Related Fields}, Springer, 1970, 59--75. \href{https://doi.org/10.1007/978-3-642-48272-4_2}{DOI}.

\bibitem{CucinottaMondino26}
A.~Cucinotta and A.~Mondino, \emph{On manifolds with almost non-negative
Ricci curvature and integrally-positive $k^{\mathrm{th}}$-scalar curvature},
Math. Ann. \textbf{394} (2026), Paper No.~49, 50 pp.
\href{https://doi.org/10.1007/s00208-026-03406-8}{DOI}.

\bibitem{DengGuWei17}
Q. T. Deng, H. L. Gu, and Q. Y. Wei, \emph{Closed Willmore minimal hypersurfaces with constant scalar curvature in $\Sph^5(1)$ are isoparametric}, Adv. Math. \textbf{314} (2017), 278--305. \href{https://doi.org/10.1016/j.aim.2017.05.002}{DOI}.

\bibitem{DGLY}
W. R. Ding, J. Q. Ge, F. G. Li, and X. Z. Yang, \emph{Lu's conjecture for minimal surfaces}, preprint, \href{https://arxiv.org/abs/2601.07194}{arXiv:2601.07194}, 2026.

\bibitem{DingXin11}
Q. Ding and Y. L. Xin, \emph{On Chern's problem for rigidity of minimal hypersurfaces in the spheres}, Adv. Math. \textbf{227} (2011), 131--145. \href{https://doi.org/10.1016/j.aim.2011.01.018}{DOI}.

% Audit note (2026-09-26): Full text unavailable in this audit: the precise density statement and equation-(9) identification in the Introduction are retained from the supplied manuscript and require source-level confirmation.
\bibitem{FT26}
B. Firester and R. Tsiamis, \emph{On Chern's conjecture for minimal submanifolds of the sphere}, preprint, \href{https://arxiv.org/abs/2608.18074}{arXiv:2608.18074}, 2026.

\bibitem{GLZLuSurfaces}
J. Q. Ge, F. G. Li, and Y. H. Zhang, \emph{Lu's conjecture for minimal surfaces in codimension two}, preprint, \href{https://arxiv.org/abs/2607.21336v2}{arXiv:2607.21336v2}, 2026.

\bibitem{GLZFlatGap}
J. Q. Ge, F. G. Li, and Y. H. Zhang, \emph{On Chern's conjecture for minimal submanifolds with flat normal bundle in spheres}, preprint, \href{https://arxiv.org/abs/2607.10733}{arXiv:2607.10733}, 2026.

\bibitem{GeLiuLuoYan26}
J. Q. Ge, T. Liu, K. Y. Luo, and W. J. Yan, \emph{Rigidity of closed minimal hypersurfaces in $\Sph^5$}, preprint, \href{https://arxiv.org/abs/2606.29246}{arXiv:2606.29246}, 2026.

\bibitem{GTYZ26}
J. Q. Ge, H. X. Tan, W. J. Yan, and Y. H. Zhang, \emph{The second gap and rigidity in Chern's conjecture with constant cubic trace}, preprint, \href{https://arxiv.org/abs/2609.27661}{arXiv:2609.27661}, 2026.

\bibitem{Guan26}
S. L. Guan, \emph{The second gap for minimal hypersurfaces in the unit sphere with constant cubic trace}, preprint, \href{https://arxiv.org/abs/2609.26852}{arXiv:2609.26852}, 2026.

\bibitem{HeXuZhao26}
C. C. He, H. W. Xu, and E. T. Zhao, \emph{Classification of closed minimal hypersurfaces with constant scalar curvature in $\Sph^5$}, preprint, \href{https://arxiv.org/abs/2603.01181}{arXiv:2603.01181}, 2026.

% Audit note (2026-09-26): Metadata and the general pinching topic were verified from the arXiv abstract; the exact scope of the claimed threshold improvement needs a full-text check.
\bibitem{Lei26}
L. Lei, \emph{Rigidity of minimal submanifolds in spheres of higher codimension}, preprint, \href{https://arxiv.org/abs/2609.04393}{arXiv:2609.04393}, 2026.

\bibitem{LeiXuXu17}
L. Lei, H. W. Xu, and Z. Y. Xu, \emph{On Chern's conjecture for minimal hypersurfaces in spheres}, preprint, \href{https://arxiv.org/abs/1712.01175}{arXiv:1712.01175}, 2017.

\bibitem{LiLi}
A. M. Li and J. M. Li, \emph{An intrinsic rigidity theorem for minimal submanifolds in a sphere}, Arch. Math. (Basel) \textbf{58} (1992), 582--594. \href{https://doi.org/10.1007/BF01193528}{DOI}.

\bibitem{LiZhaoGap}
F. G. Li and Y. H. Zhao, \emph{A curvature gap for minimal submanifolds in spheres}, preprint, \href{https://arxiv.org/abs/2608.16095v3}{arXiv:2608.16095v3}, 2026.

\bibitem{LiZhangSpiral26}
H. Z. Li and Y. S. Zhang, \emph{The geometry and dynamics of spiral minimal products}, preprint, \href{https://arxiv.org/abs/2608.02370v2}{arXiv:2608.02370v2}, 2026.

\bibitem{LZ26}
F. G. Li and Y. H. Zhao, \emph{Flat minimal tori and Lu's second-gap conjecture}, preprint, \href{https://arxiv.org/abs/2606.30432v2}{arXiv:2606.30432v2}, 2026.

\bibitem{LiuYang26}
H. M. Liu and L. Yang, \emph{An optimal pinching theorem on compact minimal submanifolds in the Euclidean spheres via eigenvalues of fundamental matrices}, preprint, \href{https://arxiv.org/abs/2609.26309}{arXiv:2609.26309}, 2026.

\bibitem{Lu11}
Z. Q. Lu, \emph{Normal scalar curvature conjecture and its applications}, J. Funct. Anal. \textbf{261} (2011), 1284--1308. \href{https://doi.org/10.1016/j.jfa.2011.05.002}{DOI}.

\bibitem{PengTerng}
C.-K.~Peng and C.-L.~Terng, \emph{Minimal hypersurfaces of spheres with constant scalar curvature}, in \emph{Seminar on Minimal Submanifolds}, E. Bombieri (ed.), Ann. of Math. Stud. \textbf{103}, Princeton University Press, Princeton, NJ, 1983, 177--198.

\bibitem{PengTerng83}
C.-K.~Peng and C.-L.~Terng, \emph{The scalar curvature of minimal hypersurfaces in spheres}, Math. Ann. \textbf{266} (1983), 105--113. \href{https://doi.org/10.1007/BF01458707}{DOI}.

\bibitem{Petersen}
P.~Petersen, \emph{Riemannian Geometry}, third edition, Graduate Texts in Mathematics \textbf{171}, Springer, 2016. \href{https://doi.org/10.1007/978-3-319-26654-1}{DOI}.

\bibitem{Simons}
J. Simons, \emph{Minimal varieties in Riemannian manifolds}, Ann. of Math. (2) \textbf{88} (1968), 62--105. \href{https://doi.org/10.2307/1970556}{DOI}.

\bibitem{suh_yang_2007}
Y. J. Suh and H. Y. Yang, \emph{The scalar curvature of minimal hypersurfaces in a unit sphere}, Commun. Contemp. Math. \textbf{9} (2007), 183--200.

\bibitem{TTXY26}
H. X. Tan, Z. Z. Tang, Y. Q. Xie, and W. J. Yan, \emph{Chern's conjecture with constant cubic trace}, preprint, \href{https://arxiv.org/abs/2609.03711}{arXiv:2609.03711}, 2026.

\bibitem{TangWeiYan20}
Z. Z. Tang, D. Y. Wei, and W. J. Yan, \emph{A sufficient condition for a hypersurface to be isoparametric}, Tohoku Math. J. (2) \textbf{72} (2020), 493--505. \href{https://doi.org/10.2748/tmj.20190611}{DOI}.

\bibitem{TangYan23}
Z. Z. Tang and W. J. Yan, \emph{On the Chern conjecture for isoparametric hypersurfaces}, Sci. China Math. \textbf{66} (2023), 143--162. \href{https://doi.org/10.1007/s11425-022-1967-4}{DOI}.

\bibitem{XuZhao26}
H. W. Xu and E. T. Zhao, \emph{The pinching constant for closed minimal submanifolds of high codimension in the sphere}, preprint, \href{https://arxiv.org/abs/2609.04631}{arXiv:2609.04631}, 2026.

\bibitem{yang_cheng_1998}
H. C. Yang and Q. M. Cheng, \emph{Chern's conjecture on minimal hypersurfaces}, Math. Z. \textbf{227} (1998), 377--390. \href{https://doi.org/10.1007/PL00004382}{DOI}.
\end{thebibliography}
\end{document}